\documentclass[12 pt]{amsart}

\usepackage{hyperref}
\usepackage{etex}
\usepackage[shortlabels]{enumitem}
\usepackage{amsmath}
\usepackage{amsxtra}
\usepackage{amscd}
\usepackage{amsthm}
\usepackage{amsfonts}
\usepackage{amssymb}
\usepackage{eucal}
\usepackage{stmaryrd}
\usepackage[all]{xy}
\usepackage{graphicx}
\usepackage{tikz-cd}
\usepackage{mathrsfs}
\usepackage{subfiles}
\usepackage{mathpazo}
\usepackage[colorinlistoftodos, textsize=tiny]{todonotes}
\usepackage{morefloats}
\usepackage{pdfpages}
\usepackage{thm-restate}
\usepackage[utf8]{inputenc}
\usepackage{epigraph}
\usepackage{csquotes}
\usepackage{subfiles}
\usepackage[margin=1.2in]{geometry}

\graphicspath{ {images/} }

\RequirePackage{color}
\definecolor{myred}{rgb}{0.75,0,0}
\definecolor{mygreen}{rgb}{0,0.5,0}
\definecolor{myblue}{rgb}{0,0,0.65}

\usepackage{color}

\newcommand{\rboxed}[1]{#1}

\usepackage{hyperref}
\hypersetup{citecolor=blue}
\usepackage{tikz}
\usetikzlibrary{matrix,arrows,decorations.pathmorphing}

\theoremstyle{plain}
\newtheorem{theorem}[subsection]{Theorem}
\newtheorem{proposition}[theorem]{Proposition}
\newtheorem{lemma}[theorem]{Lemma}
\newtheorem{corollary}[theorem]{Corollary}
\theoremstyle{definition}
\newtheorem{definition}[theorem]{Definition}
\newtheorem{remark}[theorem]{Remark}
\newtheorem{example}[theorem]{Example}

\newtheorem{question}[theorem]{Question}
\newtheorem{conjecture}[theorem]{Conjecture} 

\newtheorem{warn}[theorem]{Warning}

\newtheorem{notation}[theorem]{Notation}
\theoremstyle{remark}
\numberwithin{equation}{section}
  
\newcommand\nc{\newcommand}
\nc\renc{\renewcommand}

\newcommand\ba{\mathbb A}

\newcommand\bc{\mathbb C}

\newcommand\bg{\mathbb G}

\newcommand\bl{\mathbb L}

\newcommand\bp{\mathbb P}
\newcommand\bq{\mathbb Q}

\newcommand\bz{\mathbb Z}

\newcommand\fm{{\mathfrak m}}

\newcommand\sca{\mathscr A}

\newcommand\sce{\mathscr E}
\newcommand\scf{\mathscr F}
\newcommand\scg{\mathscr G}
\newcommand\sch{\mathscr H}
\newcommand\sci{\mathscr I}

\newcommand\scl{\mathscr L}
\newcommand\scm{\mathscr M}
\newcommand\scn{\mathscr N}
\newcommand\sco{\mathscr O}

\newcommand\scu{\mathscr U}
\newcommand\scv{\mathscr V}
\newcommand\scw{\mathscr W}
\newcommand\scx{\mathscr X}
\newcommand\scy{\mathscr Y}
\newcommand\scz{\mathscr Z}

\newcommand \ra{\rightarrow}

\DeclareMathOperator\spec{\operatorname{Spec}}
\DeclareMathOperator\proj{\operatorname{Proj}}

\newcommand*{\shom}{\mathscr{H}\kern -.5pt om}
\newcommand*{\stor}{\mathscr{T}\kern -.5pt or}
\newcommand*{\sext}{\mathscr{E}\kern -.5pt xt}

\makeatletter
\providecommand\@dotsep{5}
\renewcommand{\listoftodos}[1][\@todonotes@todolistname]{%
  \@starttoc{tdo}{#1}}
\makeatother

\newcommand{\customlabel}[2]{\protected@write \@auxout {}{\string \newlabel {#1}{{#2}{\thepage}{#2}{#1}{}} }\hypertarget{#1}{#2}}

\newcommand\vect[2]{\mathrm{Vect}^{#1}_{#2}} 
\newcommand\vecthn[2]{\mathrm{Vect}^{#1}_{#2}} 
\newcommand\bighur[2]{\overline{\mathrm{Hur}}_{#1,#2}} 
\newcommand\bighurg[3]{\overline{\mathrm{Hur}}_{#1,#2, #3}} 
\newcommand\bigrhur[4]{\overline{\mathrm{Hur}}_{#1,#2, #3}^{#4}} 
\newcommand\hur[3]{\mathrm{Hur}_{#1,#2,#3}} 
\newcommand\grouphur[2]{\mathrm{Hur}_{#1,#2}} 
\newcommand\rhur[4]{\mathrm{Hur}_{#1,#2,#3}^{#4}} 
\newcommand\ce{\mathscr M} 
\newcommand\rce[1]{\mathscr M^{#1}} 
\newcommand\ramlocus[2]{\mathrm{U}_{#1}^{#2}} 
\newcommand\coverstack[1]{\mathrm{Covers}_{d}} 
\newcommand\auttorsor[1]{\mathrm{T}_{d}} 
\newcommand\gorensteinsections[1]{\mathrm{U}_{d}} 

\newcommand\systemdim[2]{n_{#1,#2}} 
\newcommand\codimbound[2]{r_{#1,#2}} 

\newcommand\reshilb[2]{\scz_{#1,#2}} 
\newcommand\ressections[2]{\scy_{#1,#2}} 
\newcommand\gerbe[2]{\scx_{#1,#2}} 

\DeclareMathOperator\res{Res}
\DeclareMathOperator\symm{Symm}

\newcommand\grSpaces{K_0(\mathrm{Spaces}_k)} 
\newcommand\grcSpaces{\widehat{\widetilde{K_0}}(\mathrm{Spaces}_k)} 
\newcommand\grStacks{K_0(\mathrm{Stacks}_k)} 
\newcommand\grcStacks{\widehat{\widetilde{K_0}}(\mathrm{Stacks}_k)} 

\DeclareMathOperator\id{id}

\renewcommand\hom{\mathrm{Hom}}
\DeclareMathOperator\coker{coker}

\DeclareMathOperator\hilb{Hilb}
\DeclareMathOperator\rk{rk}

\DeclareMathOperator\conf{Conf}
\DeclareMathOperator\codim{codim}

\DeclareMathOperator\im{im}
\DeclareMathOperator\End{End}
\DeclareMathOperator\sym{Sym}
\DeclareMathOperator\pgl{PGL}

\DeclareMathOperator\chr{\operatorname{char}}

\DeclareMathOperator\aut{Aut}
\DeclareMathOperator\gl{GL}

\renewcommand\sl{\mathrm{SL}}

\newcommand\sm{sm}

\DeclareFontFamily{U}{wncy}{}
\DeclareFontShape{U}{wncy}{m}{n}{<->wncyr10}{}
\DeclareSymbolFont{mcy}{U}{wncy}{m}{n}
\DeclareMathSymbol{\Sha}{\mathord}{mcy}{"58}

\usepackage[english]{babel}

\newcommand{\Q}{\mathbb Q}
\newcommand{\F}{\mathbb F}
\newcommand{\Z}{\mathbb Z}
\newcommand{\Spec}{\operatorname{Spec}}

\renewcommand{\L}{\mathbb L}

\newcommand{\on}[1]{\operatorname{#1}}

\makeatletter
\let\@wraptoccontribs\wraptoccontribs
\makeatother

\makeatletter
\@namedef{subjclassname@2020}{\textup{2020} Mathematics Subject Classification}
\makeatother

\title{Low degree Hurwitz stacks in the Grothendieck ring}

\author{Aaron Landesman}
\address{Department of Mathematics, Massachusetts Institute of Technology, Cambridge, MA 02139-4307, USA}

\author{Ravi Vakil}
\address{Department of Mathematics, Stanford University,
450 Jane Stanford Way, Stanford, CA 94305, USA}

\author{Melanie Matchett Wood}
\address{Department of Mathematics, Harvard University,
1 Oxford St, Cambridge, MA 02138, USA}

\contrib[with an appendix by]{Aaron Landesman and Federico Scavia}

\usepackage{microtype}
\begin{document}

\maketitle
\begin{abstract}
	For $2 \leq d \leq 5$, we show that the class of the Hurwitz
	space of smooth degree $d$, genus $g$
	covers of $\mathbb P^1$
	stabilizes in the Grothendieck ring of stacks
	as $g \to \infty$, and we give a formula for the limit.
	We also verify this stabilization when one imposes ramification
	conditions on the covers, and obtain a particularly simple answer for this limit
	when one restricts to simply branched covers.
\end{abstract}

\section{Introduction}

The main results of this paper are Grothendieck ring analogues of 
classical theorems on the density of
discriminants of number fields of degree at most $5$
\cite{Davenport1969,Bhargava2005,Bhargava2010a}.
Let $\rboxed{\hur d g k}$  be the moduli stack of degree $d$
covers 
of $\bp^1$ with Galois group $S_d$ by smooth geometrically connected
genus $g$
curves over a field $k$,  see \autoref{definition:hur}. Let $\rhur d g k s$ be the
open substack of $\hur d g k$ corresponding to simply branched covers, i.e.,
the open subset where the map to $\bp^1$ has geometric fibers with at least $d-1$ points.  
The main results of this paper are that for each $d \leq 5$, the classes of
these moduli
spaces converge  in the Grothendieck ring as $g \rightarrow \infty$, to
particularly nice limits.
More precisely, we work in a suitably defined
Grothendieck ring of stacks $\grcStacks$, see
\autoref{definition:completed-grothendieck-ring}, where as usual
$\rboxed{\bl} := \{\ba^1\}$
is the class of the affine line.

\begin{theorem}[Theorem A]  \label{t:A}
	Suppose $2 \leq d \leq 5$
	and $k$ is a
	field of
	characteristic not dividing $d!$. In
$\grcStacks$, $$ \lim_{g \rightarrow \infty}  \frac {\{\rhur d g k s\}}
{\bl^{ \dim \hur d g k}} = 1 - \bl^{-2}.  $$  \end{theorem}

\begin{theorem}[Theorem B]  \label{t:B}Suppose $2 \leq d \leq 5$ and $k$ is a
	field of
	characteristic not dividing $d!$.  In
	$\grcStacks$, \begin{align*} \lim_{g \rightarrow \infty} \; \frac {
		\{\hur d g k\}} {\bl^{ \dim \hur d g k}} = \begin{cases}
			1-\bl^{-2}	 & \text{ if } d = 2 \\ (1+\bl^{-1})
			\left( 1 - \bl^{-3} \right) & \text{ if } d = 3 \\ \frac
			1 {(1-\bl^{-1})} \prod_{x \in \bp^1_k} \left(1+\bl^{-2}
			- \bl^{-3} - \bl^{-4} \right) & \text{ if } d= 4 \\
			\frac 1 {(1-\bl^{-1})} \prod_{x \in \bp^1_k}
			\left(1+\bl^{-2}- \bl^{-4} - \bl^{-5} \right) & \text{
if } d = 5.  \end{cases} \end{align*} \end{theorem} 

The products on the right in
the cases $d=4$ and $d=5$ are motivic Euler products in the sense of Bilu
\cite{bilu:thesis, biluH:motivic-euler-products-in-motivic-statistics}.
Alternatively, these can be viewed as powers in the sense of power structures,
as introduced by Gusein-Zade, Luengo, and Melle–Hern\`{a}ndez \cite{Gusein-Zade2004},
see \autoref{subsection:motivic-euler-products}.  
 
\autoref{t:A} is a special case of
\autoref{corollary:simply-branched-hurwitz-class}
while \autoref{t:B} is a special case of \autoref{corollary:full-hurwitz-class}.
Both are consequences of 
\autoref{theorem:hurwitz-class}, describing the limits of branched
covers with specified ramification, along with rates of convergence. 
These results lead to conjectures in higher degree, see
\autoref{s:conjectures}.

\begin{remark}
	\label{remark:}
	The results \autoref{t:A} and \autoref{t:B} 
of this paper are stated above with the restriction that the
Galois group of the cover is $S_d$. 
These results continue to hold when one removes this restriction, except that
when $d = 4$, covers with Galois
group $D_4$ must be removed. One can deduce these claims from
\autoref{lemma:no-transposition-high-codimension}.
\end{remark}

\subsection{Motivation}

Motivations for \autoref{t:A} and \autoref{t:B} come from 
 number theory, topology, and algebraic geometry.

\subsubsection{Arithmetic motivation}
One can also view results relating to counting
number fields of bounded discriminant as ``point counting analogs''
of the stabilization of Hurwitz spaces.
To spell this out, our main results on stabilization of the classes of Hurwitz
spaces, suggest the number of $\mathbb F_q$ points of these Hurwitz spaces
also stabilize in $g$. (This is not actually implied by our results, because we
work in the dimension filtration of the Grothendieck ring, and so it is possible that high codimension
substacks of these Hurwitz spaces contain many $\mathbb F_q$ points which
could potentially alter the $g \to \infty$ limiting behavior of the $\mathbb
F_q$ point counts.)
In the degree $3$ case, stabilization of the number of $\mathbb F_q$
points was  shown by Datskovsky and Wright in
\cite{Datskovsky1988}.
Their results actually count $S_3$ covers of any global field using Shintani zeta functions. 
However a more geometric proof counting $S_3$ covers of any curve over a finite field has also been
given by Gunther in \cite{guntherCOUNTINGCUBICCURVE}.
These results have also been generalized to work in degrees $4$ and $5$
by Bhargava-Shankar-Wang in \cite{bhargavaSW:geometry-of-numbers-over-global-fields-i}.
Analogs over $\bq$
were known much earlier than  these results over global function fields.
That is, instead of counting $\mathbb F_q$ points of Hurwitz spaces, corresponding to
$S_d$ covers of $\bp^1_{\mathbb F_q}$, the arithmetic analog is to count
$S_d$ extensions of $\mathbb Q$.
When $d = 3$, these counts were carried out by 
Davenport and Heilbronn \cite{Davenport1969, Davenport1971}.
When $d =4$ and $d = 5$, the number field counting was done by Bhargava in
\cite{Bhargava2005,Bhargava2010a,bhargava:the-geometric-sieve}.  Our 
theorems can thus be viewed as Grothendieck ring analogs of these number field counting results.
Indeed, the ``Euler products'' occurring in \autoref{t:B} with $\mathbb{L}$ replaced by $p$
are exactly those that occur in the densities of discriminants of $S_d$-number fields of
degree $d\leq 5$ 
\cite{Davenport1969,Bhargava2005,Bhargava2010a}, which in particular demonstrates the great success of the notion of motivic Euler products.
Similarly to our methods, the methods behind the number field counting results only apply when $d\leq
5$ because they rely on specific
parametrizations  \cite{Delone1964,Bhargava2004b, Bhargava2008a} of low degree covers of $\Spec \Z$.

\subsubsection{Topological motivation}

We now describe topological results demonstrating stabilization of Hurwitz
spaces.
One striking result is due to Ellenberg-Venkatesh-Westerland \cite{EllenbergVW:cohenLenstra}, which has 
deep applications to number theory.
Their result \cite[Thm.\, p. 732]{EllenbergVW:cohenLenstra}
implies that the dimension of the $i$th homology $h_i(\rhur 3 g \bc s, \bq)$
stabilizes as $g \to \infty$.
Unfortunately, although their methods apply in the case of degree $3$ covers,
they already fail to apply when $d= 4$, see the remarks in
\cite[p. 732]{EllenbergVW:cohenLenstra}.

If, instead of working with covers of $\bp^1$, one
works with the full moduli stack of curves with marked points, $\mathcal M_{g,n}$,
then these stacks satisfy certain homological stabilities, due to Harer,
Madsen-Weiss, and others. See, for example,
\cite{madsen-weiss:the-stable-moduli-space-of-riemann-surfaces} and
the survey article
\cite{hatcher:a-short-exposition}.

\subsubsection{Algebro-geometric motivation}

Finally, from an algebraic geometric viewpoint, there are some further related
unirationality results on objects of low degree and genus.
For degrees $d \leq 5$ a simple parametrization of degree $d$ covers
was originally given in \cite[Thm.\ 1.1]{miranda:triple-covers-in-algebraic-geometry},
\cite[ Thm.\ 4.4]{casnatiE:covers-algebraic-varieties}, and \cite[Thm.\
3.8]{casnati:covers-algebraic-varieties-ii},
see also
\autoref{theorem:3-structure},
\autoref{theorem:4-structure}, and
\autoref{theorem:5-structure} (as well as \cite[Prop.\ 5.1]{Poonen2008},
\cite[Thm.\ 1.1]{Wood2011}).

There have also been results proving stabilization 
of algebraic data relating to $\hur d g k$.
When $d = 3$, the rational Chow ring of the simply branched Hurwitz space
is known to stabilize to $\bq$ \cite[Thm.\ C]{patel2015chow}.
It is also known that the rational Picard groups stabilize when
$d \leq 5$, due to Deopurkar-Patel \cite[Thm.\ A]{deopurkarP:the-picard-rank-conjecture}.
More recently, stabilization of the rational Chow groups for $d \leq 5$ 
(removing $D_4$ covers when $d = 4$)
was demonstrated in
\cite[Theorem 1.1]{canningL:chow-rings-of-low-degree-hurwitz-spaces}.

There have also been some related stabilization results working in the
Grothendieck ring.
For example, the class of smooth hypersurfaces of degree $d$ in $\bp^n$
stabilizes as $d \to \infty$
in the Grothendieck ring.
This, and various related results are shown by the second and third authors in \cite{Vakil2015}.
Building on this, Bilu and Howe
prove more general stabilization results for sections of vector bundles in the
Grothendieck ring
\cite[Thm.\ A]{biluH:motivic-euler-products-in-motivic-statistics}.
We will use these results crucially in the present paper.

\subsection{Conjectures and questions motivated by Theorems A and B} \label{s:conjectures}

The most natural question following \autoref{t:A} is whether the pattern continues for higher $d$.  The continuation of analogies of this pattern have been conjectured in several different domains.

\subsection{Arithmetic conjectures}
In the context of counting degree $d$ number fields whose Galois closure has
Galois group $S_d$,
Bhargava \cite[Conj.\ 1.2]{Bhargava2010} has conjectured that an analog of \autoref{t:B} holds for all $d$ (which, as mentioned above, is known for $d\leq 5$). Bhargava has given a specific conjectural expression for the Euler factors.
It is natural to ask whether \autoref{t:B} holds for $d\geq 6$ using the analogous
Euler factors. 
That is, one may ask whether \autoref{theorem:hurwitz-class} holds for $d\geq 6$ when all types of ramification are allowed.  Further, the heuristics of \cite{Bhargava2010} also predict the
analog of \autoref{t:A} in the number field counting setting for all $d$ (which again is a theorem for $d\leq 5$ 
\cite[Thm.\ 1.1]{bhargava:the-geometric-sieve}).
Bhargava's heuristics more generally apply to give a conjecture for counting
$S_d$ degree $d$ fields with various ramification restrictions, and the analogy in the Grothendieck ring setting would be to conjecture that \autoref{theorem:hurwitz-class} holds for $d\geq 6$.

The heuristics above are based on a mass formula proven by Bhargava
\cite[Theorem 1.1]{Bhargava2010}. We
prove an analogous mass formula in the Grothendieck ring in
\autoref{theorem:local-aut-class},
which we now state a consequence of.
To make a precise statement, let $\mathscr X_d$ denote the stack over $k$ whose $T$
points are finite locally free degree $d$ Gorenstein covers $Z$ of $T
\times_{\spec k} \spec
k[\varepsilon]/(\varepsilon^2)$ so that for each geometric point $\spec \kappa
\to T$, $Z \times_{T \times \spec k[\varepsilon]/(\varepsilon^2)} \left(\spec \kappa \times \spec
k[\varepsilon]/(\varepsilon^2) \right)$ has $1$-dimensional Zariski tangent space at
each point.
We write
$R \vdash d$ to mean that $R$ is a partition
of $d$.
Given  $R  \vdash d$ comprised of $t_i$ copies of $r_i$ for $i=1,\dots n$,
we define $r(R) := \sum_{i=1}^n (r_i- 1)t_i$ to be its ramification order.
We can then deduce the following corollary of \autoref{theorem:local-aut-class},
also see \autoref{remark:functorial-gerbe-description},
by summing over partitions of $d$
in the same way that
\cite[Proposition 2.3]{Bhargava2010}
was deduced from 
\cite[Proposition 2.2]{Bhargava2010}.
\begin{corollary}
	\label{corollary:}
	For $d \geq 1$ and $k$ a field of characteristic not dividing $d!$,
	in $\grcStacks$,
	\begin{align*}
	\{\mathscr X_d\} = 
	\sum_{R \vdash d} \bl^{-r(R)} = 
	\sum_{j=0}^{d-1} q(d,d-j) \bl^{-j},
	\end{align*}
	where $q(d,d-j)$ is the number of partitions of $d$ into at  exactly $d-j$
	parts.
\end{corollary}

The above heuristics can be expanded to make predictions when other finite groups
replace $S_d$. 
These expanded heuristics are often called the Malle-Bhargava Principle (see
\cite{wood:asymptotics-for-number-fields-and-class-groups}),
though in complete generality the predictions are not always correct.  This
principle, as long as one is imposing only geometric local conditions, (i.e. only
local conditions on ramification,) naturally extends to the Grothendieck ring
setting. 
Then, one can ask in what generality the predictions of the principle
hold. 
Moreover, in the field counting setting, one naturally counts extensions of
global fields other than $\Q$ or $\F_q(t)$, and the analog here would be replacing $\bp^1$ with another fixed curve, which is another interesting direction to try to understand.

In addition to the above conjectures on $S_d$ extensions,
there are also many open questions about Grothendieck ring versions of
other extension counting problems.
One particularly accessible problem may be that of counting $D_4$ extensions.
In \cite[Corollary 1.4]{Cohen2002}, the number of $D_4$ extensions of $\mathbb Q$ was computed
when counted by discriminant, though the answer does not appear to have a simple
closed form, and was expressed in terms of a sum over quadratic extensions of
$\mathbb Q$.
However, in \cite[Theorem 1]{altugSVW:the-number-of-quartic-d4-fields-ordered-by-conductor}
these extensions were counted by conductor, and there was a closed form answer,
expressed in terms of an Euler product.
\begin{question}
	\label{question:}
	What is the asymptotic class of the locus of $D_4$ covers of $\mathbb P^1$ in the
	Grothendieck ring $\grcStacks$ when counted by discriminant or conductor?
\end{question}

Similarly, it would be interesting to compute the class of abelian covers of
$\mathbb P^1$.
\begin{question}
	\label{question:}
	Fix an abelian group $G$. What is the asymptotic class of the locus of $G$ covers
	of $\mathbb P^1$ in the Grothendieck ring $\grcStacks$ when counted by
	discriminant or conductor?
\end{question}
One way to approach this question could be to use that
the moduli spaces of abelian covers of $\mathbb P^1$ can be described
in terms of certain configuration spaces of (colored) points on $\mathbb P^1$.
The classes of such configuration spaces can be extracted from
\cite[\S5]{Vakil2015}.

\subsection{Error terms and second order terms}

It would be interesting to understand the error terms in
\autoref{theorem:hurwitz-class}.
More precisely, in \autoref{theorem:hurwitz-class}, we show the equalities of 
\autoref{t:A} and \autoref{t:B} hold not just in the limit, but even hold for
any fixed $g$ up to codimension 
$\codimbound d g := \min(\frac{g + c_d}{\kappa_d}, \frac{g +
d-1}{d} - 4^{d-3})$,
for $c_3 = 0, c_4=-2, c_5 = -23, \kappa_3 = 4, \kappa_4 = 12, \kappa_5 = 40$. 
We say two classes of dimension $d$ are equal modulo codimension $r$ in 
$\grcStacks$ if their difference lies in filtered part of 
$\grcStacks$ of dimension at most $d - r$.
Concretely, in degree $3$, a special case of
\autoref{theorem:hurwitz-class} says:
\begin{corollary}
	\label{corollary:}
	Suppose $k$ is a
	field of
	characteristic not dividing $6$. 
	Then
	\begin{align*}
	\frac{\{\hur 3 g k\}} {\bl^{ \dim \hur 3 g k}} \equiv
(1+\bl^{-1})\left( 1 - \bl^{-3} \right)
	\end{align*}
	modulo codimension 
	$\frac{g}{4}$ in
	$\grcStacks$. 
\end{corollary}

Focusing on the degree $3$ case, Roberts' conjecture \cite{Roberts2001}
states that the number of degree $3$ field extensions of $\mathbb Q$
of discriminant at most $X$ is $\alpha X + \beta X^{5/6} +
o(X^{5/6})$, for appropriate constants $\alpha, \beta$.
This was proved in \cite{Bhargava2013b} and \cite{Taniguchi2013} independently.
Moreover, the error term was further improved to
$\alpha X + \beta X^{5/6} + O(X^{2/3+\varepsilon})$
in \cite{bhargavaTT:improved-error}.

In the context of function fields, one might similarly expect
$\alpha_q q^{\dim \hur 3 g k} + \beta_q q^{5/6 \dim \hur 3 g k} + o(q^{5/6 \dim \hur
3 g k})$ to count the number of degree 3 extensions of $\mathbb F_q(t)$ of genus $g$, for some constants $\alpha_q,\beta_q$.
Progress towards this was made in 
\cite{zhao:on-sieve-methods-for-varieties-over-finite-fields}.
In the context of the Grothendieck ring, 
as mentioned above, we were able to compute the class of the Hurwitz stack up to
codimension $\codimbound 3 g := \min(\frac{g}{4}, \frac{g + 2}{3}-1) =
\min(\frac{g}{4}, \frac{g-1}{3})$.
Hence, once $g \geq 4$, $\codimbound 3 g  = \frac{g}{4}$.
Since $\dim \hur 3 g k = 2g + 4$, we find
$\frac{5}{6} \dim \hur 3 g k = \dim \hur 3 g k  - \frac{g + 2}{3}$,
and so a weakened form of Roberts' conjecture
is the following:
\begin{conjecture}
	\label{conjecture:second-order}
Suppose $k$ is a
	field of
	characteristic not dividing $6$. 
	Then
	\begin{align*}
	\frac{\{\hur 3 g k\}} {\bl^{ \dim \hur 3 g k}} \equiv
(1+\bl^{-1})\left( 1 - \bl^{-3} \right)
	\end{align*}
	modulo codimension 
	$\frac{g-1}{3}$ in
	$\grcStacks$. 
\end{conjecture}

\begin{remark}
	\label{remark:}
	Note that $\frac{g-1}{3}$ is in fact the second term in the minimum
defining $\codimbound 3 g$. 
There is only one step in our proof where the error term we introduce has
codimension less than $\frac{g-1}{3}$, namely when we apply the sieve of
\cite{biluH:motivic-euler-products-in-motivic-statistics} in
\autoref{proposition:ce-strata-class} and
\autoref{lemma:smooth-sieve-codimension-bound}.
So if the sieving machinery could be improved, it may lead to a proof of
\autoref{conjecture:second-order}.
\end{remark}

\begin{remark}
	\label{remark:}
	In the degree $3$ case, it would be quite interesting to actually find the
second order term, instead of just predicting the codimension of the error.
The paper 
\cite{bhargavaTT:improved-error}
improves the error term in Davenport-Heilbronn to $O(X^{2/3+\epsilon})$,
where $X$ is the discriminant of the relevant cubic extension.
Since $X^{2/3} = X \cdot X^{-1/3}$, when one translates this to a codimension
bound in the Hurwitz stack, it suggests 
one might hope to determine an 
asymptotic expression for
$\{\hur 3 g k\}$
up to codimension $\frac{1}{3} \dim \hur 3 g k$.
Such a computation would be extremely interesting to us, and we expect it would 
require tools far beyond those of the current paper.

Additionally, it would be interesting, though likely more difficult, to determine the codimension of the error
and the second order terms in degrees $4$ and $5$.
\end{remark}

\subsection{Topological conjectures}
If $\conf_n$ denotes the configuration space of
points on $\bp^1$, i.e., the open subscheme of $\sym^n_{\bp^1}$ parameterizing
reduced degree $n$ subschemes of $\bp^1$,
then, for $n\geq 3$, we have $\frac{\{\conf_n\}}{\bl^{\dim \conf_n}} = 1 - \bl^{-2}$ 
in the Grothendieck ring of varieties.
This follows from \cite[Lem.\ 5.9(a)]{Vakil2015} as we explain further toward
the end of \autoref{subsection:d=2-proof}.
There is a map $\rhur d g k s \to \conf_{2g-2 + 2d}$ sending a curve to its
branch locus, 
see \cite{fantechiP:stable-maps-and-branch-divisors}.
Using this, \autoref{t:A} and the explicit formula for
$\{\conf_{2g-2 + 2d}\}$
implies that the source and target of
this map have classes in $\grcStacks$,
defined in
\autoref{definition:completed-grothendieck-ring},
which differ only by a class of high
codimension.
\begin{corollary}
	\label{corollary:}
	For $2 \leq d \leq 5$ and $k$ a field of characteristic not dividing
	$d!$,
	\begin{align*}
		\lim_{g \to \infty} \frac{\{\rhur d g k s \}}{\bl^{\dim \rhur d
		g k s}} 
		=\frac{ \{ \conf_{2g-2 + 2d}\}}{\bl^{\dim \rhur d g k s}}
	\end{align*}
	in $\grcStacks$. 
\end{corollary}

It was also conjectured in
\cite[Conj.\ 1.5]{EllenbergVW:cohenLenstra} that this map
$\rhur d g {\bc} s \to \conf_{2g-2 + 2d}$ induces an isomorphism on $i$th homology
for fixed $d$ and sufficiently large $g$. 
(Technically a slight variant
was conjectured in \cite[Conj.\ 1.5]{EllenbergVW:cohenLenstra}, with $\ba^1$ base in
place of $\bp^1$.)
This is in fact open for $d\geq 3$,
though recent work of Zheng \cite[Thm.\
1.2]{zheng:stable-cohomology-in-the-moduli-space-of-trigonal-curves} proves a closely related result in the $d=3$ case,
by finding the stable cohomology of $\hur 3 g {\bc} $. 
\autoref{t:A} could be seen as an additional motivation for this conjecture, especially for $d\leq 5$.

\subsection{Spelling out some questions}
Despite the numerous parametrizations mentioned above, the question of 
whether there exist
simple parametrizations of covers of degree $6$, or even whether the Hurwitz
stack of genus $g$ degree $6$ covers (for large $g$) is unirational,
remains wide open.

Returning to the simply branched case for simplicity, 
we have now seen several ways in which we could ask whether the spaces $\rhur d g k s$ and 
$\conf_{2g-2 + 2d}$ are similar as $g\ra\infty$. The following questions have
been raised:
\begin{enumerate}
	\item Do they have the same points counts (asymptotically) over $\F_q$?  
	\item Do they have the same cohomology, in some stable limit?  
	\item Do they have the same normalized limit in the Grothendieck ring?  
\end{enumerate}	
We also include: 
\begin{enumerate}
	\item[(4)] Are they piecewise isomorphic up to pieces of codimension going to $\infty$? 
\end{enumerate}
Even though it is not technically about these spaces, in this sequence of
questions one should also include:
\begin{enumerate}
	\item[(1')] Are the asymptotic counts of $S_d$ number fields as
		predicted by Bhargava in \cite{Bhargava2010}?
\end{enumerate}
For $d\geq 6$, it seems progressively harder to believe the questions (1) and
(1'), (2), (3), and (4) could have positive answers, though for $d\leq 5$ the same parametrizations lead to positive answers to (1), (1'), (3), and (4) (and nearly to (2) for $d=3$).

\subsection{Idea of the proof}

The idea of the proof of \autoref{t:A} and \autoref{t:B} is simplest to understand in
the degree $3$ case, so we describe this first.  Miranda \cite{miranda:triple-covers-in-algebraic-geometry}
gave a parametrization of degree $3$ covers of a base scheme, and we explain here how we can apply it
for degree $3$ covers of $\bp^1$.
Any degree $3$ cover of $\bp^1$ has a canonical embedding into a $\bp^1$-bundle
$\bp \sce$ over $\bp^1$.
We can write $\sce \simeq \sco(a) \oplus \sco(b)$ where $a + b = g + 2$ and
$a \leq b$.
We can therefore stratify the Hurwitz space by the isomorphism type of the
bundle $\sce$.
The degree $3$ curves lie in a particular linear series on $\bp \sce$.
The idea is now to compute the locus of smooth curves in this linear system with
particular ramification conditions, and then sum over all splitting types of
bundles $\sce$.
The condition for a degree $3$ cover of $\bp^1$ to be smooth in a fiber over $p$ can be
checked over the preimage of the second order neighborhood of $p$ in $\bp \sce$.
We directly compute the classes of such curves in such an infinitesimal
neighborhood.
Using the notion of motivic Euler products, we can “multiply” these local
classes to obtain the global class of smooth curves in $\bp \sce$
in the relevant linear system, at least up to high codimension.
We then sum these resulting classes over allowed splitting types of $\sce$.
It turns out that we must have $\sce \simeq \sco(a) \oplus \sco(b)$ with 
$a \leq b, 2a \geq b$, and a general member of the relevant linear system on any
such bundle gives a smooth trigonal curve.
Miraculously, in the simply branched case, this motivic Euler product exactly
cancels out with the sum over splitting types of $\bp \sce$, weighted by their
automorphisms. This follows from a motivic Tamagawa number formula for $\sl_2$.

To generalize this idea to the cases of degrees $4$ and $5$ requires substantial
additional work. First, it is no longer the case that curves of degrees $4$ and
$5$ are elements of linear systems on a surface. Rather, there are
parametrizations due to Casnati-Ekedahl \cite{casnatiE:covers-algebraic-varieties,casnati:covers-algebraic-varieties-ii} describing covers of degree
$d$ in terms of pairs of vector bundles $\sce$ and $\scf$, where
$\sce$ has rank $d - 1$ and $\scf \subset \sym^2 \sce$ corresponds to a certain
family of quadrics determined by the curve. In the $d = 4$ case, $\scf$ has
rank $2$, corresponding to $4$ points in $\bp^2$ being a complete intersection of
two quadratics, while in the case $d = 5$, $\scf$ has rank $5$, corresponding to
$5$ points in $\bp^3$ being the vanishing locus of the five $4 \times 4$
Pfaffians of a certain $5 \times 5$ matrix of linear forms.
As in the degree $3$ case, we can then stratify the Hurwitz stack in terms of
the splitting types of $\sce$ and $\scf$, and compute the classes yielding
curves of degree $d$ as sections of a certain vector bundle $\sch(\sce, \scf)$ on $\bp^1$,
depending on $\sce$ and $\scf$.
It is significantly more difficult to calculate the relevant local classes
giving the smoothness conditions in fibers in degrees $4$ and $5$ than it is in
degree $3$.
Nevertheless, we are able to do so by reformulating the question in terms of computing
classes of certain classifying stacks for positive dimensional disconnected
algebraic groups, and applying a number of results of Ekedahl.
The result is Theorem~\ref{theorem:local-aut-class}, which can be viewed as a motivic analog of Bhargava's mass formula \cite{Bhargava2010} counting extensions of local fields in arbitrary degree. 
The specific splitting types of $\sce$ and $\scf$ which appear are not nearly so
simple as in the degree $3$ case, but it turns out that the expressions work out
modulo high codimension. 
For this it is important not to count $D_4$ covers, i.e., degree
$4$ covers which factor through a hyperelliptic curve. 
As in the degree $3$ case, it turns out that, at least in the simply branched case, the sum
over splitting types of $\sce$ and $\scf$ cancel out with the local conditions
we impose, again by the Tamagawa number formula.

\subsection{Outline of the paper}

The structure of the remainder of the paper is as follows.
In \autoref{section:background}, we give background on the Grothendieck ring of
stacks, setup the precise variant we will work in, and recall the notion of
motivic Euler products.
Then, in \autoref{section:generalizing-ce}
we prove generalizations of parametrizations due to Miranda, Casnati-Ekedahl, and
Casnati regarding Gorenstein covers of degree $d\leq 5$.
In degrees $3$ and $4$, generalizations to arbitrary covers of an arbitrary base have been previous shown by Poonen \cite[Prop.\ 5.1]{Poonen2008} and the third author \cite[Thm.\ 1.1]{Wood2011}, but in degree $5$ we require new arguments, and here we present a (mostly) uniform argument for degrees $3,4$, and $5$.
In
\autoref{section:presentation-finite-covers}
we upgrade the above mentioned parametrizations for $d \leq 5$
to describe simple presentations of the stack of degree $d$ Gorenstein covers as
a global quotient stack.
Having settled the above preliminaries, we define
the Hurwitz stacks we will work with in 
\autoref{section:hurwitz-definition}
and prove they are algebraic.
We then describe natural stratifications of these Hurwitz stacks
that arise from the structure of the parametrizations in
\autoref{section:ce-loci}.
Using these parametrizations, we give descriptions of these strata
as quotient stacks in
\autoref{section:presentations-of-ce}.
We next begin our proof of the main theorem by computing the local conditions in
the Grothendieck ring corresponding to a cover being smooth with specified
ramification conditions in 
\autoref{section:local-class}.
In \autoref{section:codimension-bounds}
we establish bounds on the codimension of the contributions to the Hurwitz stack 
from various  strata, which will enable us to prove
our main result in
\autoref{section:main-proof}.
The proof for the case of degree $2$ is slightly different from that
in degrees $3 \leq d \leq 5$, and we complete this in
\autoref{section:2}.

\subsection{Notation} \label{subsection:notation}

Let $\rboxed{X_Z}$ denote the fibered product $X \times_Y
Z$ of schemes, when $Y$ is clear from context.  Similarly define $\rboxed{X_R} := X
\times_Y \spec R$.
For $X$ an integral variety, we use $K(X)$ to denote its function field.

Recall that for $G$ a group, the wreath product $\rboxed{G \wr S_n}$
is  the semidirect product $G^n \rtimes S_n$ where $S_n$ acts on
$G^n$ by the permutation action on the $n$ copies of $G$.   More generally, for $\sce$ a category, let
$\rboxed{\sce \wr BS_j}$ denote the corresponding wreath product of categories
(see \cite[p. 5]{ekedahlGeometricInvariantFinite2009}) so that in particular,
$BG \wr BS_j = B(G \wr S_j)$.

For $\scx$ a stack, and $G$ a group scheme acting on $\scx$, we use
$\rboxed{[\scx/G]}$
to denote the quotient stack.
To avoid confusion with this notation, for $\scx$ a stack, we use $\{\scx\}$ to
denote its class in the Grothendieck ring of stacks, see
\autoref{definition:completed-grothendieck-ring}.

We call an algebraic group $G$ over a field $k$ {\em special} 
if every $G$-torsor over a $k$-scheme $X$ 
is trivial Zariski locally on $X$.

When working in $\grcSpaces$, defined in
\autoref{definition:completed-grothendieck-ring},
	we say two classes $A, B \in \grcSpaces$ of
dimension $d$ are {\em equal modulo codimension $n$} to mean $A - B$ lies in the
dimension $d - n$ filtered part of $\grcSpaces$. 

Let $\rboxed{D} := \spec k[\varepsilon]/(\varepsilon^2)$ be the dual
numbers.    For $X$ a
projective scheme over $Y$, let $\rboxed{\hilb_{X/Y}^d}$ denote the Hilbert
scheme parameterizing degree $d$ dimension $0$ subschemes of $X$ over $Y$.

For $X \rightarrow Y$ a finite locally free map, and $Z$ an $X$-scheme, let
$\rboxed{\res_{X/Y}(Z)\rightarrow Y}$ denote the Weil restriction.
Recall (e.g., \cite[\S
7.6]{BoschLR:Neron}) that the Weil
restriction is the functor defined on $T$ points by $\res_{X/Y}(Z)(T) =
Z(T\times_Y X)$.  For $Z$ quasi-projective over $X$, $\res_{X/Y}(Z)$ is
representable \cite[\S 7.6, Thm.\ 4]{BoschLR:Neron}.

\subsection{Acknowledgements}

We thank 
Hannah Larson as well as multiple tremendously meticulous and helpful referees
for carefully reading the paper and offering 
especially detailed comments.
We also thank Enrico Schlesinger
for pointing out an important omission in our version of the Casnati-Ekedahl
structure theorem, see \autoref{remark:ce-improvements}(10).
This was missed in the published version and we have issued an erratum.
We thank
Manjul Bhargava, 
Margaret Bilu,
Samir Canning, 
Gianfranco Casnati, 
Jordan Ellenberg,
Sean Howe,
Nikolas Kuhn, 
Anand Patel, 
Bjorn Poonen,
Will Sawin,
Federico Scavia, 
Enrico Schlesinger,
Craig Westerland,
Takehiko Yasuda,
and Wei Zhang
for helpful discussions related to this paper.  
AL was supported by the National Science
Foundation 
under Award No.
DMS 2102955.
RV was partly supported by NSF grant DMS-1601211.
MMW was partly supported by a Packard Fellowship for Science and Engineering, a NSF Waterman Award DMS-2140043, and NSF CAREER grant DMS-2052036.

\section{Background: the Grothendieck ring of stacks and motivic Euler
products} 
\label{section:background}

In this section, we begin by defining useful variants of the Grothendieck ring.
Ultimately, we will compute the classes of Hurwitz stacks
in a ring we call $\grcSpaces$, 
obtained from the usual Grothendieck ring
of varieties by quotienting by universally bijective (i.e, 
	bijective on topological spaces after any base change,
	or equivalently	radicial surjective) morphisms, inverting $\bl =
\{\ba^1\}$, and then completing with respect to the dimension filtration.
Following this, we recall basic definitions associated to motivic Euler
products,
following \cite{bilu:thesis} and
\cite{biluH:motivic-euler-products-in-motivic-statistics}.
We also prove these Euler products satisfy a multiplicativity property
(\autoref{lemma:evaluated-product}).

\subsection{Variations of the Grothendieck Ring}
\label{subsection:grothendieck-ring-definitions}

Recall that we are working over a fixed field $k$.
We begin by introducing the Grothendieck ring of algebraic spaces.

\begin{definition} \label{definition:grspaces}Let  $\rboxed{\grSpaces}$  denote the {\em
	Grothendieck ring of algebraic spaces over $k$}.  This is the ring
	generated by classes $\{X\}$ of algebraic spaces $X$  of finite type over $k$ with
	relations given by $\{X\} = \{Y\}$ if there is an isomorphism $X \simeq Y$
	over $k$ and $\{X\} = \{Z\} + \{X-Z\}$ for any closed
        sub-algebraic space $Z  \subset X$.  Applying this in the case
        $Z = X^{red}$, we have $\{X\} = \{X^{red}\}$.        
Multiplication is given by $\{X\} \cdot \{Y\} = \{X \times_k Y\}$. 

More generally, if $S$ is a finite type algebraic space over $k$
we can define $K_0(\on{Spaces}_k/S)$
as the free abelian group generated by
classes of morphisms $X \to S$ with relations $\{X/S\} = \{Z/S\} + \{X - Z/S\}$ 
for any closed sub-algebraic space $Z \subset X$,
where the
implicit maps $f|_Z: Z \to S, f|_{X-Z} : X - Z \to S$ are obtained by
restricting the map $f: X \to S$. 
Multiplication is given by $\{X/S\} \cdot \{Y/S\} = \{X \times_S Y\}$. 

We use a {\em $k$-variety} to mean a reduced, separated, finite type
$k$-scheme. We let $\on{Var}_k$ denote the category of $k$-varieties.
One can similarly define 
$K_0(\on{Var}_k),$ see \cite[\S2]{biluH:motivic-euler-products-in-motivic-statistics}. 
Similarly, for $S$ a $k$-variety, one can analogously define
$K_0(\on{Var}_k/S)$, see \cite[\S2]{biluH:motivic-euler-products-in-motivic-statistics}. 
\end{definition}

\begin{proposition} \label{proposition:spaces-versus-schemes} 
	Let $S$ be $k$-variety.
	The natural map
	$\phi: K_0(\on{Var}_k/S) \to K_0(\on{Spaces}_k/S)$,
	sending $\{X/S\}$ viewed as a $k$-variety to the same $\{X/S\}$ viewed
	as a finite type $k$-space,
	is an isomorphism.
\end{proposition}
%
\begin{proof}
	We first show that for any finite type $k$-space $X$, there is a finite
	collection $X_1, \ldots, X_n$
	of locally closed $k$-subspaces isomorphic to schemes, forming a stratification of $X$. Here,
	a stratification means that a $\overline{k}$ point of $X$ factors through some $X_i$.
        The key input we will need is that finite type
	spaces are quasi-separated, and so they contain a dense open isomorphic to a
	scheme \cite[Thm.\  6.4.1]{olsson2016algebraic}.  
	This, together with the facts that
	$\{X/S\} = \{X^{red}/S\}$ and that any finite type $k$-scheme has a
	stratification by separated finite type $k$-schemes shows that
	any finite type $k$-space $X$ has a stratification by locally closed
	subschemes.

	Next, let us show $\phi$ is surjective. For this, if $\{X/S\}$ is any
	finite type algebraic space, we can use the above stratification to
	write
	$\{X/S\} = \sum_{i=1}^n \{X_i/S\}$, for $X_i$ varieties over $S$,
	implying $\phi$ is surjective.

	We conclude the proof by showing that $\phi$ is injective.
	In order to show this, it is enough to show that any single relation in
	$K_0(\on{Spaces}_k/S)$ is expressable in terms of relations from
	$K_0(\on{Var}_k/S)$.
	More precisely, if $\{X/S\} \in K_0(\on{Spaces}_k/S)$ and $Z$ is a
	closed subspace $Z \subset X$, it suffices to show that the relation
	$\{X/S\} = \{Z/S\} + \{X - Z/S\}$ can be expressed as the image under
	$\phi$ of a sum of relations from $K_0(\on{Var}_k)$.
	To see this is the case, write $\{X/S\} = \sum_{i=1}^n \{X_i/S\}$ and
	where
	$X_1, \ldots, X_n$ are $k$-varieties. 
	Then, let $Z_i$ be the reduction of $X_i \times_X Z$.
	Note that $Z_i$ is a scheme from the definition of algebraic space 
	because $X_i$ and $Z$ are both schemes. Also, $Z_i$ is separated since it
	is a
	closed subscheme of the separated scheme $X_i$.
	Hence, $Z_i$ is a variety.
	We can also write $\{Z/S\} + \{X - Z/S\} = \sum_{i=1}^n
\{Z_i/S\} +\sum_{i=1}^n \{X_i-Z_i/S\}$.
	Therefore, it suffices to verify that
	\begin{align*}
	\sum_{i=1}^n \{X_i/S\} = \sum_{i=1}^n \{Z_i/S\} +
\sum_{i=1}^n \{X_i-Z_i/S\}
	\end{align*}
	is the image under $\phi$ of a sum of relations in the Grothendieck ring
	of varieties. Indeed, it is the sum over $i$ of the relations
	$\{X_i/S\} = \{Z_i/S\} + \{X_i-Z_i/S\}$.
\end{proof}

We next introduce the Grothendieck ring of algebraic stacks.
      \begin{definition} \label{definition:grstacks} The
        {\em Grothendieck ring of algebraic stacks} (over
        $k$) is the ring $\rboxed{\grStacks}$ generated by classes of
        algebraic stacks $\{\scx\}$ {\em of finite type over $k$ with affine
        diagonal}, with the three relations:
	\begin{enumerate}
		\item $\{\scx\} = \{\scy\}$ if there is an isomorphism $\scx
			\simeq \scy$ over $k$,
\item $\{\scx\} = \{\scz\} + \{\scu\}$ for any
	closed substack $\scz \subset \scx$ with open complement $\scu \subset
	\scx$,
\item  $\{\spec_\scx (\sym^\bullet_\scx \sce)\} = \{\scx \times_k \ba^n\}$
	for $\sce$ any locally free sheaf on
	$\scx$ of rank $n$.  
	\end{enumerate}
	Multiplication in this ring is given by $\{\scx\} \cdot \{\scy\} = \{\scx \times_k
\scy\}$. 
 \end{definition}

Note that condition $(3)$ above follows
from the first two in the case of schemes, because vector bundles on schemes are
Zariski locally trivial.
However, vector bundles over stacks may
fail to be Zariski locally trivial, as is the case for nontrivial vector bundles
on $BG$.

\begin{remark} \label{remark:ekedahl-grstacks-equivalence} Let  $\rboxed{\bl}
	:= \{\ba^1_k\}$ denote the class of the affine line.  The natural map $\grSpaces \to
	\grStacks$ induces an isomorphism
	$$\grSpaces[\bl^{-1}, (\bl^n-1)^{-1}_{n \geq 1}] \overset \sim
        \longrightarrow
	\grStacks$$
        	\cite[Thm.\  1.2]{Ekedahl2009}.  Here
	$\grSpaces[\bl^{-1}, (\bl^n-1)^{-1}]$ denotes the ring obtained
	from
	$\grSpaces$ by inverting  $\bl$, as well as 
	$\bl^n - 1$ for all positive integers $n$.
	This isomorphism is motivated by 
	\autoref{definition:grstacks}(3)
	and the fact that
	inverting the classes of $\bl$ and $\bl^n -1$ for all $n$
	is equivalent to inverting the classes of $\gl_n$ for all $n$.
\end{remark}

In order to apply the results of
\cite{biluH:motivic-euler-products-in-motivic-statistics} to sieve out smooth
covers from all covers, we will need to work in a slight modification of the
Grothendieck ring of stacks where we invert universally bijective
(i.e., radicial surjective) morphisms and
then complete along the dimension filtration.  

\begin{definition}
	\label{definition:completed-grothendieck-ring} Let $k$ be a field and let $\rboxed{\grSpaces}$
	denote the Grothendieck ring of algebraic spaces over $k$ from
	\autoref{definition:grspaces}.  From
	$\grSpaces$, we will construct another ring, $\rboxed{\grcSpaces}$, in
	three steps.  
	\begin{enumerate}
		\item 	For any universally bijective map $f: X \to Y$ of
	finite type algebraic spaces over $k$, we impose the additional relation
	that $\{X\} = \{Y\}$.   Call the result  (only for the next paragraph)
	$\rboxed{K_0(\mathrm{Spaces}_k)_{\operatorname{RS}}}$.  
\item  Define
	$\rboxed{\widetilde{K}_0(\mathrm{Spaces}_k)} :=
	\rboxed{K_0(\mathrm{Spaces}_k)_{\operatorname{RS}}[\bl^{-1}]}$.   
	Like $\grStacks$,  the ring 
	$\widetilde{K}_0(\mathrm{Spaces}_k)$
has a filtration given by dimension with the $i$th filtered part
	$\rboxed{F^i \widetilde{K}_0(\mathrm{Spaces}_k)} \subset
	\widetilde{K}_0(\mathrm{Spaces}_k)$ denoting the subset of
	$\widetilde{K}_0(\mathrm{Spaces}_k)$ spanned by classes of dimension at
	most $-i$. 
\item  Finally, let $$\rboxed{\grcSpaces} := \varprojlim_{i \geq 0}
	\widetilde{K}_0(\mathrm{Spaces}_k)/ F^i
	\widetilde{K}_0(\mathrm{Spaces}_k)$$ be the completion along the
dimension filtration.  
	\end{enumerate}
Similarly, for $\grStacks$ the Grothendieck ring of algebraic stacks over $k$ of
\autoref{definition:grstacks}, we analogously define $\grcStacks$ in the same
three steps, replacing the word “Spaces” above by “Stacks”:
\begin{enumerate}
	\item We first impose the relation $\left\{ X \right\} = \left\{ Y
\right\}$ for every universally bijective map of algebraic stacks $f: X \to Y$
of finite type with affine diagonal,
and denote the result
$\rboxed{K_0(\mathrm{Stacks}_k)_{\operatorname{RS}}}$.  
	\item Define
	$\rboxed{\widetilde{K}_0(\mathrm{Stacks}_k)} :=
	\rboxed{K_0(\mathrm{Stacks}_k)_{\operatorname{RS}}[\bl^{-1}]}$.   
	Like $\grStacks$,  the ring 
	$\widetilde{K}_0(\mathrm{Stacks}_k)$
has a filtration given by dimension with the $i$th filtered part
	$\rboxed{F^i \widetilde{K}_0(\mathrm{Stacks}_k)} \subset
	\widetilde{K}_0(\mathrm{Stacks}_k)$ denoting the subset of
	$\widetilde{K}_0(\mathrm{Stacks}_k)$ spanned by classes of dimension at
	most $-i$. 
\item  Finally, let $$\rboxed{\grcStacks} := \varprojlim_{i \geq 0}
	\widetilde{K}_0(\mathrm{Stacks}_k)/ F^i
	\widetilde{K}_0(\mathrm{Stacks}_k)$$ be the completion along the
dimension filtration.  
\end{enumerate}
\end{definition} 

\begin{remark} \label{remark:} In characteristic $0$, identifying classes along
	universally bijective  morphisms does not alter the Grothendieck ring.  See
	\cite[Rem.\ 2.0.2, Rem.\
	7.3.2]{biluH:motivic-euler-products-in-motivic-statistics} for some
	justification of why we are inverting universally bijective morphisms.

	But we do not know if  inverting universally bijective
morphisms alters the Grothendieck ring of spaces or stacks in positive
characteristic.  
\end{remark}

Since Hurwitz stacks are not in general algebraic spaces, but
the results of \cite{biluH:motivic-euler-products-in-motivic-statistics}
apply to the completed Grothendieck ring of algebraic spaces
$\grcSpaces$,
it will be useful to know that one can also obtain $\grcSpaces$ from
$\grStacks$ by inverting universally bijective maps and completing
along the dimension filtration, as we next verify.

\begin{lemma}
	\label{lemma:spaces-ring-equals-stacks-ring}
	The natural map $\grcSpaces \to \grcStacks$ is an isomorphism.
\end{lemma}
\begin{proof}
First note that although we
	constructed $\grcSpaces$ from $\grSpaces$ by first quotienting by
	universally bijective morphisms and then inverting $\bl$, we could have
	equally well first inverted $\bl$ and then inverted universally bijective
	morphisms.  
 Since localization commutes with taking quotients, we obtain the same result by doing these steps in either order.

	Since we can localize and take quotients in any order,  using
	\autoref{remark:ekedahl-grstacks-equivalence}, we can equivalently
	obtain $\grcStacks$ by identifying universally bijective morphisms of
	spaces and then inverting $\bl, \bl^n - 1$ and completing along the
	dimension filtration.  To show $\grcSpaces \to \grcStacks$ is an
	isomorphism, we wish to show that beginning with
	$\widetilde{K}_0(\mathrm{Spaces}_k)$
	and completing along the dimension
	filtration is equivalent to first inverting $\bl^n -1$ for all $n \geq
	1$ and then completing along the dimension filtration.  Indeed, one may
	define a map $\grcStacks \to \grcSpaces$ induced by the map
	$\widetilde{K}_0(\mathrm{Spaces}_k)[(\bl^n-1)^{-1}_{n \geq 0}] \to
	\grcSpaces$ extended by sending the class of $(\bl^n-1)^{-1} \mapsto
	\sum_{i \geq 1} \bl^{-in}$.
	Upon completing along the dimension filtration this defines the desired
	isomorphism $\grcStacks \to \grcSpaces$ inverse to the natural map
	$\grcSpaces \to \grcStacks$ given above.  
\end{proof}

\begin{remark}
	\label{remark:}
	Due to the equivalence of \autoref{lemma:spaces-ring-equals-stacks-ring}, in
	what follows, we will work in $\grcSpaces$.
	In particular, it makes sense to speak of classes of stacks with affine
	diagonal in
$\grcSpaces$ by \autoref{lemma:spaces-ring-equals-stacks-ring}.
\end{remark}

\subsection{Motivic Euler Products} \label{subsection:motivic-euler-products}

We recall the notion of motivic Euler products in the Grothendieck ring, which
is crucial in our proof.  
See \cite{bilu:thesis} for an introduction to motivic Euler products,
and \cite[\S6]{biluH:motivic-euler-products-in-motivic-statistics} for more details.

We begin by introducing notation to
give the definition of motivic Euler products in the setting we will need.  
For a finite multiset $\mu$, with underlying set $I$, we write $\mu=(m_i)_{i \in I}$, where $m_i$ is the number of copies of $i$ in $\mu$. 
Let $X$ be a reduced, quasi-projective 
scheme over a field $k$.  
For any finite multiset $\mu =
(m_i)_{i \in I}$, there is a  finite surjective map $p: \prod_{i \in I} X^{m_i}
\to \prod_{i \in I} \sym^{m_i} X$.  Let $U$ denote the open subscheme 
of $\prod_{i \in I} X^{m_i}$
where
no two coordinates agree and let $\rboxed{C^\mu(X)}$ denote the open
subscheme $p(U) \subset \prod_{i \in I} \sym^{m_i} X$.  Informally speaking,
$C^\mu(X)$ parameterizes configurations of $\mu$-labeled points on $X$.

%
More generally, for $\mathcal X = (X_i)_{i
\in I}$ a collection of reduced, quasi-projective schemes
$X_i$ with morphisms to $X$, and $\mu = (m_i)_{i \in I}$ a
multiset, define $C^\mu_{X}(\mathcal X)$ as the preimage of
$C^{\mu}(X) \subset \prod_{i \in I} \sym^{m_i}X$
under the projection
$\prod_{i
\in I} \sym^{m_i} X_i \to \prod_{i \in I} \sym^{m_i}(X)$.  
As in \cite[Defn.\
6.1.8]{biluH:motivic-euler-products-in-motivic-statistics}, one can extend this
definition to make sense of $C_{X}^\mu(\mathcal A)$ as an element of
$\grSpaces$ where $\mathcal A = (a_i)_{i \in I}$ with $a_i$ in
$K_0(\on{Spaces}_k/X)$.  

%
 Let $\mathbb N$ denote the positive integers.
 Let $\mathcal P$ be the set of non-empty finite multi-sets of positive
 integers, and for such a multiset $\mu=(m_i)_{i\in \mathbb N}$, let
 $|\mu|:=\sum_i i \cdot m_i$.
 Following 
\cite[Section 2.2.2]{Bilu2022},
for $\mathcal A =
(a_i)_{i \in \mathbb N}$ a collection of classes in $K_0(\on{Spaces}_k/X)$, define the
{\em motivic Euler product} \begin{align} 
	\label{equation:motivic-euler-product}	
	\rboxed{	\prod_{x \in X} \left(
	1 + \sum_{i =1}^{\infty} a_{i,x} t^i \right)} := 1+ \sum_{\mu \in \mathcal P}
	C^\mu_{X}((a_i)_{i\in \mathbb N}) t^{|\mu|} \in \grSpaces \llbracket
	t \rrbracket. 
\end{align}   
Here, $\rboxed{a_{i,x}}$ is formal notation to indicate the $a_i$ on which the definition depends.  When we write a class
$b_i\in K_0(\on{Spaces}_k)$ in place of $a_{i,x}$, it indicates that $a_i=[Y_i \times X]-[Z_i\times X],$ where $Y_i,Z_i$ are algebraic spaces of finite type over $k$ such that 
$b_i=[Y_i]-[Z_i]$ and $Y_i \times X,Z_i\times X$ have the natural projection map to $X$.

%

Let $r\in \mathbb N$, and let $I$ be the set of $r$-tuples of non-negative integers, not all $0$.  Note that $I$ is a semigroup under coordinate-wise addition.  Let $\mathcal{P}_r$ denote the set of non-empty finite multisets of elements of $I$, and for $\mu\in \mathcal{P}_r$, let $|\mu|\in I$
denote the sum of the elements of $\mu$.
More generally, for indeterminates $t_1,\dots, t_r$
one can define, for $\mathcal A =
(a_{\underline{i}})_{\underline{i} \in I}$ a collection of classes in $K_0(\on{Spaces}_k/X)$, 
\begin{align} 
	\label{equation:motivic-euler-product-generalized}	
	\prod_{x \in X} \left(
		1 + \sum_{\underline{i} \in I} a_{\underline{i},x} \underline{t}^{\underline{i}}
		\right) := 1+ \sum_{\mu \in \mathcal P_r}
	C^\mu_{X}((a_{\underline{i}})_{\underline{i} \in I}) 
 \underline{t}^{|\mu|}
 \in \grSpaces \llbracket
	t_1,\dots,t_r \rrbracket.
\end{align} 
where for $\underline{i}=(i_1,\dots,i_r)\in I$, we write $\underline{t}^{\underline{i}}$ for $t_1^{i_1}\cdots t_r^{i_r}$.

\begin{warn}
	\label{warning:}
	The left hand side of \eqref{equation:motivic-euler-product} is
	merely (evocative) notation, and has no intrinsic meaning beyond the
	right hand side. 
\end{warn}

In the special cases that we will use them in,  motivic Euler products are the same as the power structures of
Gusein-Zade, Luengo, and Melle–Hern\`{a}ndez \cite{Gusein-Zade2004}.
  We now specialize to the one variable case.

In good circumstances, there is an {\em evaluation map} at $t = 1$ sending a
motivic Euler product, viewed as an element of $\grSpaces \llbracket t
\rrbracket$ to an element of $\grcSpaces,$
as in \cite[Definition 6.4.1 and Notation
6.4.2]{biluH:motivic-euler-products-in-motivic-statistics}.
This makes sense whenever the
motivic Euler product ``converges at $t =1$'', meaning there are only finitely
many terms $\mu$ so that $C^\mu_{X/S}(a)$ is outside any given piece of the
dimension filtration.  

\begin{notation} \label{notation:evaluation-at-1} 
For a
	motivic Euler product 
	$\prod_{x \in X} \left( 1 + a_x t \right)$
	which converges at $t = 1$, we use
	$$\rboxed{\prod_{x \in X} \left( 1 + a_x t \right)|_{t = 1}}$$ to denote the
	evaluation at $t=1$ in $\grcSpaces$.  

	We will often also write
	$\rboxed{\prod_{x \in X} \left( 1 + a_{x} \right)}$ to also denote the
	evaluation of the motivic Euler product $\prod_{x \in X} \left( 1 +
	a_x t \right)$ at $t = 1$ in $\grcSpaces$,
	in order to shorten notation,
	but see
	\autoref{warning:evaluation-notation}.
\end{notation} 

\begin{warn}
	\label{warning:evaluation-notation} Due to the extreme care with which one must handle
	motivic Euler products, we acknowledge that
	\autoref{notation:evaluation-at-1} is not very good notation.  It is
	likely best to think of motivic Euler products as power series in $t$
	which are being evaluated at values of $t$, rather then actual elements
	in $\grcSpaces$, as the manipulations one wants to make have only
	primarily been established in terms of the power series, and not in
	terms of their evaluations in $\grcSpaces$.  We choose to use this
	convention so as to shorten unwieldy formulas.  

	In particular, one must be careful that the two expressions
	$\prod_{x \in X/S} \left(1 + \sum_{i \in I} a_{i,x} p_i( (s_j)_{j \in J} )\right)$
and 
$\prod_{x \in X/S} \left(1 + \sum_{i \in I} a_{i,x} t_i\right)|_{t_i =
p_i(\underline{s})}$
do not necessarily agree. However, when these sets indexing the variables $t_i$
and $s_j$ are finite, and all $p_i((s_j)_{j \in J})$ are monomials, these
two expressions do agree by \cite[Lemma
6.5.1]{biluH:motivic-euler-products-in-motivic-statistics}. 
\end{warn}

An important lemma will be that these Euler products in $\grcSpaces$ are
multiplicative. 
We now verify this, the key input being multiplicativity of motivic Euler
products in 
$\grSpaces \llbracket t_1,t_2 \rrbracket$.

\begin{lemma} \label{lemma:evaluated-product} Suppose $a$ and $b$ are two
	classes in $K_0(\on{Spaces}_k)$  
	such that the Euler products $\prod_{x \in X}
	\left( 1 + a_x t \right)$ and $\prod_{x \in X} \left( 1 + b_x t \right)$
	converge at $t = 1$ in $\grSpaces$.  Then, \begin{align}
		\label{equation:evaluated-product} \prod_{x \in X} \left( 1 +
		a_x \right) \cdot \prod_{x \in X} \left( 1 + b_x \right) =
		\prod_{x \in X} \left( \left( 1 + a_x \right) \left( 1 + b_x
		\right)\right) \end{align} in $\grcSpaces$.  \end{lemma}
		\begin{proof} 
			We would like to say this follows from
			multiplicativity of Euler products \cite[Prop.\
			3.9.2.4]{bilu:thesis}, but the issue is that when we
			apply
			\cite[Prop.\ 3.9.2.4]{bilu:thesis}
                the left hand side of
			\eqref{equation:evaluated-product} is equal to
			\begin{equation}
			\begin{aligned} \label{equation:motivic-product-commute}
				\left. \left(\prod_{x \in X} \left( 1 + a_x t
					\right) \cdot \prod_{x \in X} \left( 1
				+ b_x t\right) \right) \right \rvert_{t = 1} &=
				\left. \left( \prod_{x \in X} \left( 1 + a_x t
					\right) \cdot \left( 1 + b_x t\right)
				\right) \right \rvert_{t = 1} \\ &= \left.
				\left( \prod_{x \in X} \left( 1 + a_x t  + b_x
			t + a_x b_x t^2 \right) \right) \right \rvert_{t = 1}
		\end{aligned}
	\end{equation}
On the other hand, the right hand side of 
\eqref{equation:evaluated-product}
is by definition
		\begin{align} \label{equation:motivic-final-product} \left.
			\left( \prod_{x \in X} \left( 1 + a_x t  + b_x t + a_x
		b_x t \right) \right) \right \rvert_{t = 1}.  \end{align}

The lemma follows because
  $$\prod_{x \in X} \left( 1 + a_x t  + b_x t
		+ a_x b_x s \right)|_{s=t=1}
  =\prod_{x \in X} \left( 1 + a_x t  + b_x t
		+ a_x b_x t \right)|_{t=1}
  $$
and also 
$$\prod_{x \in X} \left( 1 + a_x t  + b_x t
		+ a_x b_x s \right)|_{s=t=1}
  =\prod_{x \in X} \left( 1 + a_x t  + b_x t
		+ a_x b_x t^2 \right)|_{t=1}
  $$ by \cite[Lemma 6.5.1]{biluH:motivic-euler-products-in-motivic-statistics}.
\end{proof}

\section{Parametrizations of low degree covers} \label{section:generalizing-ce}

The key to computing the class of Hurwitz stacks of low degree covers of $\bp^1$
is the parametrization of covers of degree $d \leq 5$ of a general base scheme.  
In the case $d=3$, the first such parametrization was given by Miranda \cite[Thm.\ 1.1]{miranda:triple-covers-in-algebraic-geometry}, for arbitrary degree 3 covers of an
 irreducible scheme over an algebraically closed field of characteristic not equal to 2 or 3.
 Pardini \cite{Pardini1989} later generalized Miranda's result to characteristic 3,
 and Casnati and Ekedahl \cite[Thm.\ 3.4]{casnatiE:covers-algebraic-varieties} generalized the result
 to Gorenstein degree 3 covers of an integral noetherian scheme.  Poonen \cite[Prop.\ 5.1]{Poonen2008}  gave a complete parametrization of degree $3$ covers of an arbitrary base scheme
 (see also \cite[Thm. 2.1]{Wood2011}).  When $d=4$, Casnati and Ekedahl
 \cite[Thm.\ 4.4]{casnatiE:covers-algebraic-varieties} gave a parametrization of 
 Gorenstein degree $4$ covers of an integral noetherian scheme.  The third author \cite[Thm.\ 1.1]{Wood2011}
generalized this to a parametrization of arbitrary degree $4$ covers along with the data
of a cubic resolvent cover (which is unique in the Gorenstein case) over an arbitrary base scheme.
When $d=5$, Casnati \cite[Theorem 3.8]{casnati:covers-algebraic-varieties-ii}
gave a parametrization of 
   degree $5$ covers, satisfying a certainly ``regularity'' condition (see Remark~\ref{remark:new-structure-theorems}), of an integral noetherian scheme.
(We also note that
Wright and Yukie \cite{Wright1992} gave
these parametrizations for a covers of a field, 
and Delone and Faddeev \cite{Delone1964}
and Bhargava \cite{Bhargava2004b, Bhargava2008a}  gave these parametrizations for covers of $\Spec \Z$.  Bhargava's parametrizations
require additional resolvent data for non-Gorenstein covers.
Bhargava, Shankar and Wang \cite[Section 3]{bhargavaSW:geometry-of-numbers-over-global-fields-i} have refined Wright and Yukie's work for covers of global fields.)

In this section, we will prove similar parametrizations, but suited for our particular application.
For our purposes, we would like to parametrize only Gorenstein covers, but over an arbitrary base.
For $d=3,4$, such a result could be deduced directly from \cite[Prop.\ 5.1]{Poonen2008} and \cite[Thm.\ 1.1]{Wood2011} by specializing to Gorenstein covers. However, for the case $d=5$ some new arguments are required both to obtain all Gorenstein covers and to generalize to an arbitrary base.  
For uniformity of exposition, we show how all of the parametrizations of Gorenstein covers can be 
obtained from the approach of Casnati and Ekedahl.

Casnati and Ekedahl \cite{casnatiE:covers-algebraic-varieties} 
prove a structure theorem \cite[Thm.\
2.1]{casnatiE:covers-algebraic-varieties} (a reformulation of \cite[Thm.\
1.3]{casnatiE:covers-algebraic-varieties}),
which describes a minimal resolution of covers of arbitrary degree of an integral scheme.
We will need to extend this
structure theorem   from integral  schemes
to arbitrary (including non-reduced) bases.  Essentially the same proof given in \cite[Thm.\
2.1]{casnatiE:covers-algebraic-varieties} applies, suitably replacing  Grauert's theorem with cohomology and base
change.  We thank Gianfranco Casnati for helpful conversations confirming this.  
We  will then apply this structure theorem to obtain our desired parametrizations of covers in degrees $3,4,$ and $5$, analogously to how it was done by Casnati and Ekedahl in  \cite[Thm.\
3.4, Thm.\ 4.4]{casnatiE:covers-algebraic-varieties} and \cite[Thm.\
3.8]{casnati:covers-algebraic-varieties-ii}.

We also upgrade Casnati's result in degree $5$ in an additional way to deal with all
Gorenstein covers, see \autoref{remark:new-structure-theorems}.

\subsection{The main structure theorem from Casnati-Ekedahl}
\label{subsection:ce-main}

We next recall the main structure theorem and give its proof in the more general
setting.  In essence, it says that degree $d$ Gorenstein covers are
classified by linear-algebraic data.  It is convenient to describe
this as saying that a number of moduli stacks are isomorphic.

We first recall some terminology.
We will consider degree $d$ covers which are finite locally free.
A finite locally free degree $d$ cover is {\em Gorenstein} if the scheme-theoretic fiber $X_y$ over
$\kappa(y)$  is Gorenstein for every $y \in Y$.  For $k$ a field,  a
subscheme $X \subset \bp^n_k$ is {\em arithmetically Gorenstein} if the affine
cone over $X$, viewed as a subscheme of $\mathbb A^{n+1}_k$, is Gorenstein.
For $\sce$ a rank $d-1$  locally free sheaf of $\mathscr O_Y$-modules on $Y$,
let $\rboxed{\pi: \bp \sce \ra Y}$ denote the corresponding projective
bundle $\rboxed{\bp \sce} := \proj \sym^\bullet \sce$.  
We use the term {\em projective bundle} to describe the projectivization of a vector
bundle.
For $\scg$ a sheaf of $\mathscr O_Z$-modules on a scheme or stack $Z$, we use $\scg^\vee := \hom_{\mathscr O_Z}(\scg,
\sco_Z)$ to denote its dual.
Finally, for $\kappa$ a field, a subscheme of $\bp^n_\kappa$ is nondegenerate if it is not contained in any
hyperplane $H \subset \bp^n_\kappa$.

\begin{theorem}[Generalization of \protect{\cite[Thm.\
	2.1]{casnatiE:covers-algebraic-varieties}}, see also
\protect{\cite[Thm.\ 2.2]{casnatiN:on-some-gorenstein-loci}}]
	\label{theorem:ce-main-generalized}
Let $X$ and $Y$ be schemes and let $\rho: X
\ra Y$ be a finite locally free 
Gorenstein cover of degree $d$, for $d
\geq 3$. Fix a vector bundle $\sce'$ of rank $d - 1$ on $Y$ with corresponding
projective bundle $\pi: \bp := \bp \sce' \to Y$, and fix 
an embedding $i: X \ra \bp$ such that $\rho = \pi \circ i$.
We further require that $\rho^{-1}(y) \subset \pi^{-1}(y) \simeq
\bp^{d-2}_{\kappa(y)}$ is a nondegenerate and arithmetically Gorenstein subscheme
for each point $y \in Y$.  
A bundle $\sce'$ and map $i$ satisfying the above properties exists.
Any two such triples $(\bp, \pi, i)$ and $(\bp_2,
	\pi_2,
i_2)$ are uniquely isomorphic, meaning there is a unique isomorphism $\psi: \bp
\simeq \bp_2$ such that $\pi_2 \circ \psi = \pi$ and $\psi \circ i = i_2$.
Moreover, for any such triple $(\bp, \pi, i)$ with $\rho = \pi \circ i$, the
following properties hold. 
\begin{enumerate} 
	\item[(i)] Let $\rboxed{\rho^\#} :
		\sco_Y \ra \rho_* \sco_X$ denote the  map defining
		$\rho: X \ra Y$ and let $\sce := (\coker \rho^\#)^\vee$.  
		Then, $\bp \simeq \bp \sce$.
	 \item[(ii)] The composition $\phi:
		\rho^* \sce \ra \rho^* \rho_* \omega_{X/Y} \ra \omega_{X/Y}$ is
		surjective, and so induces a map 
		$j: X \to \bp \mathscr E$, and $(\bp \sce, \sigma: \bp \sce \to
		Y, j)$ is a triple satisfying the properties above.
		The ramification divisor $R \subset X$ of $\rho$ satisfies
		$\sco_X(R) \simeq \omega_{X/Y} \simeq j^*
		\sco_{\bp \sce}(1)$.  
	\item[(iii)] There is a sequence
		$\scn_0, \scn_1, \ldots, \scn_{d-2}$ of finite locally free
		$\sco_{\bp \sce'}$ sheaves on $\bp \sce'$
		with $\scn_0 :=\sco_{\bp \sce'}$
		and an exact sequence
		\begin{equation} \label{equation:minimal-resolution}
			\begin{tikzcd} 0 \ar {r} &  \scn_{d-2}(-d) \ar
				{r}{\alpha_{d-2}} & \scn_{d-3}(-d+2) \ar
				{r}{\alpha_{d-3}} & \cdots \\
				& \cdots \ar {r}{\alpha_2} &
				\scn_1(-2) \ar{r}{\alpha_1} & \sco_{\bp \sce'} \ar{r} &
		\sco_X \ar{r}& 0, \end{tikzcd}\end{equation}
				such that
				the restriction of
				\eqref{equation:minimal-resolution} to the fiber
			$\rboxed{(\bp \sce')_y} := \pi^{-1}(y)$ over $y$ is a
				minimal free resolution of the structure sheaf
				of 
				$\rboxed{X_y} := \rho^{-1}(y)$ for every point
				$y \in Y$. 
	Given $\rho, \mathscr E', i$ as above,
			the exact sequence 
			\eqref{equation:minimal-resolution}
			is unique up to unique isomorphism, such that the
			isomorphism restricts to the identity map on final nonzero
			term $\mathscr O_X$,
			among all
			sequences with the above listed properties.
The locally free sheaves $\scf_i := \pi_* \scn_i$ on $Y$ satisfy $\pi^* \scf_i
\simeq \scn_i$.
			Further $\scn_{d-2}$ is
				invertible, and, for $i = 1, \ldots, d-3$, one
				has \begin{align} \label{equation:beta} \rboxed{
				\beta_i} := \rk \scn_i = \rk \scf_i =
			\frac{i(d-2-i)}{d-1}\binom{d}{i+1}.  \end{align}
			Moreover, 
			$\pi^* \pi_* \scn_i \simeq
			\scn_i$ for $0 \leq i \leq d-2$, and
			$\shom_{\sco_{\bp \sce'}} (\scn_\bullet,
			\scn_{d-2}(-d)) \simeq \scn_\bullet$. 
			Additionally, the formation
			of $\pi_* \scn_\bullet$ commutes with base change on
			$Y$.  
		\item[(iv)]
			For $\scn_{d-2}$ as in
			\eqref{equation:minimal-resolution}, we have
			$\sce'\simeq \sce$ if and only if $\scn_{d-2} \simeq \pi^*
			\det \sce'$.
		\item[(v)] The pushforward of the map $\alpha_1: \scn_1(-2) \to
			\sco_{\bp}$ along $\pi$
			induces an injection $\scf_1 \to \sym^2 \sce$ and for $d-3
			\geq i \geq 2$, the pushforward $\alpha_i: \scn_i(-i-1)
			\to \scn_{i-1}(-i)$ along $\pi$
induces an injection $\scf_i \to \scf_{i-1} \otimes \sce$.  
		\item[(vi)] For any point $y \in Y$, no subscheme $X_y' \subset
			X_y$ of degree $d-1$ is sent under $\rho$ to a hyperplane of $\pi^{-1}(y)$.
\end{enumerate}
\end{theorem} 

\begin{remark}
	\label{remark:ce-improvements}
	The statement of 
	\autoref{theorem:ce-main-generalized}
	differs in several ways from the original statement \cite[Thm.
	2.1]{casnatiE:covers-algebraic-varieties} from
	\cite{casnatiE:covers-algebraic-varieties}.

	\begin{enumerate}
		\item As pointed out in \cite[Thm.
			2.2]{casnatiN:on-some-gorenstein-loci}, it is necessary
			to add a nondegenerate hypothesis to the statement (which was an oversight in the original result).
		\item We  do not require our base $Y$ to be noetherian.
		\item We do not require our base $Y$ to be integral.
\item	We show that given any two triples $(\mathbb P, \pi, i)$ there
	is a unique isomorphism between them, as in
	the sense of the statement of \autoref{theorem:ce-main-generalized}. In
	\cite[Thm. 2.1]{casnatiE:covers-algebraic-varieties}, it is only shown
	that the bundle $\mathbb P$ is unique.
\item In (ii), we additionally show that $(\bp \sce, \sigma: \bp \sce \to Y, j)$
	is one of the unique above mentioned triples.
\item In (iii) we show the formation of $\pi_* \mathscr N_\bullet$ commutes with
	base change.
\item In (iii), we include the requirement that the isomorphism is unique among
	isomorphisms restricting to the identity on $\mathscr O_X$. This
	assumption was also needed in
	\cite{casnatiE:covers-algebraic-varieties}, but not explicitly stated
	there.
\item	We have also added property (v).
\item We have added property (vi).
\item	We also need to assume the fibers $\rho^{-1}(y) \subset
			\pi^{-1}(y)$ are arithmetically
			Gorenstein (and not only Gorenstein) to make the
			embedding $X \to \mathbb P \mathscr E$ unique. 
			A counterexample
			to the theorem statement without the arithmetically
			Gorenstein hypothesis was provided to us by
			Enrico Schlesinger:
If one embeds $4$ points in $\mathbb P^2$ so that three lie on a line, the
resulting scheme will not be arithmetically Gorenstein by \cite[Example 4.1.11(c)]{migliore:introduction-to-liason-theory}.
Instead, the theorem yields the arithmetically Gorenstein embedding of $4$
points in $\mathbb P^2$ as the intersection of two conics.
	\end{enumerate}
\end{remark}

\begin{proof}
As a first step, we reduce the proof to the case $X$ and $Y$ are noetherian. 
\subsubsection*{Removing noetherian hypotheses}
In view of
the asserted uniqueness, by Zariski descent, we may reduce to the case that $Y$
is affine.  
Because $\rho: X \to Y$ is locally finitely presented as it is
finite locally free, we will next show we can spread out all of the above data to a finite type
scheme $Y_0$.
More precisely, as a first step,
by \cite[Prop. 8.9.1]{EGAIV.3},
we can find some finite type
schemes $Y_0$ and $X_0$, a map $\rho_0: X_0 \to Y_0$ and a map $Y \to Y_0$ so that
$\rho$ is the base change of $\rho_0$ along $Y \to Y_0$.
By the various spreading out results in \cite[\S8]{EGAIV.3}
after possibly replacing $Y_0$ with another finite type scheme, we may
additionally assume
$\rho_0$ is Gorenstein, $\sce'$ is the pullback of a
vector bundle $\sce'_0$ on $Y_0$,
and the triples $(\bp, \pi, i)$ and $(\bp_2, \pi_2, i_2)$ are base changes
of corresponding triples on $Y_0$.
Nearly all parts of the theorem, except the unique isomorphism of two triples $(\bp, \pi,
i)$ and $(\bp_2, \pi_2, i_2)$ and the unique isomorphism in (iii),
follow from the corresponding statement over $Y_0$.
However, if these isomorphisms are not unique, there will be some noetherian
scheme to which two different such isomorphisms descend, and hence this claim
can be verified after replacing $Y_0$ with another noetherian scheme.
In particular, it suffices to prove the theorem in the case $Y$ and $X$
are noetherian, and even finite type over $\spec \bz$.
This removes the noetherian hypothesis, addressing
\autoref{remark:ce-improvements}(2).

For the remainder of the proof, we assume $X$ and $Y$ are noetherian.
The proof given in
\cite[Thm.\ 2.1]{casnatiE:covers-algebraic-varieties}
is broken up into steps A, B, C, and D. Step A has a minor inaccuracy
which we next address. The only generalization needed occurs in step B, while
steps C and D go through without change.

\subsubsection*{Addressing step A}
We next explain the proof of Step A, though we make the additional assumption
that the field $k$ is
infinite.
Before explaining this proof, we remark on an error in the proof of Step A from
\cite[Step A, p. 443]{casnatiE:covers-algebraic-varieties} when $k$ is finite.
\begin{remark}
	\label{remark:}
	Let $A$ be a finite $k$-algebra $A$ with maximal
	ideals $\mathfrak m_1, \ldots, \mathfrak m_p$. 
	Let $\eta: A \to k$, be a generalized trace map, i.e., a surjection of $k$-vector spaces
	such that the only ideal contained in the kernel is the $0$ ideal. 
	It is then claimed 
 that
there exists
  $a \in \ker \eta -
	\cup_{i=1}^p \mathfrak m_i$.  

	This is 
    not always true over finite fields, such as when $k= \spec \mathbb F_2$ and $A = \mathbb F_2^5$ and
	$\eta : A \to k$ is the map given by summing the five coordinates.
	Indeed, the oversight in \cite[Step A, p.
	443]{casnatiE:covers-algebraic-varieties} is that while
	over infinite fields, $\ker \eta \subset \cup_{i=1}^p
	\mathfrak m_i$ implies $\ker \eta \subset \mathfrak m_i$ for a single
	$i$,  this does not always hold over finite fields.  It is straightforward 
 to check that
 this claim holds over an infinite field.
	Since we cannot have $\ker \eta \subset \mathfrak m_i$ by the definition of a
	generalized trace map, over infinite fields we conclude that $\ker \eta
	\subsetneq \cup_{i=1}^p
	\mathfrak m_i$.
\end{remark}

	Having explained the error when $k$ is finite, we now conclude our
	commentary on the proof of Step A.
	As mentioned above, the proof still works correctly in the case $k$ is
	infinite.
	We also note
	that in the statement of \cite[Lem., p. 119]{schreyer1986syzygies} which
	is cited in \cite[Step A, p. 443]{casnatiE:covers-algebraic-varieties},
	the subscheme $D$ there should have degree $d$ and lie in $\bp^{d-2}$,
	as opposed to degree $d-2$ in $\bp^{d-1}$.  
	Note that in order to apply \cite[Lem., p. 119]{schreyer1986syzygies},
	it is necessary to use the
	hypothesis that $X \subset \bp\sce$ is nondegenerate, a hypothesis
	which was omitted in \cite[Thm.\
	2.1]{casnatiE:covers-algebraic-varieties}, addressing
	\autoref{remark:ce-improvements}(1).
	At this point, (vi) follows from 
	\cite[Lem., p. 119]{schreyer1986syzygies}, addressing
	\autoref{remark:ce-improvements}(8).

\subsubsection*{Addressing step B}
	Having established the result when $Y = \spec \overline{k}$, 
	it remains to carry out the proof for general bases following \cite[Step B, C,
	and D, p. 445-447]{casnatiE:covers-algebraic-varieties}.
	In what follows, we next recapitulate the argument for step B \cite[p.
	445]{casnatiE:covers-algebraic-varieties}, modifying the application of
	Grauert's theorem to one of cohomology and base change, which allows us
	to remove the integrality hypothesis on $Y$, as in
	\autoref{remark:ce-improvements}(3).

	Recall the statement of Step B: 
\begin{enumerate}
\item[Step B:] Suppose there is a factorization $\rho =
	\pi \circ i$, for $\pi: \bp \ra Y$ a projective $\bp^{d-2}$ bundle and
	$i: X \ra \bp$ an embedding with $X_y$ a nondegenerate arithmetically Gorenstein
	subscheme of $\bp_y$ for each $y \in Y$.  Then,
	\eqref{equation:minimal-resolution} exists, is unique up to unique
	isomorphisms,  
	restricts to a minimal free resolution of $\sco_{X_y}$ over each point
	$y \in Y$,
	and $\pi^* \pi_* \scn_\bullet \simeq \scn_\bullet$.
\end{enumerate}

	Note that when it is written the resolution is unique up to unique
	isomorphism in Step $B$, the statement implicitly means such an isomorphism is unique up to
	those restricting to the identity on $\mathscr O_X$, as if such a
	specification were not given, we could compose with multiplication by a
	unit. This is the reason for the modification from
	\autoref{remark:ce-improvements}(7).

	We next observe that it suffices to prove a version of Step $B$ where we
	replace $Y$ with a geometric point over $y$.
	To be more precise, in order to verify Step B, it suffices to verify
	Step B', given as follows.

	\begin{enumerate}
		\item[Step B':] Suppose there is a factorization $\rho =
	\pi \circ i$, for $\pi: \bp \ra Y$ a projective $\bp^{d-2}$ bundle and
	$i: X \ra \bp$ an embedding with $X_y$ a nondegenerate arithmetically Gorenstein
	subscheme of $\bp_y$ for each geometric point $y$ of $Y$.  Then,
	\eqref{equation:minimal-resolution} exists, is unique up to unique
	isomorphisms,  
	restricts to a minimal free resolution of $\sco_{X_y}$ over each
	geometric point $y$ of $Y$
and $\pi^* \pi_* \scn_\bullet \simeq \scn_\bullet$.
\end{enumerate}

	We now explain why Step B' implies Step B.
	Indeed, the conditions
	of $X_y$ being a nondegenerate arithmetically Gorenstein subscheme and
	for a resolution of $\mathscr O_{X_y}$ being a minimal free resolution
	may be verified after replacing $y$ with a geometric point $\overline
	{y}$ mapping to $y$.
	Therefore, Step B' implies Step B.

	We next verify Step B'.
	In what follows, we therefore use $y$ to denote a geometric point of
	$Y$, as opposed to a point of the underlying topological space whose
	with scheme structure given as the spectrum of the residue field at that
	point.

	For the remainder of the verification of Step B', we only handle the case $d \geq 4$. The case $d =
	3$ is quite analogous to the case $d \geq 4$, though significantly
	easier as the resolution has length $2$.
	
	Define maps $j_y, i_y$ as in the diagram \begin{equation}
	\label{equation:} \begin{tikzcd} X \ar {r}{i} \ar {rd}{\rho} & \bp \ar
{d}{\pi} & \bp_y \ar {l}{j_y} \ar{d} \\ & Y & y \ar{l}{i_y}
\end{tikzcd}\end{equation} Letting $\rboxed{\sci}$ denote the ideal sheaf
of $X$ in $\bp$, we claim that $j_y^*\sci$ is the ideal sheaf of $X_y$ in
$\bp_y$.  To see this, we only need to verify that $j_y^* \sci \ra j_y^*
\sco_{\bp} \ra j_y^* \sco_X$ is exact.  Since $\sco_X$ is flat over $Y$, we will
verify more generally that for $\sch, \scg, \scf$ three sheaves on $X$ with
$\scf$ flat over $Y$, and an exact sequence $0 \ra\sch \ra \scg \ra \scf \ra 0$,
the pullback sequence $0 \ra j_y^* \sch \ra j_y^* \scg \ra j_y^* \scf \ra 0$ is
exact.  Indeed,
this holds because $R^1 j_y^* \scf = \stor_1^{\sco_\bp}(\scf, \sco_{\bp_y}) =
\stor_1^{\sco_Y}(\scf, \kappa(y)) = 0$.  Here we are using that $\scf$ is flat
over $Y$ for the final vanishing and $\scf \otimes_{\sco_\bp} \sco_{\bp_y}
\simeq \scf \otimes_{\sco_Y} \kappa(y)$ for the equality of $\stor$ sheaves.
	
	Next, \cite[Step A, p. 443]{casnatiE:covers-algebraic-varieties}
	provides a resolution of $\sci_{X_y/ \bp_y} = j_y^*\sci$ of the form
	\begin{equation} \label{equation:ideal-sheaf-fiber} \begin{tikzcd} 0 \ar
			{r} & \sco_{\bp_y}(-d) \ar {r}{\alpha_{d-2,y}} &
			\sco_{\bp_y}(2-d)^{\oplus \beta_{d-3}} \ar
			{r}{\alpha_{d-3,y}} & \cdots \\
			& \cdots \ar {r}{\alpha_{2,y}} &
			\sco_{\bp_y}(-2) \ar{r}{\alpha_{1,y}} & j_y^* \sci
			\ar{r}&  0.  \end{tikzcd}\end{equation} 
	Note here that we have only verified Step A in the case $y$ is the
	spectrum of an algebraically closed field, but at this point we are
	assuming that $y$ is a geometric point, as we are verifying Step B'.
	
	We claim
			$j_y^* \sci$ is $3$-regular, in the sense of
			Castelnuovo-Mumford regularity, i.e., $H^i( \bp_y, j_y^*
			\sci(3-i)) = 0$ for $i \geq 1$.  To verify this, it
			follows from the definition of regularity that for an
			exact sequence $0 \ra \scf' \ra \scf \ra \scf'' \ra 0$
			of sheaves with $\scf'$ $m+1$-regular and $\scf$
			$m$-regular, $\scf''$ is also $m$ regular.  Using this
			and the fact that $\sco_{\bp_y}(-k)$ is $k$-regular (and
			hence it is also $k+1$ regular by \cite[Lem.\
		5.1(b)]{FantechiGIK:fundamentalAlgebraicGeometry}), it follows
		by induction that $\im \alpha_{d-i,y}$ is $d-i+2$ regular.
		Therefore, $j^*_y \sci= \im \alpha_{1,y}$ is $3$-regular.  By
		\cite[Lem.\ 5.1(b)]{FantechiGIK:fundamentalAlgebraicGeometry},
		we obtain $H^1(\bp_y, j_y^* \sci(n)) = 0$ for $n \geq 2$.
		Hence, by cohomology and base change, $R^1 \pi_* \sci(n) = 0$
		for $n \geq 2$. Note that, often, cohomology and base change is
		only stated in the case $y$ is a point (as opposed to a
		geometric point) but the case that $y$ is a point follows from
		the case that $y$ is a geometric point since the vanishing of
		cohomology groups can be verified after base change to an
		algebraic closure, using flat base change.

	For our next step, we verify that $\pi_* \mathscr I(n)$ commutes with base
	change on $Y$ for $n \geq 2$.
	For $\scf$ a sheaf, let us denote by $\rboxed{\phi^i_y(\scf)} : R^i
	\pi_* \scf \otimes \kappa(y) \ra H^i(X_y, \scf|_{X_y})$ the natural base
	change map.  Then we have
	seen above that, for $n \geq 2$, $\phi^1_y(\sci(n))$ is an isomorphism
	at all $y$. Further, $R^1 \pi_* \sci(n)$ is locally free (and in fact
	equal to $0$) which implies by cohomology and base change that
	$\phi^0_y(\sci(n))$ is an isomorphism for all $n \geq 2$.  In other
	words, the formation of $\pi_* \sci(n)$ then commutes with base change
	on $Y$.  Further, again by cohomology and base change, $\pi_* \sci(n)$
	is a locally free sheaf when $n \geq 2$ (since the condition from the
		theorem on cohomology and base change that
	$\phi^{-1}_y$ be an isomorphism is vacuously satisfied).
			
	Set $\rboxed{\scf_1} := \pi_* \sci(2)$ and $\scn_1 := \pi^* \scf_1$.
	Let $\rboxed{\alpha_1} : \scn_1(-2) \ra \sci$ denote the evaluation map
	coming from the adjunction $\pi^* \pi_* \sci(2) \otimes \sco_\bp(-2) \ra
	\sci(2) \otimes \sco_{\bp}(-2) \ra \sci$.  As we have shown above, the
	formation of $\scf_1$, and hence $\scn_1$, commutes with base change.
	Further, naturality of the map $\alpha_1$, coming from the adjunction,
	also implies $j_y^*(\alpha_1) = \alpha_{1,y}$.  Therefore, $\alpha_1$ is
	surjective, as its cokernel has empty support.

	We next construct sheaves $\scf_i$ and $\scn_i$ inductively, with
	$\scn_i = \pi^* \scf_i$, for $2 \leq i \leq d-3$.  Let $\rboxed{\sca_1}
	:= \sci$.  For $i \geq 2$, 
	assume inductively we have constructed the map $\alpha_{i-1}$ and	
	define $\rboxed{\sca_i} := \ker \alpha_{i-1}$.
	Analogously to the above verification that $j_y^* \sci$ is $3$-regular,
	it follows that $j_y^* \sca_i$ is $i+2$ regular.  Therefore, by
	\cite[Lem.\ 5.1(b)]{FantechiGIK:fundamentalAlgebraicGeometry},
	$H^1(\bp_y, j_y^* \sca_i(k)) = 0$ for $k \geq (i + 2) - 1 = i+1$.
	Analogously to the above case when $i = 1$, it follows from cohomology
	and base change that $R^1 \pi_* \sca_i(k) = 0$ for $k \geq i+1$, $\pi_*
	\sca_i(k)$ is locally free for $k \geq i+1$, and the formation of $\pi_*
	\sca_i(k)$ commutes with base change for $ k \geq i+1$.  Then, set
	$\rboxed{\scf_i} := \pi_* \sca_i(i+1)$ and $\rboxed{\scn_i} := \pi^*
	\scf_i$.  
	
	We next construct the map $\alpha_i: \scn_i \ra \scn_{i-1}$.
	Begin with the inclusion $\sca_i(i+1) \ra \scn_{i-1}(1)$
	(obtained by twisting the inclusion $\sca_i \ra \scn_{i-1}(-i)$, coming
	from the definition of $\sca_i$, by $i+1$).  Apply $\pi^* \pi_*$ to
	obtain a map $\pi^* \pi_* \sca_i(i+1) \ra \pi_* \pi^* \scn_{i-1}(1)$.
	Twist by $-i-1$ which yields the composite map \begin{equation}
		\begin{aligned} \label{equation:adjunction-map} \scn_i(-i-1) &= \left(
			\pi^* \pi_* \sca_i(i+1) \right)(-i-1) \\ &\ra \left(
			\pi^* \pi_* \scn_{i-1}(1) \right)(-i-1) \\ &\simeq
			\left( \scn_{i-1} \otimes \pi^* \pi_* \sco(1) \right)
			(-i-1) \\ &\ra \scn_{i-1} (-i), \end{aligned}
		\end{equation} which we call $\rboxed{\alpha_i}$.  Since
		$\scn_i$ commutes with base change, and this map is obtained
		from adjunction, the formation of $\alpha_i$ also commutes with
		base change.  Also, since pushforward is left exact, we obtain
		condition $(v)$ in the theorem from the above construction of
		$\scf_i$, provided we show the above construction is the unique
		such one as in the statement (which will be done later in the
		proof). This addresses \autoref{remark:ce-improvements}(9).

			Finally, we similarly construct $\scf_{d-2}$,
			$\scn_{d-2}$, and $\alpha_{d-2}$, assuming we have
			constructed $\alpha_{d-3}$.  Let
$\rboxed{\sca_{d-2}} := \ker \alpha_{d-3}$.  By cohomology and base change, we
find $j_y^* \sca_{d-2}$ is in fact $d$-regular (as opposed to only $d-1$
regular, as was the case for $\sca_i$ with $i < d-2$).  Therefore, by cohomology
and base change, we find $R^1 \pi_* \sca_{d-2}(-d) = 0$ and also that $\pi_*
\sca_{d-2}(-d)$ is locally free and commutes with base change.  We set
$\rboxed{\scf_{d-2} }:= \pi_* \sca_{d-2}(-d)$ and $\rboxed{\scn_{d-2}} := \pi^*
\scf_{d-2}$.  Analogously to \eqref{equation:adjunction-map}, there is a canonical map
$\alpha_{d-2} :\scn_{d-2}(-d) \ra \scn_{d-3}(-d+2)$ coming from adjunction which
commutes with base change.  Altogether, we have constructed a complex as in
\eqref{equation:minimal-resolution} which commutes with base change on $Y$ and
restricts to the minimal free resolution \eqref{equation:ideal-sheaf-fiber} on
each fiber $y \in Y$.  It follows from Nakayama's lemma that the complex
\eqref{equation:minimal-resolution} is exact, because it is exact when
restricted to each fiber over $y \in Y$.

Further, because $\scn_i = \pi^* \scf_i$, it follows from the projection formula
that $\pi_* \scn_i \simeq \pi_*(\sco_\bp \otimes \pi^* \scf_i) \simeq \pi_*
\sco_\bp \otimes \scf_i \simeq \scf_i$, and so $\pi^* \pi_* \scn_i \simeq
\scn_i$.

We next verify uniqueness of our constructed resolution $\scn_\bullet$, up to
unique isomorphism, in the sense claimed in (iii).  Suppose $\scm_\bullet$ is another such resolution which
restricts to a minimal free resolution over each geometric fiber over $y \in Y$.
Over any local scheme $\spec \sco_{y, Y} \subset Y$, there is an
isomorphism $\phi_U: \scn_\bullet|_{\spec \sco_{y,Y}} \simeq
\scm_\bullet|_{\spec \sco_{y,Y}}$ by a sheafified version of \cite[Thm.\ 20.2]{Eisenbud:commutativeAlgebra}.  Such an isomorphism spreads out to an
isomorphism over some affine open $U \subset Y$.  Further, this isomorphism is
unique up to homotopy by a sheafified version of \cite[Lem.\ 20.3]{Eisenbud:commutativeAlgebra}.  We claim there are no nonzero homotopies
$s: \scn_\bullet|_U \ra \scm_\bullet|_U$.  Indeed, such an homotopy would yield
a map $s_i: \scn_i|_U \ra \scm_{i+1}|_U$.  We wish to show this map is $0$. To
check it is $0$, it suffices to show it is $0$ over each $y \in Y$. Over a point
$y \in Y$, this corresponds to a map $\sco_{\bp_y}(a)^{\oplus b} \ra
\sco_{\bp_y}(c)^{\oplus d}$ with $c < a$.  It follows that there are no nonzero
such maps, so the isomorphism $\phi_U$ is unique.  Hence, by this uniqueness, we
obtain via Zariski descent an isomorphism $\phi: \scn_\bullet \simeq
\scm_\bullet$.  This isomorphism is unique because it is unique when restricted
to each member of an open cover.

\subsubsection*{Addressing steps $C$ and $D$}
We have completed the verification of \cite[Thm.\
2.1, Step B]{casnatiE:covers-algebraic-varieties} and now note that
steps C and D given in the proof of \cite[Thm.\
2.1]{casnatiE:covers-algebraic-varieties} go through without change.
Recall that Step D states that the factorization $\rho = \pi \circ i$ exists.
However, the proof shows more: it shows that the triple $(\mathbb P \mathscr E, \sigma, j)$ gives such a
triple, where $\sigma: \mathbb P \mathscr E \to Y$ is the structure map. 
This concludes the verification of part (ii), as mentioned in
\autoref{remark:ce-improvements}(5).

\subsubsection*{Addressing uniqueness of the triples}
At this point, we have proved everything except the uniqueness of the triple
$(\mathbb P, \pi, i)$. We conclude the proof by verifying this statement, which
will complete the verification of the modification noted in \autoref{remark:ce-improvements}(4).
We have shown so far in part (i) that if 
$(\mathbb P_1, \pi_1, i_1)$ and
$(\mathbb P_2, \pi_2, i_2)$ are two triples as in the statement of
\autoref{theorem:ce-main-generalized},
then there is an isomorphism $\mu : \mathbb P_1 \simeq \mathbb
P_2$.
Since $\mu$ is an isomorphism of projective bundles over $Y$, we have $\pi_1 \circ
\mu \simeq \pi_2$.
Using this and (ii), we can reduce to the case that $\mathbb P_1 \simeq \mathbb P_2
\simeq \mathbb P \mathscr E$:
it suffices to find an automorphism $\psi: \mathbb P \to \mathbb P$
over $Y$
so that $\psi \circ i = i_2$, and moreover show this automorphism $\psi$ is the
unique one with
this property.

To verify existence and uniqueness of $\psi$, we first reduce to the case $Y$ is the spectrum of a local
	ring.
We know that both $i_1^* \sco_{\bp (\sce)}(1) \simeq \omega_{X/Y}$ and 
$i_2^* \sco_{\bp (\sce)}(1) \simeq \omega_{X/Y}$,
	by \autoref{theorem:ce-main-generalized}(ii).  Hence, we obtain that the
	automorphism $\psi$ is induced by some automorphism $\phi$ of $\pi_*\omega_{X/Y}$,
	determined up to unit.
		The maps $i_1$ and $i_2$ induce two surjections
	$q_1, q_2: \pi_* \omega_{X/Y} \to \sco_Y$ with the maps $i_1$ and $i_2$ 
	coming via the linear subsystems $\ker (q_1)$ and $\ker(q_2)$.
	To show we have an induced map between $\ker(q_1)$ and $\ker(q_2)$,
	which are both abstractly isomorphic to $\sce$, it is enough to show
	that, up to unit, $q_1 = q_2 \circ \phi$.
	We may verify this locally, and hence assume $Y$ is the spectrum of a
	local ring.

	We conclude by verifying existence and uniqueness of $\psi$ in the case $Y$ is the
	spectrum of a local ring.
	Using \autoref{theorem:ce-main-generalized}(vi),
	in both of the maps $i_1$ and $i_2$, 
	there is no subscheme of degree $d - 1$ on the closed fiber contained in
	a hyperplane, and hence the same holds over the whole local scheme $Y$.
	We may rephrase this as the condition that
	the two relative hyperplane sections of $\bp \sce$ associated to $q_1$ and $q_2$ do not meet
	$i_1(X)$ and $i_2(X)$.
	Equivalently the two hyperplane sections associated to $q_1$ and $q_2$ are nowhere vanishing on $X$, and
	therefore related by a unit.
	By modifying $\phi$ by this unit, we may assume $q_1 = q_2 \circ \phi$.
	This verifies that $\phi$ is unique up to unit, and hence that $\psi$ is
	unique.
	Under the above identifications, the image of $\sce \to \pi_*
	\omega_{X/Y}$ is identified with the kernel of the natural map $\pi_*
	\omega_{X/Y} \to \sco_X$ dual to $\rho^\#$.  Since this map is also
	fixed by the resulting automorphism $\phi$, the automorphism $\phi$ of $\pi_*
	\omega_{X/Y}$ restricts to an automorphism of $\sce$ which induces the
	desired automorphism $\psi: \mathbb P \mathscr E \to \mathbb P \mathscr E$. 
\end{proof}

The following useful corollary tells us that any two ``canonical embeddings''
of a Gorenstein cover are related by an automorphism of $\bp \sce$ coming from
$\sce$.
A special case of this was stated in 
\cite[Corollary 2.3]{casnatiN:on-some-gorenstein-loci}, though the proof there
seems quite terse, as it omits the verification of uniqueness of the triple
$(\bp \sce', \pi, i)$ which we carry out in
\autoref{theorem:ce-main-generalized}.

\begin{corollary} \label{corollary:automorphism-of-e} With notation as in
	\autoref{theorem:ce-main-generalized}, suppose we are given $\rho: X \to
	Y$ and two embeddings $i_1: X \to \bp \sce$ and $i_2: X \to \bp \sce$ so that
	$\rho = \pi \circ i_1 = \pi \circ i_2$ and $\rho^{-1}(y)$ is
	arithmetically Gorenstein and nondegenerate under both embeddings $i_1$
and $i_2$. 
Then, the unique isomorphism $\psi: \bp \sce \to \bp \sce$ taking
$i_1(X)$ to $i_2(X)$ is induced by an automorphism of $\sce$.  
\end{corollary}
\begin{proof} 
	This is a direct consequence of the uniqueness property for triples
	$(\mathbb P, \pi, i)$ as stated in
	\autoref{theorem:ce-main-generalized}, applied to two triples $(\mathbb
	P \mathscr E, \pi, i_1)$ and $(\mathbb P \mathscr E, \pi, i_2)$.
\end{proof}

\subsection{Low degree parametrizations}
\label{subsection:low-degree-ce-structure-theorems}

We now apply Theorem~\ref{theorem:ce-main-generalized}, as in the work of Casnati and Ekedahl,
to obtain parametrizations of Gorenstein covers of degrees $3$, $4$, and $5$.


\begin{remark}
	\label{remark:new-structure-theorems}
	Our parametrization in degree $5$,
	\autoref{theorem:5-structure},
	is stronger than  previous work in several ways.
	The similar result in degree $5$ proven in \cite[Thm.
	3.8]{casnati:covers-algebraic-varieties-ii}
	has certain additional restrictions on the covers and sections
	that Casnati refers to as being ``regular.''
	This regularity condition amounts to the assumption that
	the map $\wedge^2 \scf^\vee \otimes \det \sce \to \sce$ associated to a
	section $\eta \in \sch(\sce, \scf)$ is surjective.
	Additionally, \cite[Thm.
	3.8]{casnati:covers-algebraic-varieties-ii}
	does not claim there is a bijection between covers and sections up to
	automorphisms of $\sce$ and $\scf$,
	but only gives constructions of maps in both directions.
	Further, \cite[Thm.
	3.8]{casnati:covers-algebraic-varieties-ii} is stated for degree $5$ finite flat 
	surjective maps $X \ra Y$ with $Y$
	integral and noetherian,
	whereas ours hold for arbitrary schemes $Y$.
\end{remark}

To introduce notation simultaneously in the cases of degrees $3,4$, and $5$,
we use the following notation.
\begin{notation}
	\label{notation:h-sheaf}
	Let $d \in \{ 3, 4, 5\}$.  Let $Y$ be a scheme. 
	Fix a locally free sheaf $\sce$ on $Y$ of rank $d-1$. 
	If $d = 4$, let $\scf$ be a locally free sheaf on $Y$ of rank $2$ and if $d
	= 5$, let $\scf$ be a locally free sheaf on $Y$ of rank $5$.  
	We use the tuple $(\sce, \scf_\bullet)$ to denote the pair
	$(\sce,\scf)$ when $d = 4$ or $d = 5$ and to denote $\sce$ when $d = 3$.
	Define the associated sheaf
	\begin{align}
	\label{equation:associate-h} 
	\rboxed{\sch(\sce, \scf_\bullet)} :=
	\begin{cases} \sym^3 \sce \otimes \det \sce^\vee & \text{ if } d=3 \\
		\scf^\vee \otimes \sym^2 \sce  & \text{ if } d=4 \\ \wedge^2
\scf \otimes \sce \otimes \det \sce^\vee & \text{ if } d=5. \\ \end{cases}
\end{align}
\end{notation}

We will often use $\sch$ to denote $\sch(\sce, \scf_\bullet)$ when the data
$(\sce,\scf_\bullet)$ is clear from context.
We will see that sections of the above sheaf $\sch$ define subschemes of $\bp
\sce$. 
When these subschemes have dimension $0$ in fibers, we
will see they induce
degree $d$ locally free covers.
The parametrizations for degrees $3,4$, and $5$ essentially say that the
resulting covers are in bijection with such sections, up to automorphisms of
$(\sce, \scf_\bullet)$.

\subsection{The resolutions in low degree}

In order to state the parametrizations in degrees $3,4$, and $5$,
we now want a way of associating a subscheme of $\bp \sce$ to a section.
We will give a description of this association separately in the cases
that $d = 3,4$, and $5$.

Renaming the sheaf $\mathscr E'$ appearing in \eqref{equation:minimal-resolution}
as $\mathscr E$ and renaming $\mathscr F_1$ as $\mathscr F$,
in the cases $d =3,4$, and $5$,
\eqref{equation:minimal-resolution} becomes respectively
\begin{equation}
\label{equation:degree-3-resolution} \begin{tikzcd} 0 \ar {r} &  \pi^* \det
	\sce(-3) \ar {r}{\sigma} & \sco_{\bp} \ar {r} & \sco_X \ar {r} & 0,
\end{tikzcd}\end{equation}
\begin{equation}
\label{equation:degree-4-resolution} 
\begin{tikzcd} 0 \ar {r} &  \pi^* \det
\sce(-4) \ar {r}{\sigma} & \pi^*\scf(-2) \ar{r} &  \sco_{\bp} \ar {r} & \sco_X
\ar {r} & 0,  \end{tikzcd}\end{equation}
\begin{equation}
	\label{equation:degree-5-resolution} \begin{tikzcd} 0 \ar {r} & 
		\pi^* \det \sce(-5) \ar {r} & \pi^*\scf^\vee\otimes \pi^* \det
		\sce(-3) \ar{r}{\sigma} & \qquad\\
		\ar{r}{\sigma} & \pi^*\scf(-2) \ar{r} &  \sco_{\bp} \ar
		{r} & \sco_X \ar {r} & 0,
\end{tikzcd}\end{equation} 
with the rank of the locally free sheaves $\sce$ and $\scf$ in the degree $3$, $4$, and $5$
cases given in \autoref{notation:h-sheaf}.

\subsection{The maps $\Phi_d$ in low degree}

In the above $3$ cases, corresponding to degrees $3,4,$ and $5$ respectively, we
have isomorphisms
\begin{align}
	\label{equation:phi-3}
	\Phi_3: H^0(Y, \sym^3 \sce \otimes \det \sce^{\vee}) \overset
\sim \longrightarrow  H^0(\bp\sce, \pi^* \det \sce^{\vee}(3)).  \end{align}
\begin{align}
	\label{equation:phi-4}
	\Phi_4: H^0(Y, \sym^2 \sce \otimes \scf^\vee) \overset
  \sim \longrightarrow 
H^0(\bp\sce, \pi^* \scf^\vee(2))	\end{align} 
\begin{align}
	\label{equation:phi-5}
	\Phi_5: H^0(Y, \wedge^2 \scf \otimes \sce \otimes \det
\sce^{\vee}) \overset \sim \longrightarrow  H^0(\bp\sce, \wedge^2  \pi^* \scf \otimes \pi^* \det
\sce^{\vee}(1)).  \end{align}

\subsection{The maps $\Psi_d$ in low degree}
\label{subsection:psid}

For $\rho: X \to Y$ a finite locally free surjective Gorenstein cover of degree
$d$, we will use $\mathscr E^X$ to denote the Tschirnhausen bundle $\coker(
\mathscr O_Y \to \rho_* \mathscr O_X)^\vee$ and $\mathscr F^X$ to denote the
bundle $\mathscr F_1$ in the case we take $\mathscr E'$ in 
\autoref{theorem:ce-main-generalized}(iii)
to be the Tschirnhausen bundle $\mathscr E^X$.

Next, for $3 \leq d \leq 5$, given a section $\eta \in H^0(Y,
\sch(\sce,\scf_\bullet))$, we define an associated scheme $\Psi_d(\eta)$
over $Y$.

When $d = 3$,
we begin with a section 
$\eta \in H^0(Y, \sym^3 \sce \otimes \det \sce^{\vee})$, which, via $\Phi_3$ 
can be viewed as 
an element of $H^0(\bp\sce, \pi^* \det \sce^{\vee}(3))$.
Such a section corresponds
to a map $\sco_{\bp\sce} \ra \pi^* \det \sce^{\vee}(3)$, or equivalently a map
$\pi^* \det \sce(-3) \ra \sco_{\bp\sce}$.  
We let $\Psi_3(\eta)$ denote the support of the cokernel of this map.
That is, we define $\Psi_3(\eta) \subset \bp\sce$ so that on $\bp \sce$ we have an exact sequence
\begin{equation}
	\label{equation:section-3-sequence} \begin{tikzcd} \pi^* \det \sce(-3)
		\ar {r} & \sco_{\bp\sce} \ar {r} & \sco_{\Psi_3(\eta)} \ar {r} & 0.
\end{tikzcd}\end{equation} 

When $d =4$,
given $\eta \in H^0(Y, \scf^\vee
\otimes \sym^2 \sce)$, define 
$\Psi_4(\eta)$ to be the subscheme of $\bp\sce$, 
considered as the support of the cokernel of the map 
$\pi^* \scf(-2) \ra \sco_{\bp (\sce)}$
corresponding to $\Psi_4(\eta)$.

Finally, when $d = 5$,
given $\eta \in H^0(Y, \wedge^2\scf \otimes \pi^*\det \sce^{\vee} \otimes
\sce)$,
from $\Phi_5(\eta)$ we obtain a corresponding alternating map
$\pi^* \scf^\vee \otimes \pi^* \det \sce(-3) \to \pi^* \scf(-2)$.
The five $4 \times 4$ Pfaffians of this map determine a map of sheaves
$\pi^* \scf(-2) \to \sco_{\bp (\sce)}$,
as may be computed locally.
Define $\Psi_5(\eta)$ as the support of the cokernel of the map
$\pi^* \scf(-2) \to \sco_{\bp (\sce)}$
in $\bp\sce$.

%

\begin{definition}
	\label{definition:right-codim}
	Let $d \in \{3,4,5\}$, $Y$ be a scheme, and $(\sce, \scf_\bullet), \sch(\sce,
	\scf_\bullet)$ be sheaves on $Y$ as in \autoref{notation:h-sheaf}.
	We say $\eta \in H^0(Y, \sch(\sce, \scf_\bullet))$ has the {\em right
	Hilbert polynomial} at a point $y \in Y$ if the fiber of $\Psi_d(\eta)$ over $y$
	has dimension $0$ and degree $d$. We say $\eta$ has the {\em right
	Hilbert polynomial} if it has the right Hilbert polynomial at every $y \in Y$.
\end{definition}

Finally, we are ready to state the low degree parametrizations.
The parametrization in degree $3$ is as follows.

\begin{theorem}[Generalization of \protect{\cite[Thm.\
	3.4]{casnatiE:covers-algebraic-varieties}}, Specialization of {\cite[Prop. \ 5.1]{Poonen2008}}] \label{theorem:3-structure}
	Fix a scheme $Y$ and a rank $2$ locally free sheaf $\sce$ on $Y$. 
	The map $\eta \mapsto \Psi_3(\eta)$ induces
	a bijection between 
	\begin{enumerate}
		\item sections $\eta \in H^0(Y, \sym^3 \sce
	\otimes \det \sce^{\vee})$ having the right Hilbert polynomial at every $y \in
	Y$, up to automorphisms of $\sce$, 
		\item and finite locally free Gorenstein covers $\rho: X
			\ra Y$ of degree $3$ such that $\sce^\vee \simeq \coker \rho^\#$.
	\end{enumerate}
\end{theorem} The following proof extends that given in
\cite[Thm.\  3.4]{casnatiE:covers-algebraic-varieties}.  We note that there the
base is assumed to be reduced and noetherian, and the bijection is not
explicitly stated.  We outline the proof for the reader's convenience.
\begin{proof} 
	We start by constructing the map from $(2)$ to $(1)$.
	Given such a $\rho: X \ra Y$, we obtain from
	\autoref{theorem:ce-main-generalized}, a resolution of
	$\sco_{\bp\sce}$ as in \eqref{equation:degree-3-resolution}, unique up
	to unique isomorphism.  The map $\sigma$ in
	\eqref{equation:degree-3-resolution} 
	can be viewed as a section $\sigma \in H^0(\bp \sce, \pi^* \det
	\sce^\vee(3))$.
	For $\Phi_3$ as defined in \eqref{equation:phi-3}, we obtain a section $\rboxed{\eta }:=
	\Phi_3^{-1}(\sigma) \in H^0(Y, \sym^3 \sce \otimes \det \sce^{\vee})$.
	Note that the resulting $\eta$ has the right Hilbert polynomial at every $y \in
	Y$ because $X \ra Y$ is finite by assumption.

	We next show the map $\eta \mapsto \Psi_3(\eta)$ indeed defines a map
	from $(1)$ to $(2)$.
	Given $\eta$ of the right Hilbert polynomial at every $y \in Y$,
	we obtain a right exact sequence \eqref{equation:section-3-sequence}.
	The assumption that $\eta$ has the right Hilbert polynomial yields that
	the first map in this sequence is injective, and hence $X \to \bp \sce$
	has a resolution of the form \eqref{equation:degree-3-resolution}.
	This resolution shows $X$ is locally finitely presented over $Y$.  Further,
	$X$ is finite as it is locally of finite presentation, proper, and
	quasi-finite \cite[8.11.1]{EGAIV.3}.  Flatness of $X \to Y$ may be
	verified locally, in which case it holds as $X$ is cut out of $\bp^1_Y$
	by a single equation of degree $3$ not vanishing on any fibers.
	Therefore, $X$ is a finite locally free degree $3$ cover of $Y$.
	Finally, exactness of \eqref{equation:degree-3-resolution} implies
	$\sce^\vee \simeq \coker \rho^\#$ from
	\autoref{theorem:ce-main-generalized}(iii) and (iv).

	It remains to see that these two maps we have defined establish a
	bijection. For this, we show the compositions of these maps in both
	orders are equivalent to the identity map. If we begin with a cover $\rho: X \to Y$,
	\eqref{equation:degree-3-resolution} defines a resolution of $X \to \bp
	\sce$ giving $X$ as the vanishing locus $\Psi_3(\eta) \subset \bp \sce$.
	To show the other composition is equivalent to the identity, begin with some $\eta \in
	H^0(Y, \sym^3 \sce \otimes \det \sce^{\vee})$, and let $X$ denote the
	associated cover $\Psi_3(\eta)$.  The Tschirnhausen bundle $\sce^X$ 
	as in \autoref{subsection:psid}
	associated to $X$
	from \autoref{theorem:ce-main-generalized} is then isomorphic to $\sce$
	using \autoref{theorem:ce-main-generalized}(iv), as we may view $\eta$
	as a map $\pi^* \det \sce(-3) \to \sco_{\bp \sce}$.  Upon choosing such
	an isomorphism $\sce \simeq \sce^X$, we obtain a section $\eta^X \in
	H^0(Y, \sym^3 \sce^X \otimes \det (\sce^X)^{\vee}) \simeq H^0(Y, \sym^3
	\sce \otimes \det \sce^{\vee})$.  Using
	\autoref{theorem:ce-main-generalized}(iv), there is an automorphism of
	$\bp \sce$ taking $\Psi_3(\eta)$ to $\Psi_3(\eta^X)$.  
	From \autoref{theorem:ce-main-generalized}(iv) and the fact that the
	leftmost term of the resolution \eqref{equation:degree-3-resolution}
	is $\pi^* \det \sce(-3)$, we find $\sce$ is isomorphic to $\ker(\rho_*
	\omega_{X/Y} \to \sco_Y)$.
	By
	\autoref{corollary:automorphism-of-e}, this automorphism of $\bp \sce$
	is induced by an automorphism of $\sce$.  Hence, after composing with
	the automorphism of $\sce$, we can assume $\eta$ and $\eta^X$ define
	isomorphic subschemes of $\bp \sce$, and so are related via
	multiplication by a global section $s^{-1} \in \sco_Y(Y)$.  By
	composing with an automorphism of $\sce$ multiplying by $s^{-1}$,
	$\eta$ and $\eta^X$ are identified.	
\end{proof} 

We next verify the parametrization in degree $4$.

\begin{theorem}[Generalization of \protect{\cite[Thm.\
	4.4]{casnatiE:covers-algebraic-varieties}}, Specialization of {\cite[Thm. \ 1.1]{Wood2011}}] \label{theorem:4-structure}
	Fix a scheme $Y$, a rank $3$ locally free sheaf $\sce$ on $Y$, and a
	rank $2$ locally free sheaf $\scf$ on $Y$ such that there exists an
	unspecified isomorphism $\det \sce \simeq \det \scf$. The map $\eta
	\mapsto \Psi_4(\eta)$ induces a bijection between
	\begin{enumerate}
\item  sections $\eta \in H^0(Y,
	\scf^\vee \otimes \sym^2 \sce)$ having the right Hilbert polynomial at every $y
	\in Y$,
up to automorphisms of $\sce$ and $\scf$,
\item and, finite locally free Gorenstein covers $\rho:X \ra Y$ of degree $4$
	with associated sheaves $\sce^X, \scf^X$ as in \autoref{subsection:psid}
	which are isomorphic to $\sce$ and $\scf$.
	\end{enumerate}
\end{theorem} \begin{proof}
	First we construct the map from $(2)$ to $(1)$.
	Beginning with a cover $X \to Y$, we obtain a resolution
		\eqref{equation:degree-4-resolution}, and, upon choosing
		isomorphisms $\sce^X \simeq \sce$ and $\scf^X \simeq \scf$, we
		obtain a section $\eta \in H^0(Y, \scf^\vee \otimes \sym^2
		\sce)$ having the right Hilbert polynomial at every $y \in Y$.

		To construct the map from $(1)$ to $(2)$, we must show
		$\Psi_4(\eta)$ satisfies the properties listed in $(2)$.
		We first verify
	$\Psi_4(\eta)$ is a finitely presented Gorenstein cover of $Y$.  On
	fibers, $\Psi_4(\eta)$ is described as a dimension $0$ intersection
	of two quadrics. Since $\eta$ has the right Hilbert polynomial at $y \in Y$, it
	has degree $4$ over $y$.  
	(We parenthetically note that by Bezout's theorem, having the right
		Hilbert polynomial is equivalent to having dimension $0$, which
		then matches with Casnati-Ekedahl's notion of having
	“the right codimension” from \cite[Definition
4.2]{casnatiE:covers-algebraic-varieties}.)
	Gorensteinness follows because
	$\Psi_4(\eta)$ is a local complete intersection.  

	We next deduce flatness of $\Psi_4(\eta)$ over $Y$.
	We first explain how to reduce to the case that
	$Y$ is smooth. 
	Let $Z$ denote the moduli space parameterizing pairs of quadrics in
	$\bp^2$ which comes with a universal $\pi: U \to Z$ whose fiber over a pair
	$[(Q_1, Q_2)]$ is $Q_1 \cap Q_2$. There is an open locus $Z^\circ
	\subset Z$ where the intersection of these quadrics is $0$-dimensional,
	and hence has constant degree $4$ by Bezout's theorem.
	Let $U^\circ := \pi^{-1}(Z^\circ)$.
	Since $Z$ is a product of projective spaces, $Z^\circ$ is an open in a
	product of projective spaces, hence, in particular, smooth.
	Working fppf locally on $Y$, we can express $X \to Y$ as an intersection
	of relative quadrics in $\bp^2$, in which case
	$X \to Y$ is pulled back from $U^\circ \to Z^\circ$ via a map $Y \to
	Z^\circ$.
	Hence, it suffices to show that $U^\circ \to Z^\circ$ itself is flat.
	In this case, since $Z^\circ$ is reduced,
	flatness follows from constancy of the degree.
		
	To conclude the construction of the map from $(1)$ to $(2)$,
	we will show it is possible to choose identifications
	$\sce^X \simeq \sce, \scf^X \simeq \scf$ so that we obtain an associated
	section $\eta^X \in H^0(Y,\scf^\vee \otimes \sym^2 \sce) \simeq
	H^0(Y,(\scf^X)^\vee \otimes \sym^2 \sce^X)$. 

	First we show $\sce^X \simeq \sce$. 
%
	Indeed, there is a Koszul complex
	\begin{equation}
		\label{equation:koszul-complex}
		\begin{tikzcd}
			0 \ar {r} & \pi^* \det \scf \otimes \sco_{\bp \sce}(-4)
			\ar {r} & \pi^* \scf \otimes \sco_{\bp \sce}(-2) \ar {r} &
			\sco_{\bp \sce} \ar {r}
			& \sco_X.
	\end{tikzcd}\end{equation}
	It also follows from 
	\cite[Theorem 20.15]{Eisenbud:commutativeAlgebra}
	(using the comments on \cite[p. 503]{Eisenbud:commutativeAlgebra}
	and the fact that Gorenstein schemes are Cohen-Macaulay)
	that
	\eqref{equation:koszul-complex} yields a minimal free resolution of
	$X_y$ in $\bp \sce_y$ for every $y \in Y$.
	Because $\det \scf \simeq \det \sce$ by assumption,
	\autoref{theorem:ce-main-generalized}(iv) implies $\sce \simeq \sce^X$.

	Using the isomorphism $\sce \simeq \sce^X$, we also verify $\scf \simeq
	\scf^X$.
	Let $i: X \to \bp \sce$ and $i^X: X \to \bp \sce^X$ denote the two
	embeddings.
	By pushing forward the twist of \eqref{equation:koszul-complex} by
	$\sco_{\bp \sce}(2)$ along $\pi$, we find $\scf^X \simeq
	\ker\left(\sym^2 \sce^X
	\to \pi_* \left(i^X_* \sco_X \otimes \sco_{\bp \sce^X}(2) \right) \right)$.
	Similarly, 
	the analogous resolution from \autoref{theorem:ce-main-generalized}
	for $X$ in terms of $\sce^X$ and $\scf^X$ yields
	$\scf \simeq
	\ker\left(\sym^2 \sce
	\to \pi_* \left(i_* \sco_X \otimes \sco_{\bp \sce}(2) \right) \right)$.
	Hence, the isomorphism
	$\sce \simeq \sce^X$ induces the desired isomorphism $\scf \simeq
	\scf^X$.

	The isomorphism $\sce \simeq \sce^X$ is compatible with the above
	restriction map, and so induces an isomorphism $\scf \simeq \scf^X$.
	This concludes the verification that the map we have produced indeed
	goes from $(1)$ to $(2)$.

	It remains to prove the compositions of the above maps in both
	directions are equivalent to the identity. As in the degree
	$3$ case, if we start with a cover, and produce the associated section
	$\eta^X$, $\Psi_4(\eta^X)$ is isomorphic to $X$ via the construction.
	For showing the reverse composition is equivalent to the identity, start
	with some section $\eta$. Let $X$
	denote the resulting cover $\Psi_4(\eta)$. 
	
	Given the above identifications
	$\sce^X \simeq \sce, \scf^X \simeq \scf$,
	we wish to show $\eta^X$ is related to
	$\eta$ by automorphisms of $\sce$ and $\scf$.  Note also here that any
	automorphism of $\sce$ and $\scf$ sends $\eta$ to another section
	defining an isomorphic cover.  Using
	\autoref{theorem:ce-main-generalized}, there is an automorphism of
	$\bp \sce$ taking the subscheme $\Psi_4(\eta^X)$ to $\Psi_4(\eta)$.  
	From \autoref{theorem:ce-main-generalized}(iv) and the fact that the
	leftmost term of the resolution \eqref{equation:degree-4-resolution}
	is $\pi^* \det \sce(-4)$, we find $\sce$ is isomorphic to $\ker(\rho_*
	\omega_{X/Y} \to \sco_Y)$.
	By
	\autoref{corollary:automorphism-of-e}, the above automorphism of $\bp
	\sce$ is induced by an
	automorphism of $\sce$.  By composing with the inverse of this
	automorphism, we may assume the resulting map is the identity on $\bp
	\sce$, and so the automorphism of $\bp \sce$ is then induced by some
	automorphism of $\sce$ via multiplication by a section $s\in \sco_Y(Y)$. 
	After composing with multiplication by $s^{-1}$, we may reduce
	to the case $s$ is the identity.
	Since $\scf$ is a subsheaf of $\sym^2 \sce$ by
	\autoref{theorem:ce-main-generalized}(v), the image of the induced map
	$\scf \to \sym^2 \sce$ is uniquely determined by $X$, but the precise
	map is only determined up to automorphism of $\scf$.  Upon composing
	with such an automorphism, we may identify not just the images of $\scf$
	in $\sym^2 \sce$, but further we may identify the maps.  Under these
identifications, $\eta$ agrees with $\eta^X$, when viewed as maps $\scf \to
\sym^2 \sce$.  \end{proof}

We next state and prove the analogous parametrization in degree $5$.
As preparation, we will need the following application of the structure theorem
for codimension $3$ Gorenstein algebras due to Buchsbaum-Eisenbud.
\begin{lemma}
	\label{lemma:b-e-resolution}
	Let $Y$ be a scheme, and let $\sce$ and $\scf$ be locally free sheaves
	on $Y$ of ranks $3$ and $5$.
	A finite locally free Gorenstein cover $\rho: X \to Y$ of degree
	$5$, described as $\Psi_5(\eta)$ for
	$\eta\in H^0(Y, \wedge^2 \scf \otimes \sce \otimes \det \sce^{\vee})$, 
	has a resolution of the form
	\begin{equation}
		\label{equation:b-e-resolution}
		\begin{tikzcd}
			0 \ar {r} & \pi^* \det \sce^\vee \otimes \pi^*\det
			\scf(-5) \ar {r}{\beta_3} & \pi^* \det \sce^\vee \otimes \pi^* \scf^\vee(-3) \ar
			{r}{\beta_2} & \qquad \\
		\ar{r} {\beta_2} & \pi^* \scf(-2) \ar {r}{\beta_1}
			& \sco_{\bp \sce} \ar{r} & \sco_X,
	\end{tikzcd}\end{equation}
	which restricts to a minimal free resolution over each $y\in Y$,
	where $\beta_2$ is alternating and $\beta_3$ is identified with the dual
	of $\beta_1$ tensored with $\pi^* \det \sce^\vee \otimes \pi^*\det
	\scf(-5)$.
\end{lemma}
\begin{proof}
	In \eqref{equation:b-e-resolution}, the map $\beta_2$ is obtained from
	$\eta$, interpreted as a section of 
	$H^0(Y, \wedge^2 \scf \otimes \sce \otimes \det \sce^{\vee}) \simeq
	H^0(\bp \sce, \pi^*(\wedge^2 \scf \det \sce^{\vee})(1))$.
	The map $\beta_3$ is obtained by taking five $4 \times 4$ Pfaffians
	of $\beta_2$.
	To make sense of this, one may first construct $\beta_3$ locally upon choosing trivializations of $\scf$
	and $\sce$. 
	One then obtains a global map $\scf(-2) \to \sco_{\bp \sce}$
	because the formation of
	the Pfaffians are compatible with restriction to an open subscheme of
	$Y$.
	Finally, $\beta_1$ is obtained as the dual to $\beta_3$, tensored with 
	 $\pi^* \det \sce^\vee \otimes \pi^*\det
	\scf(-5)$.

	Since we have constructed the maps in \eqref{equation:b-e-resolution}
	globally over $\bp \sce$, it is enough to verify they furnish a
	minimal free
	resolution on geometric fibers.
	To this end, we may work locally on $Y$ and choose a trivialization
	$u: \det \sce \simeq \sco_Y$.
	Upon choosing this trivialization and composing with the isomorphism $u$ for
	the two left nonzero sheaves in \eqref{equation:b-e-resolution},
	we obtain a sequence
	\begin{equation}
		\label{equation:local-b-e-resolution}
		\begin{tikzcd}
			0 \ar {r} & \pi^*\det
			\scf(-5) \ar {r}{\beta_3'} & \pi^* \scf^\vee(-3) \ar
			{r}{\beta_2'} & \pi^* \scf(-2) \ar {r}{\beta_1'}
			& \sco_{\bp \sce} \ar{r} & \sco_X,
	\end{tikzcd}\end{equation}
	where $\beta_2'$ is still alternating, i.e., it corresponds to an
	element of
	$H^0(Y, \wedge^2 \scf \otimes \sce)$, 
	and $\beta_3'$ remains identified with the dual of $\beta_1'$, now
	tensored with 
	$\pi^*\det \scf(-5)$.
	Since the sequence \eqref{equation:local-b-e-resolution} commutes with
	base change on $Y$, we may further restrict to a geometric point $y \in
	Y$,
	and hence assume $Y$ is the spectrum of an algebraically closed field.

	We wish to show \eqref{equation:local-b-e-resolution} is a minimal
	locally free resolution.
	To do so, we wish to apply
	\cite{buchsbaumE:algebra-structures-for-finite-free-resolutions-codimension-3},
	and so we translate the above to the setting of commutative algebra.
	Writing $\bp \sce = \proj \kappa(y)[x_0, x_1, x_2, x_3]$,
	the cone over $X$ defines a subscheme of
	$\spec \kappa(y)[x_0, x_1, x_2, x_3]_{(x_0, x_1, x_2, x_3)}$,
	the localization of $\mathbb A^4_{\kappa(y)}$ at the origin.
	Taking $R := \kappa(y)[x_0, x_1, x_2, x_3]_{(x_0, x_1, x_2, x_3)}$,
	we can identify $\pi^* \scf$ with a rank $5$ free $R$-module $F$.
	Let $J$ denote the ideal of the cone over $X$ in $R$.
	The resolution \eqref{equation:local-b-e-resolution} can then be
	reexpressed in the form
	\begin{equation}
		\label{equation:ring-b-e-resolution}
		\begin{tikzcd}
			0 \ar {r} & 
			R \ar{r}{\beta_3''} &
			F^\vee \ar{r}{\beta_2''} &
			F      \ar{r}{\beta_1''} &
			R \ar{r} & R/JR,
	\end{tikzcd}\end{equation}
	with $\beta_2'' \in \wedge^2 F$ alternating and $\beta_3''$ the dual
	of $\beta_1''$.
	By definition of $\Psi_5(\eta)$ this sequence is exact at $R$,
	so $J$ is the image of $\beta_1''$.
	Since $X$ has codimension $3$ in $\bp \sce$ by assumption,
	$J$ is of grade $3$. 
	Hence,
	\eqref{equation:ring-b-e-resolution} satisfies
	the hypotheses of \cite[Theorem
	2.1(1)]{buchsbaumE:algebra-structures-for-finite-free-resolutions-codimension-3}.
	It is stated that any such resolution satisfying these hypotheses is a minimal free resolution
	of $R/JR$ 
	in the bottom paragraph of
	\cite[p. 463]{buchsbaumE:algebra-structures-for-finite-free-resolutions-codimension-3}
	and the proof is given in
	\cite[p. 464]{buchsbaumE:algebra-structures-for-finite-free-resolutions-codimension-3}.
\end{proof}

\begin{theorem}[Generalization of \protect{\cite[Thm.\
	3.8]{casnati:covers-algebraic-varieties-ii}}]
	\label{theorem:5-structure} Fix a scheme $Y$, a rank $4$ locally free
	sheaf $\sce$ on $Y$, and a rank $5$ locally free sheaf $\scf$ on $Y$
	such that there exists an unspecified 
	isomorphism $\det \scf \simeq (\det \sce)^{\otimes 2}$.
	The map $\eta \mapsto \Psi_5(\eta)$ induces a bijection between 
	\begin{enumerate}
\item 		sections $\eta \in H^0(Y, \wedge^2 \scf \otimes \sce \otimes
	\det \sce^{\vee})$ having the right Hilbert polynomial at every $y \in Y$,
up to automorphisms of $\sce$
	and $\scf$,
\item and, finite locally free Gorenstein covers $\rho:X
	\ra Y$ of degree $5$ with associated sheaves $\sce^X, \scf^X$ 
	as in \autoref{subsection:psid}	
	which
	are isomorphic to $\sce$ and $\scf$.
	\end{enumerate}
\end{theorem} 
\begin{proof}
	To start, we construct the map from $(2)$ to $(1)$.
	Beginning with a cover $X \to Y$, we obtain a resolution
	\eqref{equation:degree-5-resolution}.  Upon choosing isomorphisms
	$\sce^X \simeq \sce$ and $\scf^X \simeq \scf$ we obtain a section $\eta
	\in H^0(Y, \scf^{\otimes 2} \otimes \sce \otimes \det \sce^\vee)$ having
	the right Hilbert polynomial at every $y \in Y$.  We wish to check next that
	this section actually lies in $H^0(Y, \wedge^2 \scf \otimes \sce \otimes
	\det \sce^\vee)$.  
	Viewing this as a map $\pi^*\scf^\vee \otimes \pi^*
	\det \sce \to \pi^*\scf(1)$ via \eqref{equation:degree-5-resolution}, it
	is enough to verify the map is alternating locally on the base.
	Therefore, for this verification, we may assume $Y$ is the spectrum of a
	local ring and $\sce$ is trivial.  After this reduction, $X \subset \bp
	\sce$ is codimension $3$ and arithmetically Gorenstein, and so the
	Buchsbaum-Eisenbud parametrization for codimension $3$ Gorenstein
	schemes \cite[Thm.\
	2.1(2)]{buchsbaumE:algebra-structures-for-finite-free-resolutions-codimension-3}
	applies. This produces a resolution of $X \subset \bp \sce$ as in
	\eqref{equation:b-e-resolution}
	which by \autoref{theorem:ce-main-generalized} must agree with
	\eqref{equation:degree-5-resolution}. 
	Since the map corresponding to
	$\pi^*\scf^\vee \otimes \pi^* \det \sce \to \pi^* \scf(1)$ is
	alternating in the resolution of \cite[Thm.\
	2.1(2)]{buchsbaumE:algebra-structures-for-finite-free-resolutions-codimension-3}
	it follows that $\pi^* \scf^\vee \otimes \pi^*\det \sce \to \pi^*
	\scf(1)$ is also alternating.  

	We next construct the map from $(1)$ to $(2)$.
	This map will send $\eta$ to $\Psi_5(\eta)$. To show this is indeed a
	well defined map, we wish to verify
	$\Psi_5(\eta)$ is a finitely presented Gorenstein cover of $Y$.  The
	finite presentation condition follows from the resolution given in
	\eqref{equation:degree-5-resolution}.  We may check the remaining
	conditions locally on $Y$, and hence assume $Y$ is the spectrum of a
	local ring.  Observe that $X \to \bp \sce$ is arithmetically Gorenstein
	and of codimension $3$, using the assumption that $\eta$ has the right
	Hilbert polynomial at each $ y \in Y$ and
\cite[Thm.\
	2.1(1)]{buchsbaumE:algebra-structures-for-finite-free-resolutions-codimension-3}.  
	We find that $X$ is Gorenstein and is cut out scheme theoretically by
	the five $4 \times 4$ Pfaffians associated to $\eta$, thought of as a
	map $\pi^*\scf^\vee \otimes \pi^*\det \sce \to \pi^* \scf(1)$.  On
	fibers, $\Psi_5(\eta)$ is described as the vanishing of the five $4
	\times 4$ Pfaffians of an alternating linear map.  The resolution
	\cite[Thm.\
	2.1(1)]{buchsbaumE:algebra-structures-for-finite-free-resolutions-codimension-3},
	can be identified with one of the form
	\eqref{equation:degree-5-resolution}, from which one may calculate that
	the Hilbert polynomial of every fiber is $5$.  Therefore, the resulting
	scheme $\Phi_5(\eta)$ is finite and each fiber has degree $5$.

	We next deduce flatness of $\Psi_5(\eta)$ over $Y$.
	The idea is to reduce to the universal case, where we can verify
	flatness using constancy of Hilbert polynomial.
	Let $Z \simeq \mathbb A^{40}$ denote the affine space parameterizing
	alternating $5 \times 5$ matrices of linear forms in $\bp^3$. 
	Let $\pi: U \to Z$ denote the universal family of intersections of the
	five $4 \times 4$ Pfaffians of the corresponding matrix, so that the fiber of $\pi$ over a
	point $[M] \in Z$ is the intersection of the five $4 \times 4$ Pfaffians
	of the alternating matrix of linear forms $M$.
	There is an open subset $Z^\circ \subset Z$ parameterizing the locus
	where the fiber of $\pi$ is $0$-dimensional and degree at most $5$.
	One may verify that every fiber of $\pi$ has degree at least $5$, and so
	this open $Z^\circ$ parameterizes subschemes of degree exactly $5$.
	Let $U^\circ := \pi^{-1}(Z^\circ)$.
	Since $Z$ is smooth $Z^\circ$ is as well.
	Working fppf locally on $Y$, we can assume $X \to Y$ is a pullback of
	$U^\circ \to Z^\circ$ along a map $Y \to Z$.
	Hence, it suffices to show that $U^\circ \to Z^\circ$ itself is flat.
	In this case, since $Z$ is reduced,
	flatness follows from constancy of the degree.

	In order to show the above constructed map indeed takes $(1)$ to $(2)$,
	we must demonstrate
	identifications
	$\sce^X \simeq \sce$ and $\scf^X \simeq \scf$.
	To obtain the first identification, we use
	\autoref{lemma:b-e-resolution}.
	Since $\det \sce^{\otimes 2} \simeq \det \scf$, the leftmost nonzero
	term of the resolution in \autoref{lemma:b-e-resolution}
	becomes 
	$\det \pi^* \sce^\vee \otimes \pi^*\det \scf(-5) \simeq \pi^* \det \sce(-5)$.
	Hence, \autoref{theorem:ce-main-generalized}(iv) implies $\sce \simeq
	\sce^X$.
	Let $i : X \to \bp \sce, i^X : X \to \bp \sce^X$ denote the inclusions.
	By twisting \eqref{equation:b-e-resolution} by $\sco_{\bp \sce}(2)$ and
	pushing forward, we find $\scf \simeq \ker(\sym^2 \sce \to \pi_* (i_*
	\sco_X \otimes\sco_{\bp \sce}(2)))$.
	The analogous resolution from \autoref{theorem:ce-main-generalized}
	for $X$ in terms of $\sce^X$ and $\scf^X$ yields
	$\scf^X \simeq \ker(\sym^2 \sce^X \to \pi_*(i^X_* \sco_X \otimes
	\sco_{\bp \sce^X}(2)))$.
	Hence, the isomorphism $\sce \simeq \sce^X$ induces the desired
	identification $\scf
	\simeq \scf^X$.
	This completes the construction of the map from $(1)$ to $(2)$.

	It remains to prove the compositions of the above maps between $(1)$ and
	$(2)$ are equivalent to
	the identity.  As in the degree
	$3$ case, if we start with a cover, produce the associated section
	$\eta^X$, $\Psi_5(\eta^X)$ is isomorphic to $X$ via the construction.

	For the reverse composition, start with some section $\eta$ and let $X$
	denote the resulting cover $\Psi_5(\eta)$. 	
	Now, choose identifications
	$\sce^X \simeq \sce, \scf^X \simeq \scf$ as above
	so that we obtain an associated
	section $\eta^X \in H^0(Y,\wedge^2 \scf^\vee \otimes \det \sce \to \sce)
	\simeq H^0(Y, \wedge^2 (\scf^X)^\vee \otimes \det \sce^X \to \sce^X)$.
	We wish to show $\eta^X$ is related to $\eta$ by automorphisms of $\sce$ and
	$\scf$.  Note also here that any automorphism of $\sce$ and $\scf$ sends
	$\eta$ to another section defining an isomorphic cover.  
	Using
	\autoref{theorem:ce-main-generalized}, there is an automorphism of
	$\bp \sce$ taking $\Psi_5(\eta^X)$ to $\Psi_5(\eta)$.  
	From \autoref{theorem:ce-main-generalized}(iv) and the fact that the
	leftmost term of the resolution \eqref{equation:degree-4-resolution}
	is $\pi^* \det \sce(-5)$, we find $\sce$ is isomorphic to $\ker(\rho_*
	\omega_{X/Y} \to \sco_Y)$.
	By
	\autoref{corollary:automorphism-of-e}, this automorphism of $\bp \sce$ is induced by an
	automorphism of $\sce$.  By composing with the inverse of this
	automorphism, we may assume $\eta$ and $\eta^X$ define the same subscheme of $\bp
	\sce$. Hence we may assume the automorphism of $\bp \sce$ is then
	induced by multiplication by a section $s \in \sco_Y(Y)$.
	After composing with multiplication by $s^{-1}$, we may therefore reduce to
	the case that $s$ is the identity section.
	By \autoref{theorem:ce-main-generalized}(iii), we obtain a unique isomorphism
	between the two resolutions
	of $X$ in $\bp \sce$
	\eqref{equation:degree-5-resolution} determined by $\eta$ and
	$\eta^X$.
	This isomorphism can be specified as a tuple of $5$ maps between the
	nonzero terms
	of \eqref{equation:degree-5-resolution}.

	We next show we can apply an automorphism of $\scf$ so as to assume the
	map $\pi^* \scf(-2) \to \pi^* \scf(-2)$ is the identity.
	Since $\scf$ is a subsheaf of $\sym^2 \sce$ by
	\autoref{theorem:ce-main-generalized}(v), the image of the induced map
	$\scf \to \sym^2 \sce$ coming from the Pfaffians associated to $\eta$ is uniquely determined by $X$, 
	but the precise	map is only determined up to automorphism of $\scf$.  
	Upon composing with such an automorphism, 
	we may identify not just the images of $\scf$
	in $\sym^2 \sce$, but further we may identify the maps.  
	Under these
	identifications, $\eta$ agrees with $\eta^X$, when viewed as maps $\scf
	\to \sym^2 \sce$.

	So far, we have constructed a map of the two resolutions
	\eqref{equation:degree-5-resolution} associated to $\eta$ and $\eta_X$.
	Upon choosing identifications $\sce \simeq \sce^X$ and
	$\scf \simeq \scf^X$ as above, we have enforced that the map of
	resolutions is given by the identity on the terms
	$\sco_X \to \sco_X, \sco_{\bp} \to \sco_{\bp}$, and $\pi^* \scf(-2) \to
	\pi^* \scf(-2)$. 
	When we write the second nonzero term of
	\eqref{equation:degree-5-resolution} as $\pi^* \scf^\vee \otimes \pi^*
	\det \sce(-3)$, we have identified it via Grothendieck duality
	as pairing with the third nonzero term $\pi^* \scf(-3)$ into $\pi^*
	\sce(-5)$, and therefore the induced automorphism of
	$\pi^* \scf^\vee \otimes \pi^* \det \sce(-3)$ must respect this
	duality. In particular, since we have reduced to the case where the automorphism of $\pi^* \scf(-2)$
	is the identity, we also obtain the induced automorphism of 
	$\pi^* \scf^\vee \otimes \pi^* \det \sce(-3)$ is the identity.
	Using \autoref{theorem:ce-main-generalized}(v)
	to guarantee that the maps $\eta$ and $\eta^X$ from $\scf^\vee \otimes
	\det \sce \to \scf \otimes \sce$ are injective, we obtain the desired
	identification of $\eta$ with $\eta^X$.
\end{proof}

Finally, we recall a rather elementary criterion for when
$\Psi_d(\eta)$ is geometrically connected.
\begin{theorem}[Part of \protect{\cite[Thm.\
	3.6]{casnatiE:covers-algebraic-varieties}, \cite[Thm.\
	4.5]{casnatiE:covers-algebraic-varieties}, \cite[Thm.\
	4.4]{casnati:covers-algebraic-varieties-ii}}]
	\label{theorem:bertini-3-4-5}
	Keeping notation as in \autoref{notation:h-sheaf},
	assume that
	$Y$ is a 
	geometrically connected and geometrically reduced projective scheme over
	a field $k$.
	If $h^0(Y, \sce^\vee) =
0$, then $\Psi_d(\eta)$ is geometrically connected.  
\end{theorem}
\begin{proof} The proof is essentially given in \cite[Thm.\
	3.6]{casnatiE:covers-algebraic-varieties}, and we repeat it for the
	reader's convenience.  
	Let $X := \Psi_d(\eta)$.	
	If $h^0(Y, \sce^\vee) = 0$
	the exact sequence \begin{equation} \label{equation:} \begin{tikzcd} 0
	\ar {r} & \sco_Y \ar {r} & \rho_* \sco_X \ar {r} & \sce^\vee \ar {r} & 0
	\end{tikzcd}\end{equation} induces an isomorphism $H^0(Y, \sco_Y) \simeq
	H^0(Y, \rho_* \sco_X) = H^0(X, \sco_X)$.  Since $Y$ is geometrically
	connected and geometrically reduced, we have $h^0(Y, \sco_Y) = 1$.  From
	this we find $H^0(X, \sco_X) = 1$ as well, and therefore $X$ is
	necessarily geometrically connected.  \end{proof} 

\section{Describing  stacks of low-degree covers as quotients}
\label{section:presentation-finite-covers}

In this section, we give a description of the stack of degree $d$ Gorenstein
covers as a global quotient stack for
$3 \leq d \leq 5$. We now introduce the groups we will be quotienting by.  
Since the Hurwitz stack is closely related to the Weil restriction of
the stack of degree $d$ covers along $\bp^1 \to \spec k$, we will simultaneously
define these automorphism groups along Weil restrictions.

\begin{remark}
	\label{remark:intuitive-aut}
	We are about to define an automorphism sheaf $\aut_{\sce,
\scf_\bullet}^{Y/B}$ for $Y \to B$ a morphism of schemes and $\sce, \mathscr
F_\bullet$ locally free sheaves on $Y$.
Before giving the formal definition, we give an intuitive description.

Consider first the case that $d = 4,$ or $d = 5$, $Y = B = \spec k$,
and additionally assume there is an isomorphism 
$\det \sce^{\otimes d-3} \simeq \det \scf_1$.
Then, the points of 
$\aut_{\sce,\scf_\bullet}^{\spec k/\spec k}$
corresponds to automorphisms of $\sce$ and $\scf_1$ which
preserve the above isomorphism.
However, in what follows we do not require such an isomorphism 
$\det \sce^{\otimes d-3} \simeq \det \scf_1$
exists, and so the definition we give is
somewhat more general.
Namely, we instead work with automorphisms $(M, N) \in \aut_\sce \times
\aut_{\scf_1}$ so that $\det M^{d-3} = \det N$.

Another important case is that where $Y = D, B = \spec k$, and again $d= 4$ or
$5$. 
If additionally, there is an isomorphism
$\det (\sce^{\otimes d-3})_D \simeq (\det \scf_1)_D$,
$\aut_{\sce,\scf_\bullet}^{D/\spec k}$ can be thought of as parameterizing automorphisms
of $\sce$ and $\scf_1$ over $D$ which preserve the isomorphism
$\det (\sce^{\otimes d-3})_D \simeq (\det \scf_1)_D$.
Again, we have the caveat that this is only correct when such an isomorphism exists.
\end{remark}

\begin{definition} \label{definition:automorphism} Given a scheme $Y$ over a
	base $B$ and an integer $d$, let {\em resolution data} for $Y$ and $d$
	denote a tuple of locally free sheaves $(\sce, \scf_\bullet)$ on $Y$,
	where $\sce$ is a locally free sheaf of rank $d-1$ and
	$\rboxed{\scf_\bullet}$ denotes the sequence $\scf_1, \ldots,
	\scf_{\lfloor \frac{d-2}{2} \rfloor}$ where $\rk \scf_i = \beta_i$ as in \eqref{equation:beta}.
		Let $3 \leq d \leq 5$, fix a scheme $Y$ over a field, and fix resolution
	data $(\sce, \scf_\bullet)$ for a degree $d$ cover of $Y$.  For $\scg$ a
	locally free sheaf on $Y$, let $\rboxed{\Delta_\scg^{Y/B}} := \bg_m \to
	\res_{Y/B}(\aut_{\scg/Y})$ denote the map adjoint to the central
	inclusion $(\bg_m \times_B Y)  \to \aut_{\scg/Y}$ on $Y$.
	We denote by $\left(\Delta_\scg^{Y/B} \right)^i$ the composition $\bg_m \xrightarrow{x
	\mapsto x^i} \bg_m \xrightarrow{\Delta_\scg^{Y/B}} \res_{Y/B}(\aut_{\scg/Y})$
	and define
	\begin{align*}
		\left( \left(\Delta_\sce^{Y/B}\right)^i,
	\left(\Delta_{\scf_1}^{Y/B}\right)^{j}\right) : \bg_m & \rightarrow \res_{Y/B}(\aut_{\sce/Y})
	\times
\res_{Y/B}(\aut_{\scf_1/Y})\\
x & \mapsto \left(\Delta_{\sce}^{Y/B}(x^i), \Delta_{\sce}^{Y/B}(x^j) \right).
	\end{align*}
	Finally, we use 
	\begin{align*}
	\coker\left(\left(\Delta_\sce^{Y/B}\right)^i,
\left(\Delta_{\scf_1}^{Y/B}\right)^{j} \right) := \frac{\left( \res_{Y/B}(\aut_{\sce/Y})
	\times
\res_{Y/B}(\aut_{\scf_1/Y}) \right) }{ \left(\left(\Delta_\sce^{Y/B}\right)^i,
\left(\Delta_{\scf_1}^{Y/B}\right)^{j}\right)(\bg_m)}.
	\end{align*}

	Then, define the {\em automorphism sheaf} of this resolution data to be
	the $B$-scheme \begin{align} \label{equation:aut-groups}
		\rboxed{\aut_{\sce, \scf_\bullet}^{Y/B}} := \begin{cases}
			\res_{Y/B}(\aut_{\sce/Y}) & \text{ if } d = 3 \\
			\coker(\Delta_\sce^{Y/B}, (\Delta_{\scf_1}^{Y/B})^{2}) &
			\text{ if } d= 4 \\ 
			\coker((\Delta_\sce^{Y/B})^2,
	(\Delta_{\scf_1}^{Y/B})^3) & \text{ if } d= 5.  \end{cases} 
\end{align} In
	the case $Y = B$, we notate $\aut_{\sce, \scf_\bullet}^{B/B}$ simply by
	$\aut_{\sce, \scf_\bullet}$.
	When $d =4$ or $5$, we will often denote $\scf_1$ by $\scf$.
\end{definition}
Throughout much of the remainder of the paper, we will typically work over the
base $B = \spec k$ for $k$ a field. 
There are notable exceptions, such as
\autoref{proposition:ce-as-quotient-stack}, where we take $B = \spec \bz$.

\begin{remark}
	\label{remark:sheaf-ranks}
	Concretely, $\sce$ and $\scf_\bullet$ in
	\autoref{definition:automorphism}
	are (sequences of) sheaves of the
	following ranks.
	For $d = 3$, $\sce$ is locally free of rank $2$ and $\scf_\bullet$ is trivial
	(i.e., the sequence of sheaves has length $0$).
	When $d = 4$, $\sce$ is locally free of rank $3$ and
$\scf_\bullet = \scf$ is locally free of rank $2$. 
When $d = 5$, $\sce$ is locally free of
rank $4$ and $\scf_\bullet = \scf$ is locally free of rank $5$.  
\end{remark}

In order be able to calculate the class of quotients by the groups of
\autoref{definition:automorphism} in the
Grothendieck ring, it will be useful to know these groups are often special.
The following description of these
quotients will allow us later, in \autoref{lemma:special-auts}, to easily
deduce these groups are special.

\begin{lemma}
	\label{lemma:aut-as-sub} Maintaining the notation of
	\autoref{definition:automorphism}, we have an isomorphism of functors
	\begin{align}
		\label{equation:aut-subs} \aut_{\sce, \scf_\bullet}^{Y/B} \simeq
		\begin{cases} \ker(\det,\det^{-1}): \res_{Y/B}(\aut_{\sce/Y})
			\times \res_{Y/B}(\aut_{\scf/Y}) \ra \res_{Y/B} (\bg_m) & \text{ if }
			d= 4 \\ \ker(\det^2,\det^{-1}):
			\res_{Y/B}(\aut_{\sce/Y}) \times
			\res_{Y/B}(\aut_{\scf/Y}) \ra \res_{Y/B}(\bg_m) & \text{ if
} d= 5 \\ \end{cases} \end{align} Here, by determinant we mean the map adjoint
to the corresponding determinant map on $Y$.  \end{lemma} 
\begin{proof} We
	produce the claimed isomorphisms by constructing a section to the quotient map $q :
	\res_{Y/B}(\aut_{\sce/Y}) \times \res_{Y/B}(\aut_{\scf/Y}) \to
	\aut_{\sce, \scf_\bullet}^{Y/B}$ defining $\aut_{\sce, \scf_\bullet}^{Y/B}$.

To start, we cover the case $d = 4$.  Given $(M, N) \in
\res_{Y/B}(\aut_{\sce/Y}) \times \res_{Y/B}(\aut_{\scf/Y})$, for $\lambda \in
\bg_m$, we can identify $q(M,N) = q(\lambda M, \lambda^2 N)$.  For any such
$(M,N)$ the key observation is that there is a unique $\lambda \in \bg_m$ such
that $\det (\lambda M) = \det( \lambda^2 N)$.  Indeed, $\det (\lambda M) =
\lambda^3 \det M$ while $\det( \lambda^2 N) = \lambda^4 \det N$ and so the
unique such $\lambda$ is $\lambda = \det M/\det N$.  This gives the desired
splitting realizing $\aut_{\sce, \scf_\bullet}^{Y/B}$ as a subgroup of
$\res_{Y/B}(\aut_{\sce/Y}) \times \res_{Y/B}(\aut_{\scf/Y})$ because the
composition $\ker(\det, \det^{-1}) \to \res_{Y/B}(\aut_{\sce/Y}) \times
\res_{Y/B}(\aut_{\scf/Y}) \to \aut_{\sce, \scf_\bullet}$ is an isomorphism.

The $d = 5$ case is quite similar to the $d = 4$ case.  Namely, in this case,
for $(M, N) \in \res_{Y/B}(\aut_{\sce/Y}) \times \res_{Y/B}(\aut_{\scf/Y})$,
there is again a unique $\lambda \in \bg_m$ so that $(\det (\lambda^2 M))^2 =
\det(\lambda^3 N)$.  Indeed, $(\det (\lambda^2 M))^2 = \lambda^{16} \det M^2$
and $\det(\lambda^3 N) = \lambda^{15} \det N$, so the unique desired $\lambda$
is $\det N/(\det M)^2$.  As in the $d = 4$ case, this provides a section to the
given quotient map realizing $\aut_{\sce, \scf_\bullet}^{Y/B}$ as the subgroup
of $\res_{Y/B}(\aut_{\sce/Y}) \times \res_{Y/B}(\aut_{\scf/Y})$ given as those
$(M,N)$ with $(\det M)^2 = \det N$.  \end{proof}

We next describe a presentation of the stack parameterizing degree $d$
Gorenstein covers for $3 \leq d \leq 5$.  To make our next definition, we will need to know the Gorenstein locus of a
finite locally free map is open.

\begin{lemma}
	\label{lemma:gorenstein-open}
	Let $f: X \to Y$ be a finite locally free morphism of schemes.
	The locus of points of $Y$ on which the fiber of $f$ is Gorenstein is an
	open subscheme of $Y$.
\end{lemma}
\begin{proof}
	First, by 
\cite[\href{https://stacks.math.columbia.edu/tag/00RH}{Tag
00RH}]{stacks-project},
the condition that the fiber be Cohen-Macaulay is an open condition.
After restricting to such an open subscheme,
by
\cite[Thm.\ 3.5.1]{conrad:grothendieck-duality},
a dualizing sheaf exists, and the Gorenstein locus is then the locus where this
dualizing sheaf is locally free, which again defines an open subscheme.
\end{proof}

We are now ready to define the relevant Gorenstein loci.
With notation as in
\autoref{definition:automorphism}, we work over $B = \spec \bz$.  
\begin{definition}
	\label{definition:gorenstein-sections}
	For each $3 \leq d \leq 5$, fix free sheaves on $Y = B = \spec \bz$, $\sce$ and
$\scf_\bullet$ as in \autoref{definition:automorphism} and
\autoref{remark:sheaf-ranks}.
Let $\rboxed{
\gorensteinsections d } \subset \spec \left( \sym^\bullet H^0(\spec \bz, \sch(\sce,
\scf_\bullet))^\vee \right)$ denote the open subscheme 
(using \autoref{lemma:gorenstein-open})
functorially parameterizing those
sections $\eta$ so that $\Psi_d(\eta)$ defines a degree $d$ locally free
Gorenstein cover, for $\Psi_d$ the maps (depending on $3 \leq d \leq 5$) defined
in \autoref{subsection:low-degree-ce-structure-theorems}.  
\end{definition}

In what follows, we use $\coverstack d$ to
denote the fibered category whose $S$ points are finite locally free covers $X
\to S$ of degree $d$ with Gorenstein fibers.

\begin{definition}
	\label{definition:sections-to-coverstack}
	For $3 \leq d \leq 5$, the map $\Psi_d$ over $B = \spec \bz$ induces a map $\mu_d:
	\gorensteinsections d \to \coverstack d$, with $\coverstack d$
	as defined above.
There is a natural
action of $\aut_{\sce, \scf_\bullet}$ on $\gorensteinsections d$, induced by the
action of $\aut_{\sce} \times \aut_{\scf_\bullet}$ on $\gorensteinsections d$.
The map $\mu_d$ is invariant under this action, since the resulting
abstract degree $d$ cover is unchanged by such re-coordinatizations. We now
define the induced map from the quotient stack $\phi_d : [\gorensteinsections
d/\aut_{\sce, \scf_\bullet}] \to \coverstack d$.
\end{definition}

\begin{proposition} \label{proposition:ce-as-quotient-stack} For $3 \leq d \leq
5$, the map $\phi_d$ defined in \autoref{definition:sections-to-coverstack} over
$B = \spec \bz$ is an isomorphism.  \end{proposition}

When $d=3,4$, Proposition~\ref{proposition:ce-as-quotient-stack} is the specialization of the isomorphisms of moduli stacks given in \cite[Prop.\ 5.1]{Poonen2008} and \cite[Thm.\ 1.1]{Wood2011}
 to Gorenstein covers.

\begin{proof}
	We will construct an inverse map using
	\autoref{theorem:ce-main-generalized}.  Using
	\autoref{theorem:ce-main-generalized}, there is an $\aut_{\sce} \times
	\aut_{\scf_\bullet}$ torsor $\auttorsor d$ over $\coverstack d$
	whose $S$-points parameterize covers $X \to S$ together with specified
	trivializations $\sce^X \simeq \sce, \scf_\bullet^X \simeq \scf_\bullet$
	of the sheaves $\sce^X$ and $\scf^X_\bullet$ associated to $X$ coming
	from \autoref{theorem:ce-main-generalized}.  
	Note here that $\auttorsor d$ maps surjectively to $\coverstack d$ because
	for any $S$ point, there is an open cover of $S$ on which these vector
	bundles become isomorphic to trivial bundles.
	The parametrizations
	\autoref{theorem:3-structure}, \autoref{theorem:4-structure}, and
	\autoref{theorem:5-structure} then give a section $\eta \in \sch(\sce,
\scf_\bullet)$.  This induces a map $\auttorsor d \to \gorensteinsections d$.

	We wish to show this induced map 
	$\auttorsor d \to \gorensteinsections d$
	is an isomorphism in degree $3$ and a
	$\bg_m$ torsor in degrees $4$ and $5$, where $\bg_m$ is the copy of
	$\bg_m \subset \aut_{\sce} \times \aut_{\scf}$ 
	as in \autoref{definition:automorphism}	
	whose quotient yields
	$\aut_{\sce, \scf_\bullet}$.  Once we verify this, the parametrizations
	\autoref{theorem:3-structure}, \autoref{theorem:4-structure}, and
	\autoref{theorem:5-structure} imply that the composition 
	$\auttorsor d \to \gorensteinsections d \to [\gorensteinsections
d/\aut_{\sce, \scf_\bullet}] \xrightarrow{\phi_d} \coverstack d$
is the structure map for the torsor $\auttorsor d \to
	\coverstack d$.  
	From this, it follows that the resulting
	isomorphism
	$[\auttorsor d/
	\aut_\sce \times \aut_{\scf_\bullet}] \to \coverstack d$ 
	factors through an isomorphism
	$[\auttorsor d/	\aut_\sce \times \aut_{\scf_\bullet}] \to
	[\gorensteinsections d/\aut_{\sce, \scf_\bullet}]$, and hence
	$\phi_d: [\gorensteinsections d/\aut_{\sce, \scf_\bullet}] \to
	\coverstack d$
	is an
	isomorphism.

	First, we verify the map $\auttorsor d \to \gorensteinsections d$ is
	invariant under the above mentioned $\bg_m$ action in the cases that $d
	=4$ and $5$.  In the degree $4$ case, scaling $\sce$ by $\lambda$ and
	$\scf$ by $\lambda^2$ scales $\scf^\vee \otimes \sym^2 \sce$ by
	$\lambda^{-2} \cdot \lambda^2 = 1$.  In the degree $5$ case, scaling
	$\sce$ by $\lambda^2$ and $\scf$ by $\lambda^3$ scales $\wedge^2 \scf
	\otimes \sce \otimes \det \sce^\vee$ by $\lambda^6 \cdot \lambda^2 \cdot
	\lambda^{2 \cdot {(-4)}} = 1$.

	Therefore, to conclude the verification, it is enough to show the only
	elements of $\aut_{\sce} \times \aut_{\scf_\bullet}$ fixing a given
	section are trivial when $d = 3$ and lie in $\bg_m$ when $d = 4$ or $5$.
	To start, the map $X \to \bp \sce$ realizes $X$ as a nondegenerate
	subscheme of $\bp \sce$, and therefore only the trivial element of
	$\pgl_{\sce}$ fixes $X$ as a subscheme of $\bp \sce$.  In the degree $3$
	case, scaling by $\lambda$ in the central $\bg_m \subset \aut_{\sce}$
	scales the resulting section by $\lambda$, and so only the identity
	element of $\aut_{\sce}$ preserves the section.  This establishes the
	claim when $d = 3$.

	We now consider the cases $d= 4$ and $d = 5$.  We are seeking
	automorphisms of $\sce$ and $\scf$ preserving a given section
	$\eta \in \gorensteinsections d$.
	We have seen above that any such automorphism must act on $\sce$ by some element $\lambda$ in the
	central $\bg_m \subset \aut_{\sce}$.  Since we are quotienting by a copy
	of $\bg_m \subset \aut_{\sce} \times \aut_{\scf}$ which maps
	surjectively to the central $\bg_m$ in $\aut_{\sce}$, we may modify our
	given automorphism so as to assume it is trivial in $\aut_{\sce}$. Note
	that when $d = 5$, we may have to pass to an fppf cover so as to extract
	a square root of $\lambda$.  We may now assume the automorphism is
	trivial on $\sce$ and wish to show it is also trivial on $\scf$.
	However, the given section $\eta$ induces an injective map $\scf \to
	\sym^2 \sce$, realizing $\scf$ as a subsheaf of $\sym^2 \sce$ by
	\autoref{theorem:ce-main-generalized}(v).  Since we are assuming the
	automorphism acts as the identity on $\sce$ and it preserves this
	inclusion, it must also act as the identity on $\scf$.  
\end{proof}

\section{Defining our Hurwitz stacks}
\label{section:hurwitz-definition}

In this section, we construct and define the Hurwitz spaces we will be working
with.
We will ultimately be interested in the Hurwitz space whose geometric points
parameterize degree $d$ $S_d$ covers of $\bp^1$ which are smooth and connected.
When one restricts to simply branched covers, such a Hurwitz scheme was constructed by Fulton
\cite{Fulton1969}.  
Another good reference is 
\cite[Theorem A]{deopurkar:compactifications-of-hurwitz-spaces}, though this
reference assumes characteristic $0$.
Another excellent reference is \cite[Thm.
6.6.6]{bertinR:champs-de-hurwitz}, which constructs the Hurwitz stacks in the
case that the cover is not Galois, but has a fixed Galois closure $G$, which is
invertible on the base. Although we are ultimately primarily interested in
counting $S_d$ covers, we will do so by realizing them as a certain proportion
of the space of all degree $d$ covers, so this reference again does not quite
suffice for our purposes.
We were unable to find a reference that allows arbitrary branching and
non-Galois covers in arbitrary characteristic, and so 
we give the construction here.  
To begin, we define a certain Hurwitz stack parameterizing covers of $\bp^1$
which are not necessarily $S_d$ covers.

\begin{definition} \label{definition:} For $S$ a base scheme, and $d \geq 0$ an
	integer, let
	$\rboxed{\bighur d S}$ denote the category fibered in groups over
	$S$-schemes whose $T$ points over a given map of schemes $T \to S$
	consists of $(T, X, h: X \to T, f: X \to \bp^1_T)$ \begin{equation}
	\label{equation:} \begin{tikzcd} X \ar {rr}{f} \ar {rd}{h} && \bp^1_T
\ar {ld} \\ & T & \end{tikzcd} \end{equation} where $X$ is a scheme, $f$ is a
finite locally free map of degree $d$ and $h$ is a smooth proper relative curve.
A map $(T, X, h, f) \to (T, X', h', f')$ consists of
a $T$-isomorphism $\alpha: X \to X'$ such that
\begin{equation} \label{equation:} \begin{tikzcd} X \ar {rr}{\alpha} \ar {rd}{f}
		\ar[bend right = 0cm, swap]{rdd}{h} && X' \ar[swap]{ld}{f'} \ar[bend
		left = 0 cm]{ldd}{h'}\\ & \bp^1_T \ar{d} &\\ & T & \end{tikzcd}\end{equation} commutes.  For $g \geq 0$
an integer, let
$\rboxed{\bighurg d g S}$ denote the substack parameterizing those $T$-points of
$\bighur d S$ such that $X \to T$ has arithmetic genus $g$.  \end{definition}
	
\begin{lemma} \label{lemma:big-hur-algebraic} For $S$ a scheme, $\bighur d S$ and $\bighurg d g
S$ are algebraic stacks.  \end{lemma} \begin{proof} First, we show $\bighur d S$
	is an algebraic stack.  It is enough to establish this in the universal
	case $S = \spec \bz$.  Observe that $\bighur d \bz$ is a stack because
	descent for finite degree $d$ locally free morphisms is effective.  To
	see it is algebraic, we construct it as a hom stack.  Let
	$\rboxed{\mathfrak A_d}$ denote the stack parameterizing finite locally
	free degree $d$ covers, as constructed in \cite[Def.\ 3.2]{Poonen2008}.

	Next, we claim the mapping stack $\hom(\bp^1, \mathfrak A_d)$ is
	algebraic.  This would follow from \cite[Thm.\  1.1]{aoki:hom-stacks},
	except the theorem there is not stated correctly, as mentioned in the
	erratum \cite{aoki:erratum-hom-stacks}.  This erratum asserts that we
	only need verify the additional condition that for any complete local
	noetherian ring $A$ with maximal ideal $\fm$ and $\rboxed{A_n} :=
	A/\fm^n$, a collection of compatible maps $\hom(\bp^1_{A_n}, (\mathfrak
	A_d)_{A_n})$ for each $n$ lifts to a map $\hom(\bp^1_{A}, (\mathfrak
	A_d)_{A})$.  In our setting, this condition is indeed satisfied because
	specifying such maps over $A_n$ corresponds to specifying degree $d$
	locally free covers $X_n \to \bp^1_{A_n}$ over $A_n$ for each $n$. 
	Then, by Grothendieck's algebraization theorem \cite[Thm.\
	8.4.10]{FantechiGIK:fundamentalAlgebraicGeometry} such a family
	algebraizes to a family $X \to \bp^1_A$ over $\spec A$, using the
	pullback of $\sco_{\bp^1}(1)$ to $X$ as the relevant ample line bundle
	on $X$.

	The stack $\bighur d S$ is then the open substack of the mapping stack
	$\hom(\bp^1, \mathfrak A_d)$ corresponding to those finite locally free
	covers $X \to \bp^1$ which are smooth over the base.
%
	
	Finally, $\bighurg d g S$ is an open and closed substack of $\bighur d
S$ because the genus is locally constant in flat families.  \end{proof}

Having constructed the Hurwitz stack parameterizing all degree $d$ covers of
$\bp^1$, we next construct an open substack parameterizing $S_d$ covers,
over geometric fibers.
For the following definition, recall that $B_n$, the $n$th Bell number, is
the number of ways to partition a set of $n$ elements into subsets.
So, for example, $B_1 = 1, B_2 = 2, B_3 = 5, B_4 = 15$.

\begin{definition} \label{definition:hur} Let $S$ be a scheme with $d!$
		invertible on $S$.  Let $\rboxed{\hur d g S}$ denote the
		substack of $\bighurg d g S$ parameterizing those $(T, X, h: X
		\to T, f: X \to \bp^1_T)$ such that
		$X^d := \underbrace{X \times_{\bp^1_T} X
		\times_{\bp^1_T} \cdots \times_{\bp^1_T} X}_{\text{d times}}$ 
		has $B_d$ irreducible components in each geometric fiber over
		$T$, where $B_d$ is the $d$th Bell number.
\end{definition} 

The above definition is a bit opaque, but the point is that it parameterizes
degree $d$ covers $X \to \bp^1$ so that the Galois closure of $K(X)
\leftarrow
K(\bp^1)$ is an $S_d$ Galois extension, as we now verify.

\begin{lemma}
	\label{lemma:hur-sd-covers}
	The fiber product $X^d$ as in \autoref{definition:hur}
	always has at least $B_d$ irreducible components in each geometric fiber
	over $T$.

	Further, it has exactly $B_d$ components if and only if 
	$X \to \bp^1_T$ is a degree $d$ cover whose Galois closure has Galois
	group $S_d$ on geometric fibers over $T$.
	\end{lemma}
\begin{proof}
	We may reduce to the case $T$ is a geometric point.
	First, we check $X^d$ has at least $B_d$ irreducible components.
	To see this, for any partition $U = \{ S_1, \ldots, S_{\# U}\}$ of $\{1,
	\ldots, d\}$ into $\# U$ many subsets,
	let $X^U \subset X^d$ denote the subscheme of $X^d$ given as the image 
	$X^{\# U} \to X^d$ sending the $i$th copy of $X$ via the identity to
	those copies of $X$ indexed by elements of $S_i$.	
	For each partition $V$ of $\{1, \ldots ,d\}$ such that $U$ refines $V$,
	the closure of $X^U - \cup_{V, U \text{ refines } V} X^V$
	defines a nonempty union of irreducible
components of $X^d$. We have therefore produced $B_d$ irreducible components of
$X_d$, showing there are always at least $B_d$ irreducible components.

Conversely, $X^d$ has exactly $B_d$ geometric components if and only if each of
the $B_d$ subschemes described in the previous paragraph are irreducible.
Let us focus on the subscheme $Y$ corresponding to the partition $U =\{ \{1\},
\{2\}, \ldots, \{d\}\}$ into singletons, which
has degree $d!$ over $\bp^1$ and is the closure of the complement of the ``fat
diagonal'' in $X^d$.
Observe that $X \to \bp^1$ is generically \'etale because $X$ is smooth and we
are assuming the characteristic of $T$ does not divide $d!$.
Therefore, $Y \to \bp^1$ is also generically \'etale, and contains a
component whose function field is the Galois
closure of the extension of function fields $K(X) \leftarrow K(\bp^1)$.
Therefore, $Y$ is irreducible if and only if $K(Y)$ is the Galois
closure of $K(X) \leftarrow K(\bp^1)$. As $Y \to \bp^1$ has degree $d!$, this in turn is equivalent to $X \to \bp^1$
having Galois closure with Galois group $S_d$.
In particular, for any cover $X \to \bp^1$ whose Galois closure is smaller than $S_d$ 
$X^d$
has strictly more than $B_d$ irreducible components.

Finally, we check that for any $S_d$ cover, each of the $B_d$ components described
above are irreducible. As we have shown, even the component $Y$ of degree
$d!$ over $\bp^1$ is irreducible. Because all the other components correspond to
intermediate extensions between $Y$ and $\bp^1$, they are also irreducible.
\end{proof}

We next carry out the surprisingly tricky verification that $\hur d g S$ is an
open substack of $\bighurg d g S$.

\begin{proposition} \label{proposition:hur-algebraic} 
	For any integers $d, g \geq 0$, and $d!$ invertible on $S$, $\hur d g S$
	is an open substack of $\bighurg d g S$, hence an algebraic stack.  Further, if we have a family of curves $X \to
	\bp^1_T \to T$ corresponding to a $T$-point of $\hur d g S$, all fibers
	of $X$ over $T$ are geometrically irreducible.  \end{proposition}
	\begin{proof} 
		It is enough to demonstrate $\hur d g S$ is an open substack of $\bighurg d
		g S$.
Let $X \to \bp^1_T \to T$ be a family of smooth curves, corresponding to a point
of $\bighurg d g S$.  Let $\rboxed{X^d}$ denote the 
$d$-fold fiber product of $X$ over $\bp^1_T$.  
By \autoref{lemma:hur-sd-covers}, 
any such point corresponds to an $S_d$ cover of $\bp^1$ on geometric fibers, and
therefore these geometric fibers are irreducible, verifying the final statement.

	It remains to show that the locus where $X^d$ has $B_d$
	irreducible fibers in geometric fibers is open on $T$.  First, we will see in
	\autoref{lemma:no-associated-points} that the geometric fibers of $X^d$
	over $T$ have no embedded points.
	
	Because the fibers have no embedded points, we may apply
	\cite[12.2.1(xi)]{EGAIV.3}, which says that the total multiplicity (in
	the sense defined in \cite[p. 77]{EGAIV.2}, following
\cite[4.7.4]{EGAIV.2}, where total multiplicity is defined for integral schemes)
is upper semicontinuous.  From this, we conclude that the locus of geometric points in $T$ where the total
multiplicity of $X^d$ is at most $B_d$ is open.
By \autoref{lemma:hur-sd-covers},
the total multiplicity of any geometric fiber is always at least $B_d$, and hence
the locus where the total multiplicity is exactly $B_d$ is also open.
To conclude, it remains to verify the total multiplicity of any
geometric fiber is equal to the number of its irreducible components.
Note that the radicial multiplicity
of any fiber is
$1$ because $X^d$ is generically reduced, since it has a generically separable
map to $\bp^1$ by assumption that $d! \nmid \chr(k)$.  It follows that the total
multiplicity is equal to the separable multiplicity.  By definition, the
separable multiplicity of a finite type scheme over a field is equal to $1$ if
and only if the scheme is geometrically irreducible, as desired.
\end{proof} 

\begin{remark}
	\label{remark:direct-group-hur}
	Later, in \autoref{lemma:no-transposition-high-codimension},
	we will appeal to \cite{wewers:thesis}
	to construct substacks of $\bighurg d g S$ parameterizing covers with
	specified Galois group $G \subset S_d$.
	One can also see using the method of proof of
	\autoref{proposition:hur-algebraic} that these form locally closed
	substacks, with partial ordering given by the partial ordering along
	inclusion of subgroups in $S_d$.
\end{remark}

\begin{lemma}
	\label{lemma:no-associated-points} Let $X \to \bp^1$ be a degree $d$ map
	of smooth proper curves over an algebraically closed field $k$.  If the
	characteristic of $k$ does not divide $d!$, then $\rboxed{X^d} :=
	\underbrace{X \times_{\bp^1} X \times_{\bp^1} \cdots \times_{\bp^1}
X}_{\text{d times}}$ is Cohen-Macaulay, and hence has no embedded points.  \end{lemma} 
\begin{proof}
	It is enough to show $X^d$ is Cohen-Macaulay, as $1$-dimensional
	Cohen-Macaulay schemes have no embedded points.
	To verify $X^d$ is Cohen-Macaulay, we may do so \'etale locally on $\bp^1$, and
		hence we may freely base change to the strict henselization of
		$\bp^1$ at any given closed point.  Using the assumption on the
		characteristic of $k$ and the classification of prime to
		$\chr(k)$ covers of the strict henselization of $k[t]$, we may
		assume our cover is given by extracting roots of the
		uniformizer.  Equivalently, it is enough to verify
		Cohen-Macaulayness 
		in the case $X^d$ is locally described as a localization of
		$k[x_1] \otimes_{k[t]} k[x_2] \otimes_{k[t]} \cdots
		\otimes_{k[t]} k[x_m]$ where the maps $k[t] \to k[x_i]$ are
		given by $t \mapsto x_i^{s_i}$, for $s_i \leq d$.  We can
		equivalently write this tensor product as $k[x_1] \otimes_{k[t]}
		k[x_2] \otimes_{k[t]} \cdots \otimes_{k[t]} k[x_m] \simeq
		k[x_1,x_2, \ldots, x_m]/(x_1^{s_1} - x_2^{s_2}, \ldots,
		x_1^{s_1} - x_m^{s_m}) =: R.$ We wish to verify $R$ is
		Cohen-Macaulay.  
		Observe that $R$ is a $1$ dimensional scheme,
		being a finite cover of $k[t]$.  Since it is defined by $m-1$
		equations in $\ba^m$, it is a complete intersection, and
		therefore Cohen-Macaulay.  
\end{proof}
The following remark will not be used in the remainder of the paper, but may be
nice for the reader to keep in mind.
\begin{remark}
	\label{remark:}
	For $d > 2$ and $g \geq 1$, $\hur d g S$
	is a scheme when $d!$ is invertible on $S$.
	We have seen above it is an algebraic stack.
	In order to see it is a scheme, one may first verify it is an algebraic
	space by checking any degree $d$ cover of $\bp^1$ with Galois group
	$S_d$ for $d > 2$ has no
	nontrivial
	automorphisms \cite[\href{https://stacks.math.columbia.edu/tag/04SZ}{Tag 04SZ}]{stacks-project}.
	Indeed, if such a cover did have automorphisms, it would factor through
	an intermediate cover obtained by quotienting by some such nontrivial
	automorphism, forcing the Galois group to be smaller than $S_d$.

	Having established $\hur d g S$ is an algebraic space, we next wish to
	explain why it is a scheme.
	Observe this Hurwitz space has a map to the symmetric power
	$\sym^{2g-2+2d}_{\bp^1}$ of $2g-2 + 2d$ points on $\bp^1$ given by
	``taking the branch locus.'' This uses that
	$d!$ is invertible on $S$ and Riemann-Hurwitz.
	One may verify this map is separated (for example, using the valuative
	criterion) and quasi-finite (since the inertia data around the branch
	points determines the cover), hence quasi-affine
\cite[\href{https://stacks.math.columbia.edu/tag/082J}{Tag
082J}]{stacks-project}.
Therefore, it is quasi-affine over a scheme, and therefore a scheme.
\end{remark}

\section{Defining the Casnati-Ekedahl strata in Hurwitz stacks} \label{section:ce-loci}

For this section, we now fix a positive integer $d$ and a base field $k$ with $d!$
invertible on $k$. We parenthetically note that much of the following can be
generalized to work
over arbitrary base schemes.
For $T$ a $k$-scheme,
given a Gorenstein finite locally free degree $d$ cover $X \ra \bp^1_T$, from
\autoref{theorem:ce-main-generalized}, we obtain a canonical sequence of vector
bundles $(\sce^X, \scf_1^X, \scf_2^X, \ldots, \scf_{d-2}^X)$ on $\bp^1_T$.
We next aim to define certain locally closed substacks of $\bighurg d g S$
corresponding to those covers $X \to \bp^1_T$ whose associated vector bundles are isomorphic
to some specified sequence $(\sce, \scf_1, \scf_2, \ldots, \scf_{d-2})$.
To define this substack, we first define the corresponding stack of these vector
bundles.

Recall that the 
stack of locally free
rank $n$ sheaves on $\bp^1_k$
is an algebraic stack, as is well known,
see for example
\cite[Prop.\ 4.4.6]{behrend:thesis}.

\begin{definition} \label{definition:splitting-type-stratification} 
	Let $\rboxed{\vect n {\bp^1_k}}$ denote the moduli stack of locally free
	rank $n$ sheaves on $\bp^1_k$. 
	For $\vec{a} = (a_1, a_2, \ldots,
	a_n)$, with $a_i \in \bz$, let $\rboxed{\sco_{\bp^1_k}(\vec{a}) } := \bigoplus_{i=1}^n
	\sco_{\bp^1_k}(a_i)$ and let $\rboxed{\vecthn {\vec{a}} {\bp^1_k}}$
	denote the residual gerbe at the point 
	corresponding to the vector bundle $\sco_{\bp^1_k}(\vec{a})$.
 \end{definition} 
	
\begin{remark}
	\label{remark:}
	Note that this residual gerbe is indeed a locally closed substack by \cite[Thm.\ B.2]{rydh:etale-devissage}.
Alternatively, the residual gerbe is given concretely as the quotient stack
$B(\res_{\bp^1_k/k}(\aut_{\mathscr O_{\bp^1_k}(\vec{a})}))$.
\end{remark}

In order to relate the genus of a cover of $\bp^1$ to the associated vector bundle $\sce$
we need the following standard lemma.
\begin{lemma}
	\label{lemma:degree-of-e}
	Suppose $\rho: X \ra \bp^1_k$ is a degree $d$ Gorenstein finite locally free cover and let
	$\sce := \ker (\rho_* \omega_X \to \sco_{\bp^1_k})$.
If $h^0(X, \sco_X) = 1$, such as in the case that $X$ is smooth and geometrically connected,
	then $\deg(\det \sce) = g + d - 1$.
\end{lemma}
\begin{proof}
	First, we claim $\rho_* \sco_X \simeq \sco_{\bp^1_k} \oplus \sce^\vee$.
	Indeed, by duality, we have a short  exact sequence $\sco_{\bp^1_k} \ra \rho_* \sco_X \ra \sce^\vee.$
	Because all vector bundles on $\bp^1$ split, and $h^0(\bp^1_k, \rho_* \sco_X) = h^0(X, \sco_X) = 1$, we find that $\sce^\vee \simeq \oplus_{i=1}^{d-1} \sco_{\bp^1_k}(-a_i)$ for $a_i > 0$.
	Because there are no extensions of $\sco_{\bp^1_k}(-a_i)$ by $\sco_{\bp^1_k}$, the above exact sequence splits, yielding $\rho_* \sco_X \simeq \sco_{\bp^1_k} \oplus \sce^\vee \simeq \sco_{\bp^1_k} \bigoplus \oplus_{i=1}^{d-1} \sco_{\bp^1_k}(-a_i)$.
	Then, for $n$ sufficiently large and $\scl$ a degree $n$ line bundle on
	$\bp^1_k$, Riemann Roch 
on the curve $X$ implies $h^0(\bp^1_k, \sco_{\bp^1_k}(n) \bigoplus
\oplus_{i=1}^{d-1} \sco_{\bp^1_k}(-a_i + n)) = h^0(\bp^1_k, \rho_* \sco_X \otimes \scl) = h^0(X, \rho^* \scl) = dn -g + 1$.
For $n$ larger than the maximum of the $a_i$, the left hand is equal to $dn + d -\sum_{i=1}^{d-1} a_i$, and so we obtain $-\sum_{i=1}^{d-1} a_i = -g -d + 1$.
Therefore, $\deg(\det \sce) = -\deg(\det \sce^\vee) = \sum_{i=1}^{d-1}a_i = g + d - 1$.
\end{proof}

With the relation between $g$ and $\sce$ of \autoref{lemma:degree-of-e}
established, we are ready to define the Casnati-Ekedahl strata.
For the next definition, we will fix 
vectors 
$\vec{a}^\sce, \vec{a}^{\scf_1}, \ldots, \vec{a}^{\scf_{d-2}}$
and 
vector bundles $\sce, \scf_1, \ldots,
\scf_{\lfloor \frac{d-2}{2} \rfloor}$ on $\bp^1$
given by $\sce \simeq \sco_{\bp^1_k}(\vec{a}^\sce)$ and
$\scf_i \simeq \sco_{\bp^1_k}(\vec{a}^{\scf_i})$.
Note that although $d-2$ vector bundles appear in \autoref{theorem:ce-main-generalized},
the isomorphism classes of vector bundles $\scf_i$ for $i > \lfloor \frac{d-2}{2} \rfloor$ are in fact
determined by those with $i \leq \lfloor \frac{d-2}{2} \rfloor$
because duality enforces the relation $\scf_{d-2} \simeq \det \sce$ and
for $1 \leq i \leq d-3$,
$\scf_{d-2-i} \simeq \det \sce \otimes \scf_i^\vee$.

\begin{definition} \label{definition:ce-locus} 
	Let $k$ be a field with $d!$ invertible on $k$, 
	and fix a tuple of vectors	
	$(\vec a^\sce, \vec a^{\scf_1} , \ldots, \vec a^{\scf_{\lfloor \frac{d-2}{2} \rfloor}})$.
	Let
	$\rboxed{g} := 1 - d + \sum_{i=1}^{d-1} a_i^{\sce}$ and define the
	{\em Casnati-Ekedahl stratum} $\ce(\vec a^\sce, \vec a^{\scf_1} , \ldots, \vec
	a^{\scf_{\lfloor \frac{d-2}{2} \rfloor}} )$ as the locally closed
	substack of $\bighurg d g k$ given as the fiber product \begin{align*}
		\bighurg
		d g k \times_{\vect {d-1} {\bp^1_k} \times \prod_{i=1}^{\lfloor \frac{d-2}{2} \rfloor}
		\vect {\beta_i} {\bp^1_k}} \vecthn {\vec a^\sce } {\bp^1_k}
		\times \prod_{i=1}^{\lfloor \frac{d-2}{2} \rfloor} \vecthn {\vec a^{\scf_i}} {\bp^1_k}.
	\end{align*} 
	Here, $\beta_i$ are as in
	\autoref{theorem:ce-main-generalized}, and the map $\bighurg d g k \to \vect
	{d-1} {\bp^1_k} \times \prod_{i=1}^{\lfloor \frac{d-2}{2} \rfloor} \vect {\beta_i} {\bp^1_k}$ is
	induced by \autoref{theorem:ce-main-generalized}.  In other words,
	$\ce(\vec a^\sce, \vec a^{\scf_1} , \ldots, \vec a^{\scf_{\lfloor \frac{d-2}{2} \rfloor}} )$ is
	the locally closed substack of the Hurwitz stack such that the
	associated morphism $T \ra \vect {d-1} {\bp^1_k} \times
\prod_{i=1}^{\lfloor \frac{d-2}{2} \rfloor} \vect {\beta_i} {\bp^1_k}$ factors through a map $T \ra
\vecthn {\vec a^\sce } {\bp^1_k} \times \prod_{i=1}^{\lfloor \frac{d-2}{2} \rfloor} \vecthn {\vec
a^{\scf_i}} {\bp^1_k}$.  \end{definition} 

\begin{remark} \label{remark:} There is a natural generalization of the
	construction of Casnati-Ekedahl strata of covers of $\bp^1$ 
	to a version for covers of genus $g$ curves $C$ in place of the genus
	$0$ curve $\bp^1$.
	Namely, 
	given a finite locally free cover $C' \to C$ over a base $T$,	
	using \autoref{theorem:ce-main-generalized},
	one can associate 
	a sequence of vector bundles 
	on the relative curve $C \ra T$.
	A given Casnati-Ekedahl stratum would
naturally be defined as the locus where these bundles have specific
Harder-Narasimhan filtration, generalizing the notion of splitting type.
\end{remark}

\begin{remark} \label{remark:ce-loci-stratify} Since the substacks $\vecthn
	{\vec a} {\bp^1_k}$ form a stratification of $\vect n {\bp^1_k}$, it
	follows that the Casnati-Ekedahl strata, varying over all tuples $(\vec
	a^\sce, \vec a^{\scf_1} , \ldots, \vec a^{\scf_{\lfloor \frac{d-2}{2} \rfloor}} )$ form a
	stratification of $\bighurg d g k$.  
	This will enable us
		to write the class of $\bighurg d g k$ in the Grothendieck ring as
		the sum of the classes of $\ce(\vec a^\sce, \vec a^{\scf_1} ,
		\ldots, \vec a^{\scf_{\lfloor \frac{d-2}{2} \rfloor}} )$ for $(\vec a^\sce, \vec
		a^{\scf_1} , \ldots, \vec a^{\scf_{\lfloor \frac{d-2}{2} \rfloor}} )$ varying over all
		integer tuples of vectors.  
	\end{remark}

To conclude this section, we introduce some notation for objects we will
associate with a Casnati-Ekedahl stratum of the Hurwitz stack.
\begin{notation}
	\label{notation:f-e-aut-ce-notation}
	For $3 \leq d \leq 5$, $\ce(\vec a^\sce, \vec a^{\scf_1} , \ldots, \vec
	a^{\scf_{\lfloor \frac{d-2}{2} \rfloor}}) \subset \bighurg d g k$ a
	Casnati-Ekedahl stratum, for
	$\rboxed{\sce} :=\oplus_j \sco(\vec a^\sce_i), \rboxed{\scf_\bullet }:=
	\oplus_j \sco(\vec a^{\scf_\bullet}_j)$, define $\rboxed{\aut_{\ce}} :=
	\aut^{\bp^1_k/k}_{\sce, \scf_\bullet}$, as defined in
	\autoref{definition:automorphism}, depending on
	the value of $d$.  Additionally, for $f: T \ra \bp^1_k$ denote
	$\rboxed{\aut_{f^* \ce}} := \aut^{T/k}_{f^* \sce, f^* \scf_\bullet}$.
	When the map $f$ is understood, we also use $\aut_{\ce|_T}$ as notation
	for $\aut_{f^* \ce}$.
\end{notation}

\begin{remark} \label{remark:}
	The construction $\aut_{\ce|_T}$ at the end of
\autoref{notation:f-e-aut-ce-notation} will primarily be used when $T = D$,
the dual numbers, mapping to a point of $\bp^1_k$.  
Note that, in this case $\sce|_D$ and $\scf_\bullet|_T$ are free vector bundles
because all locally free bundles over $D$ are free.
\end{remark}

We conclude this section with a general discussion about the moduli stack of vector bundles on
$\bp^1_k$. This will be useful in later sections, specifically
in \autoref{lemma:smooth-sieve-codimension-bound}.

\subsection{Discussion of the moduli stack of vector bundles on $\mathbb P^1_k$}
\label{subsection:vector-bundle-closure}

Recall that, for $k$ a field, every vector bundle $\scv$ on $\mathbb P^1_k$ of rank $r$
and degree $\delta$ can be written
as $\scv \simeq \oplus_{i=1}^r \sco_{\bp^1_k}(a_i)$ where $\sum_{i=1}^r a_i =
\delta$.
The moduli stack of vector bundles
of rank $r$ and degree $\delta$ on $\bp^1_k$ is smooth and connected.
The generic point of this moduli stack is given by a balanced bundle. 
Formally,
a vector bundle $\scv$ on $\mathbb P^1$ is
{\em balanced} if it can be written as $\scv \simeq \oplus_{i=1}^r
\sco(a_i)$ with $|a_i - a_j| \leq 1$ for all $1 \leq i \leq j
\leq r$.

We can now describe when one degree $\delta$, rank $r$ vector bundle
$\scv$, viewed
as a point on the moduli stack, lies in the closure of a point corresponding to
another degree $\delta$, rank $r$ vector bundle $\scw$.
See \cite[Theorem 14.7(a)]{eisenbudH:3264-&-all-that} for a proof of the
following description.
Suppose $\scv \simeq \oplus_{i=1}^r \sco_{\bp^1_k}(a_i) = \sco_{\bp^1_k}(\vec{a})$, and there are some $a_i, a_j$ with
$a_i \leq a_j-2$. 
Let $\sigma_{i,j}(\vec{a}):= (a_1, \ldots, a_i +1, \ldots, a_j - 1,
\ldots, a_r)$ so that $\sigma_{i,j}(\vec{a})$ agrees with $\vec{a}$ except in positions $i$
and $j$.
Then $\sco_{\bp^1_k}(\sigma_{i,j}(\vec{a}))$ lies in the closure of $\sco_{\bp^1_k}(\vec{a})$.
Informally, one bundle lies in the closure of another if one can find a sequence
of moves as above relating one to the other.
More precisely, $\sco_{\bp^1_k}(\vec{b})$ lies in the closure of
$\sco_{\bp^1_k}(\vec{a})$ if we can write $\vec{b} = \sigma_{i_1, j_1} \circ
\sigma_{i_2, j_2} \cdots \circ
\sigma_{i_m, j_m}(\vec{a})$ for some non-negative integer $m$ and integers $i_1, \ldots, i_m,
j_1, \ldots, j_m$.
In particular, if one starts with any vector bundle of rank $r$ and degree
$\delta$,
one can sequentially move the entries of $\vec{a}$ closer together, which shows
that a balanced bundle correspond to the generic point of the moduli stack.

\section{Presentations of the Casnati-Ekedahl strata}
\label{section:presentations-of-ce}

We next aim to use the parametrizations  from
\autoref{section:generalizing-ce} in order to describe each of the $\ce(\sce,
\scf_\bullet)$ for $3 \leq d \leq 5$ as the quotient of an open in affine space
by an appropriate group action.
Because we will also want to parameterize simply branched covers, it will be
useful to restrict the possible ramification types of these covers.
We now introduce the notion of ramification profile, which describes the
possible ramification types of a finite cover of $\bp^1$ by a smooth curve.

\begin{definition}[Ramification Profile] \label{definition:ramfication-profile}
	Fix a positive integer $d$ and let $R = (r_1^{t_1}, r_2^{t_2}, \ldots, r_n^{t_n})$ denote a partition
	of $d$, i.e., a collection of integers with $t_1, \ldots, t_n \geq 1$
	so that $\sum_{i=1}^n t_i r_i = d$.  Here, we think of $r_i$ as the
	part sizes appearing in the partition and $t_i$ as the corresponding
	multiplicity.  A {\em ramification profile of degree $d$} is a
	partition of $d$.  For $X \ra S$ a scheme, we say $X$ {\em has
	ramification profile $R$} if for every geometric point $\spec k \in S$,
	the base change $X_k := X \times_S \spec k$ is isomorphic to
	$\coprod_{i=1}^n \left( \coprod_{j=1}^{t_i} \spec k[x]/(x^{r_i})
	\right)$.  We let $r(R) := \sum_{i=1}^n (r_i-1)t_i$ denote the
	associated {\em ramification order}.  \end{definition}
One way to think about ramification profiles as defined above is to think of
each
fiber $X_k$ of $X \to S$ having a partition into curvilinear schemes (i.e., schemes with cotangent spaces
	of dimension at most $1$ at every point) of degrees determined by the
	partition $R$.  

We next introduce the notion of an allowable collection of ramification
profiles. The point of allowable collections is that covers of $\bp^1$ whose
ramification profiles lie in an allowable collection define an open substack of
the Hurwitz stack with closed complement of high codimension.
We use the notation $\lambda \vdash n$ to indicate that $\lambda$ is a partition
of $n$.

\begin{definition} \label{definition:allowable} Fix an integer $d$. Let
	$\rboxed{\mathcal R}$ denote a collection of ramification profiles of
	degree $d$.
	We say
	$\mathcal R$ is an {\em allowable} collection of ramification profiles of
	degree $d$ if \begin{enumerate} \item $\mathcal R$ includes $(1^d)$ and
			$(2, 1^{d-2})$.
		\item Whenever $\lambda \vdash d$ lies in
			$\mathcal R$, and $\lambda' \vdash d$ is a partition
			refining $\lambda$,
			then $\lambda'$ also lies in $\mathcal R$.
\end{enumerate} \end{definition}

In the remainder of this section, we first define certain open 
substacks of Hurwitz stacks with restricted
ramification, lying in an allowable collection $\mathcal R$.
Following this, we define a certain space of sections of a vector bundle on
$\bp^1$
parameterizing smooth degree $d$ covers (for $3 \leq d \leq 5$) with specified
ramification profiles in an allowable collection.

\begin{definition} \label{definition:rhur} 
	Suppose $k$ is a field with $d!$ invertible on $k$.
	For $\mathcal R$ an allowable collection
	of ramification profiles of degree $d$, let 
$\bigrhur d g k {\mathcal R} \subset \bighurg d g k$ 
denote the open substack of $\bighurg d g
k$ (we prove it is open in \autoref{lemma:restricted-ramification-is-open}) whose $T$ points parameterize smooth curves $X \to \bp^1_T$ over $T$
	so that for each geometric point $\spec \kappa \to \bp^1_T$,
	$X_{\kappa}$ has ramification profile in $\mathcal R$.
	Let
	$\rboxed{\rhur d g k {\mathcal R}} \subset \hur d g k$
	denote the restriction of $\bigrhur d g k {\mathcal R}$ along
	$\hur d g k \subset \bighurg d g k$.

	Similarly, for $\ce \subset \bighurg d g k$ a Casnati-Ekedahl stratum,
let $\rboxed{ \rce{\mathcal R} } \subset \ce$ denote the open substack $\ce
\times_{\bighurg d g k} \bigrhur d g k {\mathcal R} \subset \ce$.
We use 
$\ce^{\mathcal R, S_d}$ to denote the pullback of $\ce^{\mathcal R}$ along
$\rhur d g k {\mathcal R}
\subset \bigrhur d g k{\mathcal R}$.
\end{definition}
\begin{remark}
	\label{remark:}
	In the case $d = 2$, the only allowable $\mathcal R$ is $\mathcal R =
\{(1^2), (2)\}$ and in this case $\rboxed{\rhur 2 g k {\mathcal R}}
= \hur 2 g k$.  
\end{remark}

We now verify that the locus of 
$\bigrhur d g k {\mathcal R} \subset \bighurg d g k$ 
is an open substack.
\begin{lemma}
	\label{lemma:restricted-ramification-is-open}
	With notation as in \autoref{definition:rhur},
	$\bigrhur d g k {\mathcal R} \subset \bighurg d g k$ is an open substack.
\end{lemma}
\begin{proof}
	As a first step, we will show this is a constructible subset.
	To do so, we can define certain substacks of
	$\bighurg d g k$ parameterizing covers $f: C \to \bp^1_k$ so that the
	multiset of ramification profiles over the geometric branch points of $f$ is equal
	to some
	fixed multiset $S$, which we call
	$(\bighurg d g k)^S$.
	One can show 
	$(\bighurg d g k)^S$
	is an algebraic stack. (See \cite[Thm.
		6.6.6]{bertinR:champs-de-hurwitz} for a very closely related
	construction.) 
	Therefore, the image of any of these stacks in
	$\bighurg d g k$ is constructible.
	Since the underlying set of $\bigrhur d g k {\mathcal R}$ is a finite
	union $\coprod_S (\bighurg d g k)^S$ for all possible multisets $S$
	producing genus $g$ covers
	which only include ramification lying in $\mathcal R$,
	we obtain that $\bigrhur d g k {\mathcal R}$ is a constructible subset
	of $\bighurg d g k$.

	To conclude, we wish to show this constructible subset is in fact an
	open subset. To do so, we only need show it is closed under
	generization.
	However, if a point of $(\bighurg d g k)^S$ has a generization which is
	a point of $(\bighurg d g k)^{S'}$ then all ramification profiles
	appearing in $S'$ must be refinements of those appearing in $S$.
	Therefore, condition $(2)$ from the definition of allowable 
	collection of ramification profiles,
	\autoref{definition:allowable},
	shows that $\bigrhur d g k {\mathcal R}$ is indeed closed under
	generization, and so defines an open substack of $\bighurg d g k$.
\end{proof}

We next give analogs of the restricted ramification loci above for spaces of
sections.

\begin{definition} \label{definition:smooth-locus} For $3 \leq d \leq 5$, 
	suppose $d!$ is invertible on $k$.
	Fix a choice of Casnati-Ekedahl stratum $\rboxed{\ce} :=
	\ce(\vec a^\sce, \vec a^{\scf_1} , \ldots, \vec a^{\scf_{\lfloor \frac{d-2}{2} \rfloor}})$, with
	associated locally free sheaves on $\bp^1_k$ given by $\rboxed{\sce_\ce}
	:= \oplus \sco(\vec a^\sce_i)$, and, if $4 \leq d \leq 5$,
	$\rboxed{\scf_\ce} := \oplus \sco(\vec a^{\scf_1}_j)$. 
	Let $g := \deg
	\det \sce_\ce - d + 1$.  Let $\rboxed{\sch_\ce}$ denote the
	associated locally free sheaf on $\bp^1$ defined in
	\eqref{equation:associate-h}.  Let $\rboxed{\mathcal R}$ denote an
	allowable collection of ramification profiles.  
	Then, define $\ramlocus \ce {\mathcal R}$ to be the open subscheme (we
		prove openness in
	\autoref{lemma:allowable-section-open}) of $\spec
	\sym^\bullet H^0(\bp^1_k, \sch_\ce)^\vee$ parameterizing $T$-points $\eta$ so that 
	$\Psi_d(\eta)$
	defines a smooth proper curve
	over $T$ with geometrically
	connected fibers such that over each geometric point $\spec \kappa \to
	\bp^1_T$,
	the pullback of $\Phi_d(\eta) \subset \bp \sce \ra
	\bp^1_T$ along $\spec \kappa \ra \bp^1_T$ has ramification profile lying in
	$\mathcal R$. 

	Also, define $\ramlocus {\ce}{\mathcal R, S_d} \subset \ramlocus
	{\ce}{\mathcal R}$ as the open subscheme parameterizing those sections
	$\eta$ for which $\Psi_d(\eta)$
	is a smooth curve $X$ with geometrically connected fibers, such that
	over each fiber, the cover $X \to \bp^1$ of degree $d$
	has Galois closure which is an $S_d$ cover.
	In other words, $\ramlocus {\ce}{\mathcal R, S_d}$ is the subset of
	$\ramlocus
	{\ce}{\mathcal R}$
	for which the map $\Psi_d$
	defines a point of $\hur d g k$.
	\end{definition} 

In the above definition, we claimed $\ramlocus \ce {\mathcal R} \subset \spec
\sym^\bullet H^0(\bp^1_k, \sch_\ce)^\vee$ is an open subscheme. We now justify
this.

\begin{lemma}
	\label{lemma:allowable-section-open}
	The subset 
$\ramlocus \ce {\mathcal R} \subset \spec \sym^\bullet H^0(\bp^1_k, \sch_\ce)^\vee$
naturally has the
structure of an open subscheme.
\end{lemma}
\begin{proof}
	Let $W_d \subset \spec \sym^\bullet H^0(\bp^1_k, \sch_\ce)^\vee$ denote
	the open subscheme parameterizing those
	sections $\eta$ for which $\Psi_d(\eta)$ has degree $d$ on all fibers.
	There is a map $W_d \to \bighurg d g k$,
	induced by $\Psi_d$ sending $\eta \mapsto \Psi_d(\eta)$.
	Under this map, $\ramlocus\ce {\mathcal R}$ is the preimage of
	$\bigrhur d g k {\mathcal R}$, which is open by
	\autoref{lemma:restricted-ramification-is-open}.
	Hence, $\ramlocus \ce {\mathcal R} \subset \spec \sym^\bullet H^0(\bp^1_k, \sch_\ce)^\vee$
	is open.
\end{proof}

\begin{example}
	\label{example:ramification-loci} If we take $\mathcal R$ in
	\autoref{definition:smooth-locus} to range over all possible
	ramification profiles (i.e., all partitions of $d$) then
	$\ramlocus \ce {\mathcal R}$ corresponds to all sections $\eta$
	as in \autoref{definition:smooth-locus} with $\Phi_d(\eta)$ a
	smooth geometrically connected degree $d$ cover of $\bp^1_k$.

	On the other hand, if we take $\mathcal R$ to be the union of two
	ramification profiles, the first given by $1^d$ and the second given by
	$(2, 1^{d-2})$, we obtain all sections $\eta$ with $\Phi_d(\eta)$
	a smooth geometrically connected curve which is simply branched
	over $\bp^1$.  
\end{example}

\subsection{Writing the class as a sum over Casnati-Ekedahl strata}

Our goal for the remainder of the section is to express the class of the Hurwitz
stack as a sum over the Casnati-Ekedahl strata, which will be somewhat more
manageable due to their descriptions as quotients of opens in affine spaces by
relatively simple algebraic groups.

\begin{proposition} \label{proposition:union-of-ce-strata} For $3 \leq d \leq
	5$, and $\mathcal R$ an allowable collection of ramification profiles of
	degree $d$, we have an equality in $\grStacks$ \begin{align*} \{\rhur d g
		k{\mathcal R}\} = \sum_{\text{ Casnati-Ekedahl strata }\ce}
	\frac{\{\ramlocus \ce {\mathcal R, S_d}\}}{\{\aut_{\ce}\}}.  \end{align*}
\end{proposition} 
\begin{proof}[Proof assuming \autoref{proposition:ce-presentation} and
  \autoref{lemma:special-auts}]  
	We claim \begin{align*} \{\rhur d g
			k{\mathcal R}\} &= \sum_{\text{ Casnati-Ekedahl strata
			}\ce} \{\rce {\mathcal R, S_d}\} \\ &= \sum_{\text{
			Casnati-Ekedahl strata }\ce} \left\{\left[\frac{\ramlocus \ce
			{\mathcal R, S_d}}{\aut_{\ce}}\right] \right\} \\ &= \sum_{\text{
		Casnati-Ekedahl strata }\ce} \frac{\{\ramlocus \ce {\mathcal
	R, S_d}\}}{\{\aut_{\ce}\}}.  \end{align*} The first equality holds because the
	Casnati-Ekedahl strata form a stratification of $\bighurg d g k$ by locally
	closed substacks.  The second holds by
	\autoref{proposition:ce-presentation}.  The final equality holds by
	\autoref{lemma:special-auts}, using both that $\aut_{\ce}$ is special so
	$\{\aut_{\ce}\} \left\{\left[\frac{\ramlocus \ce {\mathcal R,
		S_d}}{\aut_{\ce}}\right]\right	\} = \{\ramlocus \ce {\mathcal R, S_d}\}$ by \cite[Prop.\ 1.4(i)]{Ekedahl2009},
	and that $\{\aut_{\ce}\}$ is invertible.  \end{proof}

To conclude our proof of \autoref{proposition:union-of-ce-strata},
we need to verify \autoref{proposition:ce-presentation}
and \autoref{lemma:special-auts}.
We omit the proof of 
\autoref{proposition:ce-presentation}
since it is analogous to
\autoref{proposition:ce-as-quotient-stack}, where we additionally fix
isomorphisms to fixed bundles $\sce, \scf_\bullet$ on $\bp^1_k$ (as opposed to
trivial bundles on $\spec \bz$) and add in conditions associated to
the ramification profiles in $\mathcal R$ and lying in $\hur d g k$ appropriately.

\begin{proposition} \label{proposition:ce-presentation} For $3 \leq d \leq 5$, 
	fix a choice of Casnati-Ekedahl stratum $\rboxed{\ce} :=
	\ce(\vec a^\sce, \vec a^{\scf_1} , \ldots, \vec a^{\scf_{\lfloor \frac{d-2}{2} \rfloor}})$ with
	associated sheaves $\sce_\ce$ and, if $4 \leq d \leq 5$, $\scf_\ce$ as
	in \autoref{definition:smooth-locus}.  There are isomorphisms
	$[\ramlocus \ce {\mathcal R} / \aut_{\ce}] \simeq \rce {\mathcal R}.$
	and
	$[\ramlocus \ce {\mathcal R, S_d} / \aut_{\ce}] \simeq \rce {\mathcal R,
	S_d}.$
\end{proposition}

We now verify the relevant automorphism groups are special.
Because later we will have to deal with an analogous
construction over the dual numbers $D$, we include that setting in the
following lemma as well.

\begin{lemma} \label{lemma:special-auts} For $\scv$ any vector bundle on $Y$,
	for $Y = \bp^1_k$ or $Y = D$, $\res_{Y/k}(\aut_{\scv})$ and
	$\res_{Y/k}(\ker(\det: \aut_{\sce} \to \bg_m))$ are special and their
	classes are invertible in $\grStacks$.

	When $Y = \bp^1$ or $Y = D$, the three
	group schemes appearing in \eqref{equation:aut-groups} in the cases $d
	= 3, 4$, and $5$ are special.  Further, the classes of these groups are
	invertible in $\grStacks$.  \end{lemma} 
	\begin{proof} We only explicate
		the proof in the case $Y= \bp^1_k$, since the proof when $Y = D$
		is analogous but simpler (noting that all vector bundles are
		trivial over $D$).  
		
		First we show that for any vector bundle $\scg$
		on $\bp^1_k$, $\res_{\bp^1_k/k}(\aut_{\scg})$ is special.  The
		reason for this is as follows.
		Write $\scg = \oplus_{i=1}^m
		\sco_{\bp^1_k}(a_i)^{n_i}$ with $a_1 \leq \cdots \leq a_m$.
		We can express
		$\res_{\bp^1_k/k}(\aut_{\scg}) \simeq \prod_i \gl_{n_i} \ltimes
		\prod_{i<j} V_{ij}$ where $V_{ij}$ is the vector group $V_{ij} =
		\res_{\bp^1_k/k}(\hom(\sco_{\bp^1_k}(a_i)^{n_i},
		\sco_{\bp^1_k}(a_j)^{n_j})) \simeq \bg_a^{(a_j - a_i+1)n_i
		n_j}$.  It will also be useful to note that $\ker(\det) :
		\res_{\bp^1_k/k}(\aut_{\scg} )\to \res_{\bp^1_k/k}(\bg_m)$ is special, since it
		can be expressed as an extension of a power of $\bg_m$ by
		$\prod_i \sl_{n_i} \ltimes \prod_{i<j} V_{ij}$, both of which are
		special.  These statements imply the first part of the lemma.

	We now check the groups $\aut^{\bp^1_k/k}_{\sce, \scf_\bullet}$
	are special when $d = 3,4$, and $5$.
	The above observations immediately implies the claim when $d = 3.$ To deal with
	the cases $d =4$ and $d =5$, we use \autoref{lemma:aut-as-sub}.  In both
	cases, the composition coming from \autoref{lemma:aut-as-sub}
	$\aut_{\sce, \scf}^{\bp^1_k/k} \to \res_{\bp^1_k/k}(\aut_{\sce/\bp^1_k}) \times
	\res_{\bp^1_k/k}(\aut_{\scf/\bp^1_k}) \to
	\res_{\bp^1_k/k}(\aut_{\sce/\bp^1_k})$ is surjective.
	From the description in \autoref{lemma:aut-as-sub}, the kernel of this
	composition is identified with $\ker (\det) :
	\res_{\bp^1_k/k}(\aut_{\scf/\bp^1_k})
	\to \bg_m$.  As mentioned above, this is special, and so
	$\aut^{\bp^1_k/k}_{\sce, \scf_\bullet}$ is an extension of special group
	schemes, hence special.

By the above explicit description of $\aut^{\bp^1_k/k}_{\sce, \scf_\bullet}$ in
terms of classes of special linear groups, general linear groups, and vector
groups, we conclude that 
$\aut^{\bp^1_k/k}_{\sce, \scf_\bullet}$
has class which is a product of powers of $\bl$, and
expressions of the form $\bl^s - 1$ for varying $s$. 
Therefore,
$\aut^{\bp^1_k/k}_{\sce, \scf_\bullet}$
is invertible in
$\grStacks$.  \end{proof}

\section{Computing the Local Classes} \label{section:local-class}

The goal of this section is to compute the classes of sections over the dual
numbers in \autoref{theorem:local-class}. 
These classes
can be thought of as describing the ``probability'' that a curve is smooth at a
point and has a certain ramification profile.
We will then use these classes to sieve for smoothness and ramification
conditions by employing the work of Bilu and Howe
\cite{biluH:motivic-euler-products-in-motivic-statistics}
in \autoref{proposition:ce-strata-class}.
The condition of smoothness can be rephrased as a local condition over an
infinitesimal neighborhood of the point in $\bp^1$.
We will first prove \autoref{theorem:local-aut-class} which computes this
``probability'' for abstract covers,
and from this deduce \autoref{theorem:local-class}, which
computes this ``probability'' for sections of $\sch(\sce, \scf_\bullet)$.
 \autoref{theorem:local-aut-class} can be thought of as a motivic analog of Bhargava's mass formulas for counting local fields \cite{Bhargava2010}, though we note that the interesting part of \cite{Bhargava2010} is when there is wild ramification, and our hypothesis eliminates that possibility.  On the other hand, it is still interesting to upgrade even the (much easier) tame mass formula to a motivic statement.

The idea for computing these local classes seems one of the main
new insights of this paper.
In the arithmetic analogs of this work, one is able to directly count
the number of sections over $\bz/p^2\bz$, see
\cite[Lem.\ 18]{Bhargava2013b} for the degree 3 case,
\cite[Lem.\ 23]{Bhargava2004b} for the degree $4$ case,
and \cite[Lem.\ 20]{Bhargava2008a} for the degree $5$ case.
In the Grothendieck ring, when working over infinite fields, 
there are infinitely many sections, and so to
determine the relevant class, direct counting is no longer possible.
We relate computing the classes of these sections to computing the classes of
the classifying stacks of abstract automorphism groups of the corresponding schemes.
These classes can in turn be computed using stacky symmetric powers $\symm^n$
(see \autoref{definition:stacky-symmetric-powers})
and the
class of $BS_n$.
An observation
which is the key to the proof of \autoref{theorem:local-class} is that for $G$
a group scheme, we have an isomorphism of stacks $\symm^n (BG) \simeq B(G \wr
S_n)$.

Throughout this section, we fix
$d \in \bz_{\geq 1}$ an integer and let $k$ be a field with $\chr(k) \nmid d!$.
For later explicit calculations, it will be convenient to  work with the
following explicit scheme $X_{(R)}$
over $D$ which has ramification profile $R =(r_1^{t_1}, \ldots, r_n^{t_n})$,
with $\sum_{i=1}^n r_i t_i = d$,
over the closed point of $D$. Define
\begin{align} \label{equation:X} 
	X_{(R)} :=
	\coprod_{i=1}^n \left(\coprod_{j=1}^{t_i} \spec k[x,
	\varepsilon]/(x^{r_i}-\varepsilon, \varepsilon^2) \right).
\end{align} 
So $X_{(R)}$ is a disjoint union of curvilinear schemes flat over the dual
numbers, which have degrees over the dual numbers corresponding to elements of
the partition. In particular, the total degree of $X_{(R)}$ over the dual
numbers is $d$.
We use the parentheses around $R$ in $X_{(R)}$ to distinguish it from
from the base change of $X$ to $R$.

Recall we defined $\coverstack d$ prior to
\autoref{definition:sections-to-coverstack} as the algebraic stack
parameterizing degree $d$ finite
locally free covers over a base field $k$.
\begin{definition}
	\label{definition:}
	We let $\gerbe R d \subset \res_{D/k}(\coverstack d \times_{\spec k} D)$
denote the residual gerbe at the $k$-point of
$\res_{D/k}(\coverstack d \times_{\spec k} D)$
corresponding to the $D$-point of $\coverstack d$
given by $X_{(R)}$.
\end{definition}

\begin{remark}
	\label{remark:residual-gerbe-trivial}
	Since we have an induced monomorphism $B(\res_{D/k}(\aut_{X_{(R)}/D})) \to
\coverstack d$ and an epimorphism $\spec k \to B(\res_{D/k}(\aut_{X_{(R)}/D}))$,
it follows that $\gerbe R d$ is equivalent to $B(\res_{D/k}(\aut_{X_{(R)}/D}))$
from the universal property for residual gerbes.
\end{remark}

Our main result of this section is to compute the class of $\gerbe R d$ in
$\grStacks$,
and we complete the proof at the end of the section in \autoref{subsection:local-class-proof}.
\begin{theorem}
	\label{theorem:local-aut-class} Let $R$ be a ramification profile which is
	a partition of $d$.
	Let $r(R)$ be
	the ramification order associated to the ramification profile $R$, as
	defined in \autoref{definition:ramfication-profile}. Then,
	for $k$ a field with $\chr(k) \nmid d!$, we have
	\begin{align*} \{\gerbe R d\} = \bl^{-r(R)} \end{align*} 
	in $\grStacks$.
\end{theorem} 
The plan for the rest of the section is to first use
\autoref{theorem:local-aut-class} to deduce
the local condition for a section of $\sch(\sce,\scf_\bullet)$ to be smooth
in \autoref{theorem:local-class}.
Following this, we devote the remainder of the section to proving
\autoref{theorem:local-aut-class}.
The main idea is to directly compute the automorphism group of $X_{(R)}$
in terms of its combinatorial data
starting in
\autoref{subsection:computing-h}
and culminating in 
\autoref{corollary:curvilinear-auts}.
Using this, we will then be able to compute the class of the classifying stack
of the resulting affine
(but typically quite disconnected)
group scheme
in \autoref{subsection:computing-bh}.
For this, we appeal to a result of Ekedahl on stacky symmetric powers and
another result of Ekedahl showing $\{BS_d\} =1$.
We complete
the proof of \autoref{theorem:local-aut-class} in \autoref{subsection:local-class-proof}.

\begin{remark}
	\label{remark:hilbert-scheme-variant}
	With some additional work, one can also prove a variant of
	\autoref{theorem:local-aut-class} which computes the class of the
	locally closed subscheme $\reshilb R d$
	of the Hilbert scheme $\res_{D/k} (\hilb_{\bp^{d-2}_D/D}^d)$
	parameterizing curvilinear nondegenerate subschemes with ramification
	profile $R$ so that on any geometric fiber, no degree $d - 1$
	subscheme is contained in a hyperplane.
	One can show, $\{\reshilb R d\} = \{\pgl_{d-1}\} \bl^{\dim \pgl_{d-1}-r(R)}$.
	Note there is some subtlety in verifying this because this Hilbert
	scheme is naturally a $\res_{D/k}(\pgl_{d-1})$ torsor over $\coverstack
	d$, and $\pgl_{d-1}$ is not a special group. Nevertheless,
	one may prove this by ``linearizing the action'' so as to construct this as a quotient
	of a $\res_{D/k}(\gl_{d-1})$ torsor by $\res_{D/k}(\bg_m)$,
	both of which are special.
\end{remark}

\subsection{Using \autoref{theorem:local-aut-class} to compute smooth sections}
\label{subsection:smooth-sections}

Before proving \autoref{theorem:local-aut-class}, we will see how it can be used
to determine local conditions for a section in a given Casnati-Ekedahl stratum
to be smooth.
In order to apply \autoref{theorem:local-aut-class} to our problem of computing
the classes of Hurwitz stacks we want to relate it to sections of the sheaf
$\sch(\sce,\scf_\bullet)$ on $D$ (for $\sce$ and $\scf_\bullet$ trivial sheaves
on $D$ of appropriate ranks as in \autoref{notation:h-sheaf}, depending on $d$ with $3 \leq d \leq 5$).
For this we need a generalization of \autoref{proposition:ce-as-quotient-stack}
where we take a Weil restriction from the dual numbers.
More precisely,
for $3 \leq d \leq 5$,
the map $\mu_d: \gorensteinsections d \to \coverstack d$ 
defined in \autoref{definition:sections-to-coverstack}
induces a map $\res_{D/k}(\mu_d) : \res_{D/k}((\gorensteinsections d)_D) \to
\res_{D/k}((\coverstack d)_D).$
Since $\mu_d$ is invariant for the action of $\aut_{\sce, \scf_\bullet}$
as in \autoref{definition:sections-to-coverstack},
we obtain a map
$\phi_d^{D/k}: [\res_{D/k}((\gorensteinsections d)_D) / \res_{D/k}(\aut_{\sce|_D,
\scf_\bullet|_D})]
\to \res_{D/k}((\coverstack d)_D)$
induced by sending a section to its vanishing locus. 
\begin{lemma} \label{lemma:dual-number-ce-quotient-stack} 
	For $3 \leq d \leq 5$, 
	the map 
	$\phi_d^{D/k}: [\res_{D/k}((\gorensteinsections d)_D) /
	\res_{D/k}(\aut_{\sce|_D, \scf_\bullet|_D})]
\to \res_{D/k}((\coverstack d)_D)$
is an isomorphism.
\end{lemma}
This is proven via a nearly identical argument to
\autoref{proposition:ce-as-quotient-stack} 
and we omit the proof.
The one minor difference one must note is that,
in order to show $\phi_d^{D/k}$ is surjective,
for any $T \to \spec k$ and any vector bundle on $T \times_k D$,
one may replace $T$ by an open cover which trivializes the bundle.

\begin{definition}
	\label{definition:restriction-sections}
Let $\ressections R d \subset \res_{D/k}((\gorensteinsections d)_D)$
denote the preimage under the composition
\begin{align*}
	\res_{D/k}((\gorensteinsections d)_D) \to [\res_{D/k}((\gorensteinsections d)_D) /
	\res_{D/k}(\aut_{\sce|_D, \scf_\bullet|_D})]
	\xrightarrow{\phi_d^{D/k}} \res_{D/k}((\coverstack d)_D)
\end{align*}
	of $\gerbe R d \subset \res_{D/k}((\coverstack d)_D)$.
\end{definition}

\begin{remark}
	\label{remark:functorial-gerbe-description}
	We will implicitly use the following geometric description
	of the residual gerbe $\gerbe R d$ and its preimage $\ressections R d$
	in $\res_{D/k}((\gorensteinsections d)_D)$.
	As a fibered category, $\gerbe R d$ has $T$ points given by 
	finite locally free degree $d$ Gorenstein covers 
	$Z \to T \times_k D$ satisfying the following properties
\begin{enumerate}
	\item $Z$ has ramification
		profile $R$ over each geometric point $\spec \kappa \ra T_D$,
	\item $Z$ is curvilinear in the sense that for each geometric point
		$\spec \kappa \ra T$, the resulting scheme $Z \times_T \spec
		\kappa$ has $1$-dimensional Zariski tangent space at each
		point. 
\end{enumerate}

Similarly, when $3 \leq d \leq 5$, 
we can describe $\ressections R d$ as those sections $\eta \in
\res_{D/k}((\gorensteinsections d)_D)(T)$
for which the associated degree $d$ cover of $T \times_k D$,
$\Psi_d(\eta)$
(as defined in \autoref{subsection:psid})
has the above properties.
We note that $\ressections R d$ is a locally closed subscheme of 
$\res_{D/k}((\gorensteinsections d)_D)$
since the same holds for the residual gerbe $\gerbe R d$ in
$\res_{D/k}(\coverstack d)$
\cite[Thm.\ B.2]{rydh:etale-devissage}.
One may also deduce this is locally closed directly from the above
functorial description.
\end{remark}

By combining \autoref{theorem:local-aut-class}
with \autoref{lemma:dual-number-ce-quotient-stack},
we can easily deduce the following.
\begin{theorem}
	\label{theorem:local-class} 
	Let $R$ be a ramification profile which is
	a partition of $d$ and let $\ressections R d$ 
	be the scheme defined in \autoref{definition:restriction-sections}
	(with
	associated free sheaves $\sce, \scf_\bullet$ on $D$).  
	Let $r(R)$
	denote the ramification order associated to the ramification profile $R$, as
	defined in \autoref{definition:ramfication-profile}. Then,
	\begin{align*} 
			\{\ressections R d\} = \{\aut^{D/k}_{\sce,
		\scf_\bullet}\} \bl^{-r(R)}.
\end{align*} \end{theorem} 

\begin{proof}[Proof of \autoref{theorem:local-class} assuming
	\autoref{theorem:local-aut-class}] 
%
%
Using
\autoref{lemma:dual-number-ce-quotient-stack} and
\autoref{remark:residual-gerbe-trivial},
\begin{align*}
	\{[\ressections R
	d/\aut^{D/k}_{\sce, \scf_\bullet}]\} = \{B
	(\res_{D/k}(\aut_{X_{(R)}/D}))\} = \gerbe R d.  
\end{align*}
Since
$\aut^{D/k}_{\sce, \scf_\bullet}$ is special and has invertible class in
$\grStacks$ by \autoref{lemma:special-auts}, 
\cite[Prop.\ 1.4(i), Prop.\
1.1(ix)]{Ekedahl2009}
implies that
\begin{align*} \{[\ressections R d/\aut^{D/k}_{\sce, \scf_\bullet}]\}  =
\frac{\{\ressections R d\}}{\{\aut^{D/k}_{\sce, \scf_\bullet}\}}.  \end{align*}
Then, by \autoref{theorem:local-aut-class},
\begin{equation} \{\ressections R d\} =
\{\aut^{D/k}_{\sce, \scf_\bullet}\} \cdot \{\gerbe R d \} =
\{\aut^{D/k}_{\sce, \scf_\bullet}\} \cdot \bl^{-r(R)}.  \qedhere \end{equation}
\end{proof}

\subsection{Computing the algebraic group $\res_{D/k}(\aut_{X_{(R)}/D})$}
\label{subsection:computing-h}

Let $X_{(R)}$ denote the scheme as defined in \eqref{equation:X}.  Our next goal is
to compute the group scheme $\res_{D/k}(\aut_{X_{(R)}/D})$,
which we will carry out in \autoref{corollary:curvilinear-auts}.
In order to do so, we
first deal with the case that $X_{(R)}$ is connected.

\begin{lemma} \label{lemma:single-component-auts} Let $d \in \bz_{\geq 1}$, and
	let $k$ be a field with $\chr(k) \nmid d!$. 
	Let $W := \spec k[y,\varepsilon]/(\varepsilon^2, \varepsilon-y^d)$.
	For $\aut_{W/D}$ the automorphism scheme of $W$ over $D$, we have
$\res_{D/k}(\aut_{W/D}) \simeq \bg_a^{d-1} \rtimes \mu_d$, explicitly given
	by $\alpha \in \mu_d$ sending $y \mapsto \alpha y$ and $(a_1, \ldots,
a_{d-1}) \in \bg_a^{d-1}$ sending $y \mapsto y + \sum_{i=1}^{d-1}
	a_iy^{d+i}$.  \end{lemma} \begin{proof}
	For $T$ a $k$ algebra, a functorial $T$ point of
	$\res_{D/k}(\aut_{W/D})$ corresponds (upon taking global sections) to
	an isomorphism of $T$ algebras \begin{align*} \phi:
		T[y,\varepsilon]/(\varepsilon^2, \varepsilon-y^d) \simeq T[y,
		\varepsilon]/(\varepsilon^2, \varepsilon-y^d).  \end{align*}
		over $T[\varepsilon]/(\varepsilon^2)$.  Such an automorphism is
		uniquely determined by where it sends $y$.  To conclude the
		proof, it suffices to verify that any such $\phi$ is of the
		form $y \mapsto \alpha y + \sum_{i=1}^{d-1} a_i y^{d+i}$ for
		$\alpha \in \mu_d(T)$ and $a_i \in \bg_a(T)$, and conversely
		that any map of this form determines an automorphism.

Let $\phi_{\alpha, a_1, \ldots, a_{d-1}}$ denote the map of $T$ algebras
sending $y \mapsto \alpha y + \sum_{i=1}^{d-1} a_i y^{d+i}$ as above.  Under
the isomorphism 
$T[y,\varepsilon]/(\varepsilon^2, \varepsilon-y^d) \simeq T[y]/y^{2d}$, 
any automorphism $\phi$ must induce an
isomorphism on cotangent spaces, and hence send $y$ to some polynomial
$p_\phi(y) = b_1 y + b_2 y^2 + \cdots + b_{2d-1}y^{2d-1}$, with $b_1 \neq 0$
and $b_i \in T$.  The condition that $\phi$ determines a map of 
$T[\varepsilon]/(\varepsilon^2)$ algebras precisely
corresponds to $y^d = p_\phi(y)^d$. Comparing the coefficients of $y^d$ in this
equation implies $b_1 \in \mu_d(T)$.  Since $\chr(k) \nmid  d!$, comparing the
coefficients of $y^{d+1}, \ldots, y^{2d-1}$ in the equation $y^d = p_\phi(y)^d$
implies $b_2 = b_3 = \cdots = b_d = 0$.  However, the coefficients $b_{d+1},
\ldots, b_{2d-1}$ can be arbitrary and $y^d = p_\phi(y)^d$ will be satisfied.
So, any automorphism $\phi$ must be of the form $\phi_{\alpha, a_1, \ldots,
a_{d-1}}$ (where we take $a_i = b_{d+i}$ in the above notation).

To see any map $\phi_{\alpha, a_1, \ldots, a_{d-1}}$ determines an automorphism
of $T$ algebras, note first that it is well defined, because $(\alpha y +
\sum_{i=1}^{d-1} a_i y^{d+i})^d = y^d$, using that $y^{2d} = 0$.  It is an
automorphism as its inverse is explicitly given by
$\phi_{\left( \alpha^{-1}, -\alpha^{-2} a_1, -\alpha^{-3} a_2, \ldots, - \alpha^{-d} a_{d-1}
\right)}$ \end{proof}

\begin{corollary} \label{corollary:curvilinear-auts}
	Choose a partition $(r_1^{t_1}, \ldots, r_n^{t_n})$ of $d$, i.e., $d =
	\sum_{i=1}^n t_i \cdot r_i$.  For $i = 1, \ldots, n$, let $W_i := \spec
	\prod_{j=1}^{t_i} k[y,\varepsilon]/(\varepsilon^2,
	\varepsilon-y^{r_i})$.  Let $W := \coprod_{i=1}^n W_i$,
	so that $W \simeq X_{(R)}$ when $R$ is the ramification profile associated to
	the above partition.
	We have an isomorphism $\res_{D/k}(\aut_{W/D}) \simeq \prod_{i=1}^n
	\left( \bg_a^{r_i - 1} \rtimes \mu_{r_i} \right) \wr S_{t_i}$, where
	each $\bg_a^{r_i - 1} \rtimes \mu_{r_i}$ is explicitly realized acting
	on each component of $W_i$ as in \autoref{lemma:single-component-auts},
and the action of the wreath product with $S_{t_i}$ is obtained by permuting
the $t_i$ components of $W_i$.  \end{corollary} \begin{proof}

	To compute the automorphism group of $W$, first observe that any
	automorphism must permute all connected components of a fixed degree,
	and therefore $\aut_{W/D} = \prod_{i=1}^n \aut_{W_i/D}$, and
	consequently \begin{align*} \res_{D/k} (\aut_{W/D}) =  \res_{D/k}
		\left(\prod_{i=1}^n \aut_{W_i/D}\right) = \prod_{i=1}^n
		\res_{D/k} (\aut_{W_i/D}).  \end{align*}

	It therefore suffices to show $\res_{D/k}(\aut_{W_i/D}) \simeq \left(
	\bg_a^{r_i - 1} \rtimes \mu_{r_i} \right) \wr S_{t_i}$. As all
	connected components of $W_i$ are isomorphic, any automorphism is
	realized as the composition of an automorphism preserving each
	connected component, followed by some permutation of the connected
	components.  Since there are $t_i$ connected components, the group of
	permutations of the components is the symmetric group $S_{t_i}$, while
	for $Z_i$ a connected component of $W_i$, we established
	$\res_{D/k}\aut_{Z_i/D} \simeq \bg_a^{r_i - 1} \rtimes \mu_{r_i}$ in
	\autoref{lemma:single-component-auts}.  It follows that
	\begin{equation*} 
		\res_{D/k}(\aut_{W_i/D}) = \left(
	\res_{D/k}(\aut_{Z_i/D}) \right) \wr S_{t_i} = \left( \bg_a^{r_i - 1}
\rtimes \mu_{r_i} \right) \wr S_{t_i}. \qedhere 
\end{equation*} 
\end{proof}

\subsection{Computing $\{B \res_{D/k}(\aut_{X_{(R)}/D})\}$} \label{subsection:computing-bh}

Our next goal is to prove \autoref{theorem:local-aut-class} by computing the
class of $\{B\res_{D/k}(\aut_{X_{(R)}/D})\}$ in
$\grStacks$, which we carry out at the end of this section in
\autoref{subsection:local-class-proof}.  Of course, we will use our computation of
$\res_{D/k}(\aut_{X_{(R)}/D})$ from \autoref{corollary:curvilinear-auts}.
In order to set up our computation we need the following
lemma.

\begin{lemma} \label{lemma:semidirect-class} For $x$ and $y$ positive integers,
	$\{B\left( \bg_a^x \rtimes \mu_y \right)\}= \{B(\bg_a^x)\} = \bl^{-x}$, where $\mu_y$
	acts on $\bg_a^x$ by the scaling action $(\alpha, (a_1, \ldots, a_x))
	\mapsto (\alpha a_1, \ldots, \alpha a_x)$.  \end{lemma} \begin{proof}
		Indeed, we have an inclusion \begin{align*} \left( \bg_a^x
		\rtimes \mu_{y} \right) \hookrightarrow \left( \bg_a^{x}
	\rtimes \bg_m \right) \end{align*}
	where the semidirect product $\bg_a^{x} \rtimes \bg_m$ is 
	defined similarly to that in \autoref{lemma:single-component-auts}
	so that $\bg_m$ acts on $\bg_a^{x}$ by \begin{align*} \bg_m
		\times \bg_a^x & \rightarrow \bg_a^x\\ \left( \alpha, \left(
		a_1, \ldots, a_x \right) \right) & \mapsto \left( \alpha a_1,
		\ldots, \alpha a_x \right).  \end{align*} The natural inclusion
		$\mu_{y} \rightarrow \bg_m$ then respects the constructed group
		structures.
	For simplicity of notation, temporarily define $K := \bg_a^{x} \rtimes
	\mu_{y}$ and $L :=\bg_a^{x} \rtimes \bg_m$.

	Since $L$ is special, and special groups are closed under extensions,
	it follows from \cite[Prop.\ 1.1(ix)]{Ekedahl2009} that $\{BK\} =
	\{L/K\}\{BL\}$.  However, since $\bg_a^{x}$ is a normal subgroup of both
	$L$ and $K$, the quotient $L/K$ is identified with \begin{align*} L/K
	\simeq \frac{ L/ \bg_a^{x}}{K/\bg_a^{x}} \simeq \bg_m/\mu_{y} \simeq
\bg_m.  \end{align*} Since $L$ is special, using \cite[Prop.\
1.4(i)]{Ekedahl2009} and \cite[Prop.\ 1.1(v)]{Ekedahl2009}, we obtain that
$\{BL\} = \frac{1}{\{L\}} = \bl^{-x} \frac{1}{\bl-1}$. Therefore, \begin{align*}
\{BK\} = \{L/K\}\{BL\} = (\bl-1)\bl^{-x} \frac{1}{\bl-1} = \bl^{-x} = \{B(\bg_a^x)\},
\end{align*} using again that $L = \bg_a^x \rtimes \bg_m$ is special.  \end{proof}

Using \autoref{lemma:semidirect-class}, we next compute the class of $B \res_{D/k}(\aut_{X_{(R)}/D})$ in the case that the partition $R$ has a single part. 
To continue our computation, we need the notion of stacky symmetric powers:
\begin{definition}
	\label{definition:stacky-symmetric-powers}
	For $\scx$ a
stack, define the {\em stacky symmetric power} $\rboxed{\symm^n \scx} :=
[\scx^n/S_n]$ (where $\rboxed{[\scx^n/S_n]}$ denotes the stack quotient for $S_n$
	acting on $\scx^n$ by permuting the factors). 
\end{definition}

The key input to our next computation will be
that taking the stacky symmetric powers is a well defined operation on the
Grothendieck ring of stacks by 
\cite[Prop.\
2.5]{ekedahlGeometricInvariantFinite2009}.

\begin{lemma} \label{lemma:single-part-bh} For integers $s$ and $t$,
\begin{align*} B\left(\left( \bg_a^{s - 1} \rtimes \mu_{s} \right) \wr S_{t}
\right) = \bl^{-(s-1) \cdot t} \in \grStacks.  \end{align*} \end{lemma}
\begin{proof} 	By definition, \begin{align*} \{B(\left( \bg_a^{s - 1} \rtimes \mu_{s}
	\right) \wr S_{t})\} =  \{B\left( \bg_a^{s - 1} \rtimes \mu_{s} \right)
\wr B(S_{t}) \} = \{\symm^{t}(B\left( \bg_a^{s - 1} \rtimes \mu_{s} \right))\}.
\end{align*} Having computed $\{B\left( \bg_a^{s - 1} \rtimes \mu_{s} \right)\} =
\bl^{-s+1}$ in \autoref{lemma:semidirect-class}, we therefore wish to next
compute $\{\symm^{t}(B\left( \bg_a^{s - 1} \rtimes \mu_{s} \right))\}$.  Since
$\{\symm^n \scx\}$ only depends on $\{\scx\}$ by \cite[Prop.\
2.5]{ekedahlGeometricInvariantFinite2009}, and we have shown $\{B\left( \bg_a^{s
- 1} \rtimes \mu_{s} \right)\}= \{B(\bg_a^{s - 1})\}$ in
\autoref{lemma:semidirect-class} it follows that \begin{align*}
	\{\symm^{t}(B\left( \bg_a^{s - 1} \rtimes \mu_{s} \right))\} =
	\{\symm^{t}(B\bg_a^{s - 1})\}.  \end{align*} Next, by \cite[Lem.\
	2.4]{ekedahlGeometricInvariantFinite2009}, we have \begin{align*}
		\{\symm^{t}\left( \ba^{s-1} \times B \bg_a^{s-1} \right)\} &=
		\{\symm^{t}\left( B \bg_a^{s-1} \right) \times \ba^{(s-1)t}\} \\
	&= \{\symm^{t}\left( B \bg_a^{s-1} \right)\} \cdot \bl^{(s-1)t}.
\end{align*} However, since $\{\ba^{s-1} \times B \bg_a^{s-1} \} = 1$, and
$\{\symm^n \scx\}$ only depends on $\{\scx\}$ by \cite[Prop.\
2.5]{ekedahlGeometricInvariantFinite2009},
we obtain \begin{align*} \{\symm^{t}\left( B \bg_a^{s-1} \right)\} &=
	\{\symm^{t}\left( \ba^{s-1} \times B \bg_a^{s-1} \right)\} \bl^{-(s-1)t}
	\\ &= \{\symm^{t}\left( 1 \right)\} \bl^{-(s-1)t}  \\ &= \{BS_{t}\}
	\bl^{-(s-1)t} \\ &= \bl^{-(s-1)t}.  \end{align*} For the last step, we
	used \autoref{ekedahl}, which says that $\{BS_t\} = 1$.  
\end{proof}

\subsection{Completing the calculation of the local class}
\label{subsection:local-class-proof}

We now complete the proof of \autoref{theorem:local-aut-class}.  

\begin{proof}[Proof of \autoref{theorem:local-aut-class}]
	By \autoref{remark:residual-gerbe-trivial}
	$\{\gerbe R d\} = \{B(\res_{D/k}(\aut_{X_{(R)}/D}))\}$,
	and so we will now compute the latter.
	By
\autoref{corollary:curvilinear-auts} we equate
\begin{align*}
	\res_{D/k}(\aut_{X_{(R)}/D}) = \prod_{i=1}^n \left( \bg_a^{r_i - 1} \rtimes
	\mu_{r_i} \right) \wr S_{t_i}.
\end{align*}
Factoring this as a product, it suffices to
compute the class of $B\left( \bg_a^{r_i - 1} \rtimes \mu_{r_i} \right) \wr
S_{t_i}$.  Using that $\sum_{i=1}^n (r_i-1)t_i =r(R)$, the result follows from
\autoref{lemma:single-part-bh}.  \end{proof}

\section{Codimension bounds for the main result}
\label{section:codimension-bounds}

In this section, we establish various bounds on the codimension or certain bad
loci we will want to weed out when computing the class of Hurwitz stacks in the
Grothendieck ring.

\subsection{Weeding out the strata of unexpected codimension} 
\label{subsection:codimension-of-strata}

In order to compute the classes of Hurwitz stacks, we will stratify the Hurwitz
stacks by Casnati-Ekedahl strata.
The following lemma computes the codimension of these loci in the Hurwitz stack.
For the following statement, recall our notation for $\sch(\sce, \scf_\bullet)$
from \eqref{equation:associate-h}.

\begin{lemma} \label{lemma:codimension-computation} Fix some $d$, resolution
	data $(\sce, \scf_\bullet)$, and define $g$ by $\deg \det \sce = g + d - 1$.
	Letting $\rboxed{\sch} := \sch(\sce, \scf_\bullet)$, the codimension of
	$\rboxed{\ce }:= \ce(\sce, \scf_\bullet)$ in $\bighurg d g k$, assuming it
	is nonempty, is \begin{align*} \begin{cases} h^1(\bp^1_k, \End \sce) & \text{ if
	} d = 3 \\ h^1(\bp^1_k, \End \sce)+ h^1(\bp^1_k, \End
	\scf) - h^1(\bp^1_k, \sch) & \text{ if } d = 4 \text{ or
	} 5 \end{cases} \end{align*} \end{lemma} \begin{proof}
	The case $d = 3$ is proven in \cite[Prop. 1.4]{patel:thesis}.
	Therefore, for the remainder of the proof, we assume $d = 4$ or $d = 5$,
	in which case we simply write $(\sce, \scf)$ in place of $(\sce,
	\scf_\bullet)$.
		Let $\rboxed{\ce^{\circ}} := \ce(\sce^\circ,
		\scf^\circ)$ denote the dense open
		stratum, corresponding to vector bundles
		$\sce^\circ$ and $\scf^\circ$ which are
	balanced (see \autoref{subsection:vector-bundle-closure}),
		subject to the conditions that $\det \scf^\circ \simeq
		\det \sce^\circ$ when $d = 4$ and $\det
		\scf^\circ \simeq
	\det (\sce^\circ)^{\otimes 2}$ when $d = 5$, coming from
		coming from \autoref{theorem:4-structure} and
		\autoref{theorem:5-structure}.
		Similarly, let $\sch^\circ := \sch(\sce^\circ, \scf^\circ).$

		We are looking to compute the codimension of $\ce$ in $\bighurg d g	k$,
		or equivalently the difference of dimensions $\dim \ce^\circ -
		\dim \ce$.
		Using the description of
			$\ce$ from 
			\autoref{proposition:ce-presentation}
			as a quotient of an open in the
			affine space associated to $H^0(\bp^1_k,
			\sch)$ by
			$\aut_{\ce}$, it follows that the
			codimension of $\ce$ in $\bighurg d g k$ is
			$\dim \ce^{\circ} - \dim \ce$.  Since
			$\dim \ce = h^0(\bp^1_k, \sch) - \dim \aut_{\ce}$, we
			are looking to compute 
			\begin{equation}
			\begin{aligned}
				\label{equation:codimension-computation}
				& & 	(h^0(\bp^1_k,
				\sch^\circ) - \dim \aut_{\ce^\circ}) -
		(h^0(\bp^1_k, \sch) - \dim
		\aut_{\ce}) \\ & =&  (h^0(\bp^1_k,
		\sch^\circ) -
	h^0(\bp^1_k, \sch)) +(\dim
	\aut_{\ce} - \dim \aut_{\ce^\circ}).  
\end{aligned}
\end{equation}

	We will first identify $\dim \aut_{\ce} - \dim \aut_{\ce^\circ}$ with
	$h^1(\bp^1_k, \End \sce)+
	h^1(\bp^1_k, \End \scf)$.  Second, we will show
	$(h^0(\bp^1_k, \sch^\circ) - h^0(\bp^1_k,
	\sch))$ agrees with
	$h^1(\bp^1_k, \sch)$.  Combining these with
	\eqref{equation:codimension-computation} will complete the proof.

	To identify $\dim \aut_{\ce} - \dim \aut_{\ce^\circ}$, we may identify
	$\dim \aut_{\ce}$ with the dimension of the tangent space to
	$\aut_{\ce}$ at the identity, which is given by $H^0(\bp^1_k, \End \sce)
	\times H^0(\bp^1_k, \End \scf)$.  It is then enough to show that
	\begin{align*} h^0(\bp^1_k, \End \sce^\circ) - h^0(\bp^1_k, \End \sce) =
	h^1(\bp^1_k, \End \sce) \end{align*} and
	\begin{align*} h^0(\bp^1_k, \End \scf^\circ) - h^0(\bp^1_k, \End
		\scf)
	= h^1(\bp^1_k, \End \scf). \end{align*} We focus on the case of
		$\sce$, as the case of $\scf$ is completely analogous.  By
		Riemann Roch, since the degrees and ranks of $\sce$ and
		$\sce^\circ$ are the same, we find \begin{align*} h^0(\bp^1_k,
			\End \sce^\circ) - h^0(\bp^1_k, \End \sce) =
			h^1(\bp^1_k, \End \sce) - h^1(\bp^1_k, \End \sce^\circ).
		\end{align*} Because $\sce^\circ$ is balanced, we find
		$h^1(\bp^1_k, \End \sce^\circ) = 0$.

	To complete the proof, it only remains to show $(h^0(\bp^1_k,
	\sch^\circ) - h^0(\bp^1_k, \sch))$ agrees with $h^1(\bp^1_k, \sch)$.  
	Similarly to our computation above for $h^0(\bp^1_k, \End
	\sce^\circ) - h^0(\bp^1_k, \End \sce)$,
	we find
	\begin{align*}
	h^0(\bp^1_k, \sch^\circ) - h^0(\bp^1_k,
	\sch) = h^1(\bp^1_k, \sch) - h^1(\bp^1_k,
	\sch^\circ)
	\end{align*}
	by Riemann Roch.  To complete the proof, we only need verify
	$h^1(\bp^1_k, \sch^\circ) = 0$. 
	Indeed, by writing out $\sce^\circ$ and $\scf^\circ$ as sums of line
		bundles on $\bp^1_k$, and using the relation between $\det \sce$
		and $\det \scf$, the balancedness of $\sce^\circ$ and
		$\scf^\circ$ implies $h^1(\bp^1_k, \sch^\circ) = 0$.
\end{proof}

With the above lemma in hand, we may note that the codimension of
the vector bundles $(\sce, \scf_\bullet)$ in the stack of vector bundles
is $H^1(\bp^1, \End(\sce)) + H^1(\bp^1, \End(\scf))$ (the latter interpreted as
$0$ when $d  =3$).

\begin{remark}
	\label{remark:expected-codimension}
	We will think of a Casnati-Ekedahl stratum as having the ``expected codimension''
when its codimension in the Hurwitz stack agrees with the corresponding
codimension of $(\sce, \scf_\bullet)$ in the stack of tuples of vector bundles.
Using \autoref{lemma:codimension-computation} and its proof, a stratum is of the expected
codimension precisely when $H^1(\bp^1, \sch(\sce, \scf_\bullet)) = 0$.
\end{remark}
The next lemma bounds the codimension of strata not having the expected
codimension.

\begin{lemma} \label{lemma:unexpected-codimension-bound} Suppose $3 \leq d \leq
	5$ and $\rboxed{\ce} := \ce(\sce, \scf_\bullet)$ is a Casnati-Ekedahl
	stratum containing a curve $C \to \bp^1$ which does not factor through
	some intermediate cover $C' \to \bp^1$ of positive degree.  If
$H^1(\bp^1, \sch(\sce, \scf_\bullet)) \neq 0$ or $H^0(\bp^1, \sce^\vee) \neq 0$,
$\codim_{\bighurg d g k} \ce \geq \frac{g + d  -1}{d} - 4^{d-3}.$ \end{lemma}
\begin{proof} If $H^1(\bp^1, \sch(\sce, \scf_\bullet)) \neq 0$ then we have
	$\codim_{\bighurg d g k} \ce \geq \frac{g + d  -1}{d} - 4^{d-3}$ by
	\cite[Lem.\ 5.8]{canningL:tautological-classes-on-low-degree-hurwitz-spaces}
	when $d = 4$,
	\cite[Lem.\
	5.12]{canningL:tautological-classes-on-low-degree-hurwitz-spaces} when $d
	= 5$ and
	\cite[(6.2)]{miranda:triple-covers-in-algebraic-geometry},
	for the cases that $d=3$
	(see also \cite[Prop.\ 2.2]{bolognesiV:stacks-of-trigonal-curves}).  
	It therefore remains to
	show that if $H^1(\bp^1, \sch(\sce, \scf_\bullet)) = 0$ 
	but $H^0(\bp^1,
	\sce^\vee) \neq 0$, we will also have $\codim_{\bighurg d g k} \ce \geq
	\frac{g + d  -1}{d} - 4^{d-3}$.  In the case $H^1(\bp^1, \sch(\sce,
	\scf_\bullet)) = 0$, the codimension of $\ce$ in $\bighurg d g k$ is simply
	$h^1(\bp^1, \End(\sce)) + h^1(\bp^1, \End(\scf_\bullet))$, by
	\autoref{lemma:codimension-computation}.  We will only have $H^0(\bp^1,
	\sce^\vee) \neq 0$ when some summand of $\sce$ is non-positive.  
	Recall from \autoref{lemma:degree-of-e} 
	$\deg \sce = g + d- 1$.
	Therefore,
	if
	$\sce = \oplus_{i=1}^{d-1} \sco(e_i)$ with $e_1 \leq 0$, then
	$\sum_{i=2}^{d-1} (e_i - e_1) \geq g + d -1$ and hence \begin{align*}
h^1(\bp^1, \End(\sce)) &\geq h^1(\bp^1, \oplus_{i=2}^{d-1} \sco_{\bp^1}(e_1 -
e_i)) \\
&\geq g + d -1 - (d-2) \\
&= g + 1 \\
&> \frac{g + d  -1}{d} - 4^{d-3}. \qedhere \end{align*}
\end{proof}

\subsection{Weeding out covers with smaller Galois groups}
\label{subsection:bounding-smaller-galois-group}

In the next few results, culminating in 
\autoref{lemma:bounding-non-Sd-covers},
we establish bounds on the codimension of degree $d$
covers of $\bp^1$ whose
Galois closure has Galois group $G$ strictly contained in $S_d$.

For a group $G$ and a base $S$, with $\# G$ invertible on $S$, we use
$\rboxed{\grouphur G S}$ to denote the stack whose $T$-points are given by $(T,
X, h: X \to T, f: X \to \bp^1_T)$ where $X$ is a scheme, $h$ is a smooth proper
relative curve, $f$ is a finite locally
free map of degree $\# G$ so that $G$ acts on $X$ over $\bp^1_T$, together with
an isomorphism $G \simeq \aut f$.  Note that $\grouphur G S$ is an algebraic
stack with an \'etale map to the configuration space of points in $\bp^1$ given
by taking the branch divisor, as follows from \cite[Thm.\  4]{wewers:thesis},
(the key point of the construction being the algebraicity criterion in
\cite[Thm.\  1.3.3]{wewers:thesis}).  Upon specifying an embedding $G \subset
S_d$ for some $d$, there is a natural map $\grouphur G S \to \bighur d S$
sending a given cover $(T, X, h: X \to T, f: X \to \bp^1_T)$ to an associated
cover $\coprod_{h \in G \backslash S_d} (h X)/S_{d-1} \to \bp^1_T$ where we take
the disjoint union over cosets of $G \backslash S_d$ and then quotienting the
resulting $S_d$ cover by $S_{d-1}$. 
The image of this map is a substack of $\bighur d S$ whose geometric points
parameterize degree $d$ covers whose Galois group is $G$ with the specified
embedding $G \subset S_d$.
We note that we could have alternatively constructed $\grouphur G S$ directly,
as mentioned in \autoref{remark:direct-group-hur}, without appealing to
\cite{wewers:thesis}.

\begin{lemma} \label{lemma:no-transposition-high-codimension} Suppose $G \subset
	S_d$ is a subgroup not containing a transposition. Then the closure of
	the image $(\grouphur G S \to \bighur d S) \cap \bighurg d g S$ has
	dimension at most $g - 1 + d$.  \end{lemma} 
	\begin{proof} By \cite[Thm.\
		4]{wewers:thesis}, if the image $(\grouphur G S \to \bighur d S)
		\cap \bighurg d g S$ parameterizes curves with $n$ branch
		points, it has dimension $n$.  
		We therefore use $n$ for the number of branch points.
		It is possible this image has multiple components, but because
		the Galois closure of a genus $g$ degree $d$ cover of
		$\bp^1_k$ is a curve of bounded genus, there can only be finitely many
		components. We now fix one of these components and wish to show
		$n \leq g - 1 + d$.

		Let $X \to \bp^1$ be a degree $d$ genus $g$ cover corresponding to a point
		on this component with $n$ branch points.
		If $G \subset S_d$ has no
		transpositions, the inertia at any point of $\bp^1$, which is
		tame by assumption, does not act as a transposition.
		Therefore the cover is not simply branched over that point, 
		i.e., the ramification partition is not
		$(1^d)$ or $(2,1^{d-2})$. 
		Hence, the fiber over that point has total ramification
		degree at least $2$.  It follows from Riemann-Hurwitz that 
		$2g - 2 \geq -2d + 2n$ so $n \leq g - 1 + d$.  \end{proof}
		\begin{corollary}
			\label{corollary:high-codimension-for-low-degree}
			Suppose $2
			\leq d \leq 5$ and $G \subset S_d$ acts transitively
			on $\{1, \ldots, d\}$ with $G$ not isomorphic to $D_4$.
			Then, the image $(\grouphur G S \to \bighur d S) \cap \bighurg
			d g S$ has dimension at most $g - 1 + d$.
		\end{corollary} \begin{proof} 
			If $2 \leq d \leq 5$, we claim the only proper conjugacy class of
			subgroups $G \subset S_d$ acting transitively
			on $\{1, \ldots, d\}$ and containing a transposition is
			$D_4 \subset S_4$, the dihedral group of order $8$. 
			Indeed, this claim follows by a
		straightforward check of all subgroups of $S_d$.	
		The corollary then follows from \autoref{lemma:no-transposition-high-codimension}
\end{proof}

The next lemma shows that in any Casnati-Ekedahl stratum having the 
expected codimension (see \autoref{remark:expected-codimension})
the locus of non $S_d$ covers has high codimension in the Hurwitz stack.

\begin{lemma} \label{lemma:bounding-non-Sd-covers} Suppose $3 \leq d \leq 5$ and
	$\ce(\sce,\scf_\bullet)$ is a Casnati-Ekedahl stratum with $H^1(\bp^1,
	\sch(\sce, \scf_\bullet)) = 0$. 
	Suppose further that $[\ramlocus {\ce(\sce,
	\scf_\bullet)}{}/\aut_{\ce(\sce, \scf_\bullet)}]$ contains some
	geometrically connected cover $X
	\to \bp^1_k$ whose Galois closure is not $S_d$.
	Then the codimension of this locus of covers in $\bighurg d g k$ is at
	least $\frac{g+3}{2}$.
\end{lemma} 
Note that the space of $D_4$ covers is typically of codimension $2$ in 
$\bighurg d g k$, but these covers will typically lie in a Casnati-Ekedahl
stratum with 
$H^1(\bp^1,\sch(\sce, \scf_\bullet)) \neq 0$. 
\begin{proof}
	The most difficult case is when $d = 4$ and the Galois closure is $D_4$, the dihedral
	group of order $8$, this was verified in 
	\cite[Lem.\
	5.5]{canningL:tautological-classes-on-low-degree-hurwitz-spaces}.
	Note here we are using that whenever the Galois group of a
	degree $4$ cover is $D_4$, $C\to \bp^1$ necessarily factors through an
	intermediate degree $2$ cover.  
	
	It remains to verify that if we have a
	smooth geometrically connected curve $C \to \bp^1$ whose Galois closure
	is not $D_4$, the codimension of such curves is at least
	$\frac{g + 3}{2}$.
	The geometric connectedness condition guarantees that the action of $G$
	on $\{1, \ldots, d\}$ is transitive.
	Note that the dimension of such a stratum is at most $g - 1 + d$ by
\autoref{corollary:high-codimension-for-low-degree}, and hence also
	codimension $g - 1 + d$ in the $2g+2d-2$ dimensional stack $\hur d g k$.  
	The lemma follows because $g - 1 + d > \frac{g + 3}{2}$.  \end{proof} 

\subsection{Weeding out the singular sections}
\label{subsection:codimension-singular-sections}

Our next goal is to show that for any given Casnati-Ekedahl stratum,
the sections defining smooth curves can be expressed in terms of a fairly simple motivic Euler
product, away from high codimension.
This is, in some sense, the key input to our approach, and draws heavily on the
work of \cite{biluH:motivic-euler-products-in-motivic-statistics} while
also making use of our computations of classes associated to sections with given
ramification profiles over the dual numbers from \autoref{section:local-class}.
It will turn out that this codimension is the dominant term, in the sense that
for large $g$, the codimension bound we obtain on these singular sections agrees
with the codimension bound we find in our main result
\autoref{theorem:hurwitz-class}.
At this point, it may be useful to recall notation for motivic Euler products
from \autoref{subsection:motivic-euler-products}.

\begin{proposition}
	\label{proposition:ce-strata-class} Let $3 \leq d \leq 5$, and let
	$\mathcal R$ be an allowable collection of ramification profiles of
	degree $d$.  Suppose $s \geq 0$ and $\rboxed{\ce} := \ce(\sce,
	\scf_\bullet)$ is a Casnati-Ekedahl stratum for which $\sch(\sce,
	\scf_\bullet)(-s)$ is globally generated and each entry of $\vec a^\sce$
	is positive.  Then,
	\begin{align} \label{equation:smooth-sieve} \{\ramlocus \ce {\mathcal R}\}
		\equiv \bl^{\dim \ramlocus \ce {\mathcal R}} \prod_{x \in
		\bp^1_k} \left. \left( 1- \left( 1 - \frac{\left( \sum_{R \in
			\mathcal R} \bl^{-r(R)} \right)
\{\aut_{\ce|_D}\}}{\bl^{h^0(D, \sch(\sce|_D, \scf_\bullet|_D))}  } \right) t
\right) \right\rvert_{t =1} \end{align} 
modulo codimension
$\lfloor
\frac{s + 1}{2} \rfloor$
in $\grcSpaces$.

In the above product, the restriction to
$D$ is understood to take place at any subscheme $D \subset \bp^1_k$, noting
that $h^0(D, \sch(\sce|_D, \scf_\bullet|_D))$ and 
$\{\aut_{\ce|_D}\}$
are independent of the choice
of such $D$.
\end{proposition} 
\begin{proof} 
We first explain the final statement that $h^0(D, \sch(\sce|_D, \scf_\bullet|_D))$ and 
$\{\aut_{\ce|_D}\}$
are independent
of the choice of $D$.
Indeed, $h^0(D, \sch(\sce|_D, \scf_\bullet|_D))$ is independent of choice of $D \subset \mathbb P^1_k$
because it only
depends on the rank of $\sch(\sce, \scf_\bullet)$.
Similarly, $\{\aut_{\ce|_D}\}$ only depends on $d$ and the ranks of the sheaves $\sce,
\scf_\bullet$, and not the specific choice of $D \subset \mathbb P^1_k$.
We therefore focus on proving \eqref{equation:smooth-sieve}.

We will deduce \eqref{equation:smooth-sieve}
by applying \cite[Thm.\
9.3.1]{biluH:motivic-euler-products-in-motivic-statistics} with the local
condition determined by the ramification profile $R$, as
determined in \autoref{theorem:local-class}, as we next explain.  
In particular, in applying 
\cite[Thm.\
9.3.1]{biluH:motivic-euler-products-in-motivic-statistics},
we will take $r = 1, m = \lfloor \frac{s + 1}{2} \rfloor, M = 0$ in their
notation.

In some more
detail, we take $(f: X \to S, \mathcal F, \mathcal L, r, M)$ in
\cite[Thm.\  9.3.1]{biluH:motivic-euler-products-in-motivic-statistics} 
to be $(\bp^1_k \to \spec k, \sch(\sce, \scf_\bullet),\sco_{\bp^1}(1),1,0)$ and the
constructible Taylor conditions $T$ of \cite[Thm.\
9.3.1]{biluH:motivic-euler-products-in-motivic-statistics} which we will define
in the next paragraph.

For $\mathscr F$ a locally free sheaf on some scheme $X$, we use $\mathcal P^1(\mathscr F)$
to denote the $1$st order sheaf of principal parts. 
This is a locally free sheaf on $X$ whose fiber over $x \in X$ can be identified
with $H^0(X, \mathscr F \otimes \mathscr O_{X,x}/\mathfrak m_{X,x}^2)$.
So in the case $X = \mathbb P^1_k$ is a curve and $D$ is the copy of the dual numbers whose
closed point maps to $x$, the fiber of $\mathcal P^1(\mathscr F)$ at $x$ is
$H^0(\mathbb P^1_k, \mathscr F|_D)$.
For a definition and standard background on bundles of principal parts, see
\cite[\S7.2]{eisenbudH:3264-&-all-that}.
Let $T$ denote the constructible subset of 
$\spec \left( \sym^\bullet \mathcal P^1(\sch(\sce, \scf_\bullet)) \right)^\vee$
defined as follows:
upon identifying the fiber of 
$\spec \left( \sym^\bullet \mathcal P^1(\sch(\sce, \scf_\bullet)) \right)^\vee$
at $x$ with 
$H^0(\mathbb P^1_k, \sch(\sce, \scf_\bullet)|_D)$, we take the subset given by those
sections $\eta \in H^0(\mathbb P^1_k, \sch(\sce, \scf_\bullet)|_D)$
so that $\Psi_d(\eta)$ defines a curvilinear scheme over $D$ whose ramification
profile lies in $\mathcal R$.
Let $T^c$ denote the complement of $T$ in 
$\spec \left( \sym^\bullet \mathcal P^1(\sch(\sce, \scf_\bullet)) \right)^\vee$.

In order to apply 
\cite[Thm.\
9.3.1]{biluH:motivic-euler-products-in-motivic-statistics},
we need to verify the above conditions are indeed admissible in the sense of
\cite[Def.\ 9.2.6]{biluH:motivic-euler-products-in-motivic-statistics}.
Indeed, to see this, we need to check
the Taylor conditions imposed by being smooth with ramification profile lying in
$\mathcal R$
are the complement of a codimension $2 = 1 + \dim \bp^1_k$ subset of the fiber
of the first sheaf of principal parts associated to $\sch(\sce, \scf_\bullet)$
over a field valued point of $\bp^1_k$.  
First, one can directly verify (for example, by using an incidence
correspondence,) that those sections $\eta$ for which
$\Psi_d(\eta)$ are not curvilinear
form a locus of codimension at least $2$ in
$\spec \left( \sym^\bullet H^0(D, \sch(\sce|_D,
\scf_\bullet|_D))^\vee \right)$.
(Note that non-curvilinear sections also include sections with $\Psi_d(\eta)$ of positive dimension.)
It remains to show those curvilinear sections having ramification profile not
lying in $\mathcal R$ have codimension at least $2$. 
This follows from
knowledge of their class in the Grothendieck ring \autoref{theorem:local-class},
which in particular shows the codimension
of those sections having ramification profile $R$ is $r(R)$.
Since the only ramification profiles with $r(R) \leq 1$ are $(1^d)$ and $(2,
1^{d-2})$, the claim follows from the first constraint in the definition of
allowable, \autoref{definition:allowable}.

	We next use \cite[Ex.\
	5.4.6]{biluH:motivic-euler-products-in-motivic-statistics} to determine
	the value of $m$ appearing in \cite[Thm.\  9.3.1]{biluH:motivic-euler-products-in-motivic-statistics}.  In place of the value $D$
	used in \cite[Ex.\
	5.4.6]{biluH:motivic-euler-products-in-motivic-statistics}, we use
	$-s$,
	since we are reserving $D$ for the dual numbers.  Otherwise following
	the notation of \cite[Ex.\
	5.4.6]{biluH:motivic-euler-products-in-motivic-statistics}, since
	$\sco(1) = \mathcal L$ is very ample and $\mathcal L^{\otimes 0} \simeq
	\sco_{\bp^1}$ globally generated, we may take $A = 1$ and $B = 0$.  It
	follows that, (in their notation except that we use $\delta$ in place of
	$d$,) there is a surjection
	$\sco(s)^{N} \to \scf$, and so $H^0(X, \scf \otimes \sco(\delta))$ 
	(so, again, we are taking $\scf$ to be $\sch(\sce, \scf_\bullet)$)	
	is
	$1$-infinitesimally $m$-generating whenever
	$\delta \geq -s + 0 + 1(1 + (m-1)\cdot (1+1)) = -s + 1 + 2(m-1).$
	Therefore, taking $\delta = 0$, we find
	$H^0(X, \scf)$ is $1$-infinitesimally $m$-generating whenever
	$s \geq 1 + 2(m-1) = 2m - 1$.
	Therefore, we may take $m = \lfloor \frac{s + 1}{2} \rfloor$.  

	Using that $\rk \mathcal P^1( \sch(\sce,\scf_\bullet))=
	h^0(D, \sch(\sce|_D,\scf_\bullet|_D)),$
	we obtain from \cite[Thm.\
	9.3.1]{biluH:motivic-euler-products-in-motivic-statistics} the congruence 
	\begin{align} \{\ramlocus \ce {\mathcal R}\}
		&\equiv \bl^{\dim \ramlocus \ce {\mathcal R}} \prod_{x \in
		\bp^1_k} \left. \left( 1- \left( \frac{ \{T^c\}_x
					}{\bl^{\rk \mathcal P^1( \sch(\sce,
			\scf_\bullet))}}  \right) t \right)
					\right\rvert_{t =1} \\
					&
					\label{equation:smooth-sieve-T}\equiv
		\bl^{\dim \ramlocus \ce {\mathcal R}} \prod_{x \in
		\bp^1_k} \left. \left( 1- \left( 1 - \frac{ \{T\}_x
					}{\bl^{h^0(D, \sch(\sce|_D,
					\scf_\bullet|_D))}} \right) t \right)
				\right\rvert_{t =1}. \end{align} 
	In order to obtain
	\eqref{equation:smooth-sieve}, we need to identify 
	\eqref{equation:smooth-sieve-T} and the right hand side of \eqref{equation:smooth-sieve}.
	By working Zariski locally on $\mathbb P^1$ so as to trivialize the
	bundle $\sch(\sce, \scf_\bullet)$,
	it is enough to identify the fiber of $T$ over $x \in \mathbb P^1$ with the class
	$\sum_{R \in \mathcal R} \bl^{-r(R)} \{\aut_{\ce|_D}\}$.
	Indeed, this identification holds because we showed that the class of the subschemes of
	$\spec \left( \sym^\bullet H^0(D, \sch(\sce|_D,
	\scf_\bullet|_D))^\vee \right)$
	having ramification profile
	$R$ is $\bl^{-r(R)} \{\aut_{\ce|_D}\}$ when we computed the class of
$\ressections R d$ in \autoref{theorem:local-class}.  \end{proof}

In order to get a good bound on the codimension up to which
\autoref{proposition:ce-strata-class} holds,
we need to show that the value of $s$ defined there is high whenever
the codimension of the stratum is low.
We now establish this.

\begin{lemma} \label{lemma:smooth-sieve-codimension-bound} For any
	Casnati-Ekedahl stratum $\ce(\sce, \scf_\bullet)$ so that the minimum
	degree of a line bundle summand of $\sch(\sce, \scf_\bullet)$ is $s$,
	and $H^1(\bp^1, \sch(\sce, \scf_\bullet)) = 0$, we have
	$\codim_{\bighurg d
g k} \ce + \frac{s + 1}{2} \geq \frac{g+c_d}{\kappa_d}$, where $c_3 = 0, c_4=
-2, c_5 = -23, \kappa_3 = 4, \kappa_4 = 12, \kappa_5 = 40$.  \end{lemma} 
\begin{proof} 
	We can verify this in the
	case that $d = 3$ directly.  Write $\sce = \sco(s) \oplus \sco(g + 2 -
	s)$ with $s \leq g + 2 -s$, so that $\codim_{\bighurg 3 g k}\ce = h^1(\bp^1, \End(\sce)) \geq (g +
	2) - 2s -1$.  Then, $\codim_{\bighurg 3 g k}\ce  + \frac{s+1}{2} \geq (g+2)
	- 2s - 1 + \frac{s+1}{2} = \frac{2g+3-3s}{2}$.  This is minimized when $s$ is
	maximized. Since we must have $s \leq \frac{g + 2}{2}$, when $s =
	\frac{g + 2}{2}$, we find $\frac{2g+3-3s}{2} = \frac{g}{4}$.

	We now concentrate on the cases $d =4$ and $d = 5$.	First, in the
	case that $\sce$ and $\scf$ are balanced, so that $\codim_{\bighurg d g
	k} \ce = 0$, we claim that $\frac{s + 1}{2} \geq
	\frac{g+c_d}{\kappa_d}$.
	
	When $d = 4$, and $\sce$ and $\scf$ are balanced, the minimum line bundle summand of $\sce$ has degree at
	least $\frac{g + 1}{3}$ while the maximum line bundle summand of $\scf$
	has degree at most $\frac{g+4}{2}$ 
	using \autoref{lemma:degree-of-e} and the isomorphism $\det \sce \simeq \det \scf$ from
	\autoref{theorem:4-structure}.
	Hence, the minimum line bundle
	summand of $\sch$ has degree $s \geq 2 \frac{g + 1}{3} - \frac{g + 4}{2}
	= \frac{g-8}{6}$.  Therefore, $\frac{s+1}{2} \geq \frac{g-2}{12}$.

	When $d = 5$, and $\sce$ and $\scf$ are balanced, the minimum line bundle summand of $\sce$ has degree at
	least $\frac{g + 1}{4}$ 
	by \autoref{lemma:degree-of-e}	
	and the minimum line bundle summand of $\scf$
	has degree at least $\frac{2(g+4)-4}{5}$
	as $\det \sce^{\otimes 2} = \det \scf$
	by \autoref{theorem:5-structure}.
	Therefore, the minimum degree of a line bundle summand of $\sch$ is $s
	\geq 2\frac{2(g+4)-4}{5} - (g+4) + \frac{g+1}{4} = \frac{g-43}{20}$ and
	$\frac{s + 1}{2} \geq \frac{g -23}{40}$.

	In the case $d = 4$ or $5$, it remains to see that 
$\codim_{\bighurg d g k} \ce + \frac{s + 1}{2} \geq \frac{g+c_d}{\kappa_d}$
remains true for non-general strata,
	supposing still that $H^1(\bp^1, \sch(\sce, \scf)) = 0$.  In
	this case, by \autoref{lemma:codimension-computation}, the codimension
	of $\ce(\sce, \scf)$ is $h^1(\bp^1, \End(\sce)) + h^1(\bp^1, \End(\scf))$.  
	This codimension is also the codimension of the point $[(\sce,\scf)]$ 
	in the moduli stack of vector bundles
	$\vect {\rk \sce} {\bp^1_k} \times \vect {\rk \scf} {\bp^1_k}$,
	see \cite[(3.1)]{larson:the-intersection-theory-of-the-moduli-stack}.
	Hence, we wish to verify
	\begin{align*}
	\codim_{\vect {\rk \sce} {\bp^1_k} \times \vect {\rk \scf} {\bp^1_k}}
	[(\sce, \scf)] + \frac{s + 1}{2} \geq \frac{g+c_d}{\kappa_d},
	\end{align*}
	granting that we have established this in the case that $\sce, \scf$ are
	both balanced, and so correspond to the generic point of
	$\vect {\rk \sce} {\bp^1_k} \times \vect {\rk \scf} {\bp^1_k}$,
	as explained in \autoref{subsection:vector-bundle-closure}.
	Following the discussion from
\autoref{subsection:vector-bundle-closure} where we describe when one vector
bundle on $\bp^1_k$, viewed as a point in the moduli stack of vector bundles,
lies in the closure of another, any
$(\sce, \scf)$ may be connected to a balanced pair by a sequence
	of $(\sce_i, \scf_i)$, each contained in the closure of the next.  Further, we can
	assume that for any two adjacent indices $i$ and $i+1$, one of the
	following two cases occurs:
	\begin{enumerate}
		\item $\sce_i \simeq \sce_{i+1}$ and $\scf_i$ differs from
			$\scf_{i+1}$ only in two line bundles summands by a single degree.
		\item $\scf_i \simeq \scf_{i+1}$ and $\sce_i$ differs from $\sce_{i+1}$
			only in two line bundle summands by a single degree.
	\end{enumerate}
	In order to show the claimed inequality
	holds for arbitrary strata, it suffices to show it remains true under
	such specializations.  Because each such stratum has codimension at least
	$1$ in the next, it suffices to show the value of $s$ under such
	specializations decreases by at most $2$.  When $d = 4$ this is the case
	because $\sch(\sce, \scf) = \sym^2 \sce \otimes \scf^\vee$ and
	increasing a summand of $\scf$ by $1$ only decreases all summands of
	$\sch(\sce, \scf)$ by at most $1$, while decreasing a summand of $\sce$
	decreases all summands of $\sch(\sce, \scf)$ by at most $2$.  Similarly,
	when $d = 5$, so $\sch(\sce, \scf) = \wedge^2 \scf \otimes \sce \otimes
	\det \sce^\vee$, and decreasing a summand of $\sce$ by $1$ while
	maintaining $\det \sce$ decreases all summands of $\sch(\sce, \scf)$ by
	at most $1$, while decreasing a summand of $\scf$ by $1$ decreases all
summands of $\sch(\sce, \scf)$ by at most $2$.  \end{proof}

\subsection{Putting the codimension bounds together}
\label{subsection:merging-codimension-bounds}
We now merge the bounds on codimension of various bad loci established
earlier in this section to obtain the following result.

\begin{proposition} \label{proposition:sieved-hur-class} For $3 \leq d \leq 5$,
	$k$ a field of characteristic not dividing $d!$, $\mathcal R$ an
	allowable collection of ramification profiles of degree $d$, let $c_d,
	\kappa_d$
	be as in \autoref{lemma:smooth-sieve-codimension-bound}.  Define
	$\rboxed{\systemdim d g} := \chi\left( \sch(\sce, \scf_\bullet )\right)$.
		Then, $\{\rhur d g k{\mathcal R}\}$ is equal to \begin{align}
			\label{equation:ce-fixed-sum} \sum_{\text{
			Casnati-Ekedahl strata }\ce} \frac{1}{\{\aut_{\ce}\}}
			\bl^{\systemdim d g} \prod_{x \in \bp^1_k} \frac{\left(
			\sum_{R \in \mathcal R} \bl^{-r(R)} \right)
		\{\aut_{\ce|_D}\}}{\bl^{h^0(D, \sch(\sce|_D, \scf_\bullet|_D))}}
	\end{align} modulo codimension $\codimbound d g := \min(\frac{g + c_d}{\kappa_d}, \frac{g +
d-1}{d} - 4^{d-3})$ in $\grcSpaces$.  \end{proposition} \begin{proof} First, by
	\autoref{proposition:union-of-ce-strata}, it suffices to show
	\begin{align} 
		\label{equation:sd-strata-sum-reduction}	
		\sum_{\text{ Casnati-Ekedahl strata }\ce}
	\frac{\{\ramlocus \ce {\mathcal R, S_d}\}}{\{\aut_{\ce}\}} \end{align} agrees
	with \eqref{equation:ce-fixed-sum}.

	We next check \eqref{equation:sd-strata-sum-reduction} agrees with
	\begin{align} 
		\label{equation:strata-sum-reduction}	
		\sum_{\text{nonempty Casnati-Ekedahl strata }\ce}
	\frac{\{\ramlocus \ce {\mathcal R}\}}{\{\aut_{\ce}\}} \end{align} 
modulo codimension
$\min(\frac{g + c_d}{\kappa_d}, \frac{g + d-1}{d} - 4^{d-3})$.  
Since
we are working modulo codimension $\frac{g + d-1}{d} - 4^{d-3}$, we can assume
$\ce(\sce, \scf_\bullet)$ has $H^1(\bp^1, \sch(\sce, \scf_\bullet)) =0$ and
$H^0(\bp^1, \sce^\vee) = 0$, by \autoref{lemma:unexpected-codimension-bound}.
Note that the condition $H^0(\bp^1, \sce^\vee) = 0$ ensures all curves
defined by sections of $\ramlocus \ce {\mathcal R}$ are geometrically connected,
by \autoref{theorem:bertini-3-4-5}.  Since we have now restricted ourselves to
work with strata for which
$H^1(\bp^1, \sch(\sce, \scf_\bullet)) = 0$, it follows from
\autoref{lemma:smooth-sieve-codimension-bound}, with notation for $s$ as in
\autoref{lemma:smooth-sieve-codimension-bound}, that $\codim \ce(\sce,
\scf_\bullet) + \frac{s + 1}{2} \geq \frac{g + c_d}{\kappa_d}$.  We also
obtain from \autoref{lemma:bounding-non-Sd-covers} that the smooth geometrically
connected curves in $\ramlocus{\ce}{\mathcal R}$ which do not lie in $\hur d g
k$ (because they do not have Galois closure $S_d$) have codimension at least
$\frac{g + 3}{2}$ in $\hur d g k$.  Hence, as we are working modulo codimension
$\frac{g+3}{2}$, we can freely ignore these, and so
\eqref{equation:sd-strata-sum-reduction}
agrees with \eqref{equation:strata-sum-reduction}.

We next claim \eqref{equation:strata-sum-reduction} agrees with
	\begin{align} \label{equation:ce-sum} \sum_{\text{
		nonempty Casnati-Ekedahl strata }\ce} \frac{1}{\{\aut_{\ce}\}} \bl^{\dim
		\ramlocus \ce {\mathcal R}} \prod_{x \in \bp^1_k} \frac{\left(
		\sum_{R \in \mathcal R} \bl^{-r(R)} \right)
	\{\aut_{\ce|_D}\}}{\bl^{h^0(D, \sch(\sce|_D, \scf_\bullet|_D))}}.
\end{align} 
Indeed, this follows from \autoref{proposition:ce-strata-class} 
using the bounds on $s$ from \autoref{lemma:smooth-sieve-codimension-bound}.
		Next, we claim that for any $\ce(\sce, \scf_\bullet)$ as above, $\dim
	\ramlocus \ce {\mathcal R}$ is independent of $\ce$ whenever
	$\codim_{\bighurg d g k} \ce \leq \min(\frac{g + c_d}{\kappa_d}, \frac{g
	+ d-1}{d} - 4^{d-3})$.  Indeed, in this case, because $H^1(\bp^1,
	\sch(\sce, \scf_\bullet)) =0$, we find $\dim \ramlocus \ce {\mathcal R}
	= h^0(\bp^1, \sch(\sce, \scf_\bullet)) = \chi\left( \sch(\sce,
	\scf_\bullet) \right)$ and indeed this Euler characteristic only
		depends on the degrees and ranks of $\sce$ and $\scf$.  For
		notational convenience, we let $\rboxed{\systemdim d g}$ denote
		this dimension $\chi\left( \sch(\sce, \scf_\bullet )\right).$
			Then, up to codimension $\min(\frac{g +
			c_d}{\kappa_d}, \frac{g + d-1}{d} - 4^{d-3})$, in
			$\grcSpaces$, we can rewrite \eqref{equation:ce-sum} as
			\eqref{equation:ce-fixed-sum}.

	To conclude the proof, we wish to remove the word ``nonempty'' in
	\eqref{equation:ce-sum}.
	That is, there may be certain strata which contain no $S_d$ covers,
	and we wish to show they do not contribute to 
	\eqref{equation:ce-fixed-sum} in low codimension.
	The summand in 
	\eqref{equation:ce-fixed-sum}
	associated to such an empty stratum
	$\ce(\sce, \scf_\bullet)$ has codimension equal to the codimension
	of $(\sce, \scf_\bullet)$, considered as a point in the moduli stack of
	tuples of vector bundles on $\bp^1$.
	Using \autoref{corollary:high-codimension-for-low-degree}, this is only
	potentially an issue in the case $d = 4$, where we must deal with strata
	$\ce(\sce, \scf_\bullet)$ so that the generic members of $H^0(\bp^1,
	\sch(\sce, \scf_\bullet))$ define $D_4$ covers.
	In 
	\cite[Lem.\
	5.5]{canningL:tautological-classes-on-low-degree-hurwitz-spaces},
	it is shown that such strata are either codimension at least
	$\frac{g+3}{2}$ or else have $H^1(\bp^1, \sch(\sce, \scf_\bullet)) \neq
	0$. In the latter case, 
	by 
	\cite[Lem.\
	5.4]{canningL:tautological-classes-on-low-degree-hurwitz-spaces},
	such strata have codimension at least
	$\frac{g + 3}{4}-4$ in the stack of vector bundles on $\bp^1$. 
	In either case, 
	we may ignore these contributions up to our codimension bounds,
	and so 
	\eqref{equation:ce-sum} agrees with \eqref{equation:ce-fixed-sum}.
\end{proof}

\section{Proving the main result}
\label{section:main-proof}
 
In this section, we prove our main result \autoref{theorem:hurwitz-class} 
by massaging the formula for $\{\hur
d g k\}$ given in \autoref{proposition:sieved-hur-class}.
We then deduce some corollaries.

In order to prove our main result
we will need one of the
simplest cases of the ``motivic Tamagawa number conjecture'' \cite[Conj.\
3.4]{behrendD:on-the-motivic-class-of-the-stack-of-bundles}.  
To start this Tamagawa number formula, we employ the following notation.
\begin{notation} \label{notation:} For $\scg$ a vector bundle on a scheme $X$,
	let $\rboxed{\aut^{\sl,X}_{\scg}}$ denote the $\sl$ bundle over $X$
	associated to $\scg$ (i.e., the kernel of the determinant map of group
	schemes $\aut_{\scg} \to \mathbb G_m$).  We use $\aut^{\sl}_{\scg}$ as
	notation for the Weil restriction $\res_{X/\spec
	k}(\aut^{\sl,X}_{\scg})$.  For $(\sce, \scf_\bullet)$ 
	resolution data, we use $\rboxed{\aut^{\sl}(\scf_\bullet)} :=
	\prod_{i=1}^{\lfloor \frac{d-2}{2} \rfloor} \aut^{\sl}(\scf_i)$.  
\end{notation}

\begin{lemma} \label{lemma:sl-tamagawa} For any positive integer $n$,
	\begin{align*} \sum_{\substack{\text{rank $n$ vector bundles $\scv$ on
		$\bp^1$} \\ \det \scv = \sco_{\bp^1}}}
		\frac{1}{\{\aut^{\sl}_{\scv}\}}  \prod_{x \in
		\bp^1_k}\frac{\{\aut^{\sl}_{\scv|_D}\}}{\bl^{\dim
		\aut^{\sl}_{\scv|_D}}} &= \bl^{-\dim \sl_n} \in
		\grcSpaces.  \end{align*} \end{lemma} 
\begin{proof} We will
		deduce this from the motivic Tamagawa number conjecture
		for $\sl_n$ over $\bp^1$ proven in
		\cite[\S7]{behrendD:on-the-motivic-class-of-the-stack-of-bundles}.
Let $\rboxed{\mathfrak{Bun}_{G,\bp^1}}$  denote the moduli stack of $G$-bundles
on $\bp^1$.  It is shown in
\cite[\S7]{behrendD:on-the-motivic-class-of-the-stack-of-bundles}, and also via
a different argument in
\cite[\S6]{behrendD:on-the-motivic-class-of-the-stack-of-bundles}, that, in
$\grcSpaces$ (and even without inverting universally bijective morphisms) we have
$\{\mathfrak{Bun}_{\sl_n,\bp^1}\} = \bl^{-\dim \sl_n} \prod_{i=2}^n Z(\bp^1,
\bl^{-i})$, where $\rboxed{Z(\bp^1, t)} := \sum_{i=0}^\infty \{\sym^i_{\bp^1}\} t^i = \frac{1}{1-t}\frac{1}{1-\bl t}$ is the motivic Zeta
function of $\bp^1$. 

Note that $Z(\bp^1, \bl^{-i}) =
\frac{1}{1-\bl^{-i+1}}\frac{1}{1-\bl^{-i}}$ is invertible in $\grcSpaces$, with
inverse equal to $(1-\bl^{-i+1})(1-\bl^{-i})$. 
To complete the proof, it is therefore enough to demonstrate the two
equalities \begin{align} \label{equation:bun-g-class}
	\{\mathfrak{Bun}_{\sl_n,\bp^1}\} &= \sum_{\substack{\text{rank $n$ vector
	bundles $\scv$ on $\bp^1$} \\ \det \scv = \sco_{\bp^1} }}
	\frac{1}{\{\aut^{\sl}_{\scv}\}} \\ \label{equation:zeta-function-class}
\left(\prod_{i=2}^n Z(\bp^1, \bl^{-i})\right)^{-1} &= \prod_{x \in
\bp^1_k}\frac{\{\aut^{\sl}_{\scv|_D}\}}{\bl^{\dim \aut^{\sl}_{\scv|_D}}},
\end{align} 
where we note that the right hand side of \eqref{equation:zeta-function-class}
turns out to be independent of $\scv$.

We first verify \eqref{equation:bun-g-class}.  Taking cohomology on
$\bp^1$ associated to exact sequence $\sl_n \to \gl_n \to \bg_m$ defining
$\sl_n$ shows that $\sl_n$ torsors over $\bp^1$ are in bijection with $\gl_n$
torsors of trivial determinant.  We can then stratify
$\mathfrak{Bun}_{\sl_n,\bp^1} = \coprod B \aut^{\sl}_{\scv}$ as a disjoint
union of locally closed substacks corresponding to residual gerbes, as is explained for general $G$ in place of
$\sl_n$ \cite[p. 636]{behrendD:on-the-motivic-class-of-the-stack-of-bundles}. 
(Much of this argument can be verified more simply and directly in the case $G = \sl_n$.)
Noting that $\aut^{\sl}_\scv$ is special with invertible class in the
Grothendieck ring by \autoref{lemma:special-auts}, we find
$\{B(\aut^{\sl}_{\scv})\} = \frac{1}{\{\aut^{\sl}_{\scv}\}}$ and
	\eqref{equation:bun-g-class} follows.

It remains only to prove \eqref{equation:zeta-function-class}.  First, note that
$\scv$ is trivial Zariski locally and hence trivial over $D$, so
$\aut^{\sl}_{\scv|_D}$ is simply $\res_{D/\spec k}(\sl_n)$ which is an extension
of $\sl_n$ by $\bg_a^{\dim \sl_n}$.  Therefore, for any vector bundle $\scv$, we
may re-express \begin{align*} \frac{\{\aut^{\sl}_{\scv|_D}\}}{\bl^{\dim
	\aut^{\sl}_{\scv|_D}}} = \frac{\{\sl_n\}}{\bl^{\dim \sl_n}} = \{\sl_n\}
	\bl^{-\dim \sl_n} = \left(\prod_{i=2}^n (\bl^i - 1 ) \right)  \bl^{-\dim
	\sl_n} = \prod_{i=2}^n (1-\bl^{-i}).  \end{align*} 
	Using
	multiplicativity of Euler products \autoref{lemma:evaluated-product},
	\begin{align*} \prod_{x \in
		\bp^1_k}\frac{\{\aut^{\sl}_{\scv|_D}\}}{\bl^{\dim
		\aut^{\sl}_{\scv|_D}}} &= \prod_{x \in \bp^1_k} \prod_{i=2}^n
		(1-\bl^{-i}) = \prod_{i=2}^n \prod_{x \in \bp^1_k}
	(1-\bl^{-i}).  \end{align*} Hence, to prove
	\eqref{equation:zeta-function-class}, we only need check $Z(\bp^1,
	\bl^{-i})^{-1} = \prod_{x \in \bp^1_k} (1-\bl^{-i})$ for $2 \leq i \leq n$.  The right hand side is
	by definition $\prod_{x \in \bp^1_k} (1-\bl^{-i}t)|_{t = 1}$.  By
	\cite[\S3.8, Property 4]{bilu:thesis}, we have $\prod_{x \in \bp^1_k}
	(1-\bl^{-i}t)|_{t = 1} = \prod_{x \in \bp^1_k} (1-t)|_{t = \bl^{-i}}$.
	(As a word of warning, it is important that the substitution we made here
		was via replacing $t$ by its product with a power of $\bl$, see
	\cite[Rem.\ 6.5.2 and
6.5.3]{biluH:motivic-euler-products-in-motivic-statistics}.) Finally, by
\cite[Ex.\ 6.1.12]{biluH:motivic-euler-products-in-motivic-statistics} and
multiplicativity of Euler products \cite[Prop.\ 3.9.2.4]{bilu:thesis}, $\prod_{x
\in \bp^1_k} (1-t)|_{t = \bl^{-i}} = Z(\bp^1, \bl^{-i})^{-1}$.  \end{proof}


In fact, we will need a slight generalization of the above formula from
\autoref{lemma:sl-tamagawa}, where we replace bundles of degree $0$ with bundles
of arbitrary fixed degree.

\begin{lemma} \label{lemma:pgl-tamagawa} For any positive integer $n$, and any
	fixed integer integer $\delta$,
	\begin{align*} \sum_{\substack{\text{rank $n$ vector bundles $\scv$ on
		$\bp^1$} \\ \deg \scv = \delta}}
		\frac{1}{\{\aut^{\sl}_{\scv}\}}  \prod_{x \in
		\bp^1_k}\frac{\{\aut^{\sl}_{\scv|_D}\}}{\bl^{\dim
		\aut^{\sl}_{\scv|_D}}} &= \bl^{-\dim \sl_n} \in
		\grcSpaces.  \end{align*} \end{lemma} 
\begin{proof}
	The case that $\delta = 0$ was precisely covered in
	\autoref{lemma:sl-tamagawa}. Therefore, it remains to show the left hand
	side of the statement of the lemma is independent of $\delta$.
	The left hand side is unchanged upon replacing $\delta$ by $\delta \pm
	n$ because tensoring with the line bundle $\sco_{\bp^1_k}(1)$ defines a
	bijection from degree $\delta$ vector bundles of rank $n$ to degree $\delta + n$
	vector bundles of rank $n$, which preserves automorphism groups.
	Therefore, it suffices to show that the left hand side is independent of
	which congruence class $\delta$ lies in $\bmod n$.

	Next, note that one can express
	\begin{align*}
	\mathbb L^{\dim \pgl_n} \cdot \sum_{\substack{\text{rank $n$ vector bundles $\scv$ on
	$\bp^1$} \\ \deg \scv = \delta}}
	\frac{1}{\{\aut^{\sl}_{\scv}\}}
	\prod_{x \in
		\bp^1_k}\frac{\{\aut^{\sl}_{\scv|_D}\}}{\bl^{\dim
		\aut^{\sl}_{\scv|_D}}},
	\end{align*}
	as $f_\delta(\mathbb L),$ for $f_\delta$ a rational function. 
	Hence, for $\delta, \delta'$ two distinct residue classes $\bmod
	n$, it is enough to show $f_\delta(q) =	f_{\delta'}(q)$ for
	infinitely many integers $q$, as then the two rational functions
	must agree.
	The reason for choosing the above expression is that $\sum_{\delta=1}^n
	f_\delta(q)$ can also be identified with the Tamagawa number of $\pgl_n$
	over the function field $\mathbb P^1_{\mathbb F_q}$. Here we are using
	that if one starts with a vector bundle $\scv$ on $\mathbb P^1_{F_q}$,
	$\# \aut \mathbb P\scv =	\frac{1}{q-1} \aut \scv$, and the same
	expression calculates the number of automorphisms of $\scv$ with trivial
	determinant.

	We now use the above description in terms of Tamagawa numbers to show
	$f_\delta(q) = f_{\delta'}(q)$ for $1 \leq \delta \leq \delta' \leq n$.
	For $x \in \bp^1_{\mathbb F_q}$ a closed point, let
	$\widehat{\sco}_{x,\mathbb P^1_{\mathbb F_q}}$
	denote the complete local ring at $x$.
	We use $K(\mathbb P^1_{\mathbb F_q})$ to denote the function field of
	$\mathbb P^1_{\mathbb F_q}$, and $\mathbb A := \prod_{v \in \mathbb P^1_{\mathbb
F_q} \text{ closed points}} \left( K(\mathbb P^1_{\mathbb F_q})_v,
\widehat{\mathscr O}_v
\right)$ to denote the ring of adeles for
this function field.
	Note that the Tamagawa number can be expressed as the Tamagawa measure
	of $\pgl_n(K(\mathbb P^1_{\mathbb F_q})) \backslash \pgl_n(\mathbb A)$.
 
	There is a projection map 
	\begin{align}
		\label{equation:tamagawa-quotient}
	\alpha: \pgl_n(K(\mathbb P^1_{\mathbb F_q})) \backslash \pgl_n(\mathbb A) 
	\to \pgl_n(K(\mathbb P^1_{\mathbb F_q})) \backslash
	\pgl_n(\mathbb A) / \prod_{\on{places } x \in \mathbb
	P^1_{\mathbb F_q}} \pgl_n(\widehat{\sco}_{x,\mathbb
	P^1_{\mathbb F_q}}).
	\end{align}
	We claim one can identify the target with the set of isomorphism classes
	of $\pgl_n$ bundles on $\mathbb P^1_{\mathbb F_q}$, and moreover, if
	$X$ is a $\pgl_n$ bundle, the
	Tamagawa measure satisfies 
	\begin{align}
		\label{equation:bundle-tamagawa-measure}
	\mu_{\on{Tam}}(\alpha^{-1}([X])) =
	\frac{\mu_{\on{Tam}}\left(\prod_{\on{places } x \in \mathbb P^1_{\mathbb
		F_q}} \pgl_n(\widehat{\sco}_{x,\mathbb P^1_{\mathbb
F_q}})\right)}{ \# \aut X}.
	\end{align}
		Our claim essentially follows from \cite[Proposition
	1.3.2.11]{gaitsgoryL:weils-conjecture}, except the statement there
	assumes the group $G$ is simply connected, which is not the case for
	$\pgl_n$. However, the only place in the proof 
	(see the proof of \cite[Proposition
	1.3.2.10]{gaitsgoryL:weils-conjecture})
	that the simply
	connected hypothesis was used was to show there is some dense open of
	$\mathbb P^1_{\mathbb F_q}$ on
	which any $\pgl_n$ bundle is trivial.
	We can instead verify this directly as follows. Note first that 
	the Brauer group of $\mathbb P^1_{\mathbb F_q}$ is trivial
	\cite[Remarques 2.5(b)]{grothendieck:brauer-iii}.
	Hence, any $\pgl_n$ bundle
	on $\mathbb P^1_{\mathbb F_q}$ is the projectivization of a $\gl_n$
	bundle.
	Since any $\gl_n$ bundle is Zariski-locally trivial, the same holds for
	any $\pgl_n$ bundle on $\mathbb P^1_{\mathbb F_q}$.

	There is natural map
	$\pi: \pgl_n(K(\mathbb P^1_{\mathbb F_q})) \backslash \pgl_n(\mathbb A) \to \bz/n\bz$
	which factors through the double quotient map
	\eqref{equation:tamagawa-quotient}
	parameterizing projective bundles on $\mathbb P^1$, and sends a
	projective bundle to its degree $\bmod n$. Note that the degree of the
	projectivization of a vector bundle is not well defined as an integer,
	but it is well defined $\bmod n$.
	In this setup, we obtain $\mu_{\on{Tam}}(\pi^{-1}(\delta)) = f_\delta(q)$, where 
	$\mu_{\on{Tam}}$ denotes the right-invariant Tamagawa measure,
	by summing \eqref{equation:bundle-tamagawa-measure}
	over all bundles of degree $\delta \bmod n$.

		Since the Tamagawa measure is translation invariant, we can right-translate by
	any element of $\pgl_n(\mathbb A)$
		in $\pi^{-1}(\delta' - \delta)$ and this sends $\pi^{-1}(\delta)$
		to $\pi^{-1}(\delta')$. Hence the Tamagawa measures of
		$\pi^{-1}(\delta)$ and $\pi^{-1}(\delta')$ agree, so
		$f_\delta(q) =f_{\delta'}(q)$, as desired.
\end{proof}

For our main theorem, we will also need the following elementary dimension
comparison.

\begin{lemma}
	\label{lemma:dimension-aut-comparison}
	For $3 \leq d \leq 5$ and $\systemdim d g$ as in
	\autoref{proposition:sieved-hur-class},
	$\systemdim d g - \dim \sl_{\rk \sce} - \sum_{i=1}^{\frac{\lfloor d - 2
	\rfloor}{2}} \dim \sl_{\rk \scf_i} = \dim \hur d g
k + 1$.  
\end{lemma}
\begin{proof}
Indeed, this can be checked separately in the cases $d = 3, 4,$ and
$5$.

	We now check the most difficult case that $d = 5$, leaving the other
	cases to the reader.  In the case $d = 5$, one computes \begin{align*}
		\systemdim d g &= \chi(\wedge^2 \scf \otimes \sce \otimes
		\det \sce^\vee) \\ &= \rk \left(\wedge^2 \scf \otimes \sce \otimes
		\det \sce^\vee\right) + \deg \wedge^2 \scf \otimes \sce \otimes
		\det \sce^\vee \\ &= \binom{5}{2} \cdot 4 + 16 \deg \scf - 30 \deg
	\sce = 40 + 32 \deg \sce - 30 \deg \sce \\ &= 40 + 2 \deg \sce \\ &= 40
+ 2g + 2d - 2.  \end{align*} Furthermore, still in the $d = 5$ case, $\dim
\sl_{\rk \sce} = 15$ and $\dim \sl_{\rk \scf} = 24$.  Therefore, \begin{align*}
	\systemdim d g - \dim \sl_{\rk \sce} - \dim \sl_{\rk \scf} &= 40 + 2g + 2d -
2 - 15 - 24 \\ &= (2g+2d-2)+1 \\ &= \dim \hur d g k + 1 \end{align*} as claimed.   
\end{proof}

We are finally prepared to prove our main theorem.
For the statement of our main theorem, recall we defined
$\codimbound d g = \min(\frac{g + c_d}{\kappa_d}, \frac{g + d-1}{d} - 4^{d-3})$
in \autoref{proposition:sieved-hur-class}, with $c_3 = 0, c_4 = -2,$ and $c_5 =
-23$.
Note that for $g \gg 0$, $\codimbound d g$ is more than
$\frac{g}{\kappa_d}-1$.

\begin{theorem} \label{theorem:hurwitz-class} Let $2 \leq d \leq 5$, $k$ a field
	of characteristic not dividing $d!$, $\mathcal R$ an allowable
	collection of ramification profiles of degree $d$. Then, \begin{align}
		\label{equation:rhur-class} \{\rhur d g k{\mathcal R}\} \equiv
		\frac{\bl^{\dim{\hur d g k}}}{1-\bl^{-1}} \left( \prod_{x \in
		\bp^1_k} \left( \sum_{R \in \mathcal R} \bl^{-r(R)} \right)
	\left( 1-\bl^{-1} \right) \right) 
\end{align} 
modulo codimension
$\codimbound d g$ 
in
$\grcSpaces$.  
In the case $d = 2$, the left hand
side and right hand side of \eqref{equation:rhur-class} are actually equal in
$\grStacks$ (and not just equivalent in $\grcSpaces$ modulo terms of a certain
dimension).  \end{theorem} \begin{proof} The proof of the $d = 2$ case is of a
	different   nature and we defer it to the end of \autoref{section:2}.
	We now concentrate on the case $3 \leq d \leq 5$.  By
	\autoref{proposition:sieved-hur-class}, our goal reduces to showing
	\eqref{equation:ce-fixed-sum} agrees with the right hand side of
	\eqref{equation:rhur-class}.

Recall our notation for $\systemdim d g$ from
\autoref{proposition:sieved-hur-class}.
	First, we claim we can rewrite 
\eqref{equation:ce-fixed-sum} as
\begin{align}
	\label{equation:e-and-f-single-sum} \frac{1}{\bl - 1} \sum_{\sce,
	\scf_\bullet} \frac{1}{\{\aut^{\sl}_{\sce}\}}
	\frac{1}{\{\aut^{\sl}_{\scf_\bullet}\}} \bl^{\systemdim d g} \prod_{x \in
	\bp^1_k} \frac{\left( \sum_{R \in \mathcal R} \bl^{-r(R)} \right) (\bl
	-1)\bl\{\aut^{\sl}_{\sce|_D}\} \{\aut^{\sl}_{\scf_\bullet|_D}\}}{\bl^{h^0(D,
	\sch(\sce|_D, \scf_\bullet|_D))}}, \end{align} with the summation over
	$\sce, \scf_\bullet$ interpreted as follows: $\sce$ ranges over all
	$\bp^1$ bundles of rank $d - 1$ and degree $g + d -1$; when $d = 3$,
	$\scf_\bullet$ is interpreted as being empty (so all classes associated
	to it are $1$); when
	$d =4$, $\scf_\bullet = \scf$ has rank $2$ and degree $g + d -1$; 
	when $d = 5$, $\scf_\bullet = \scf$ has rank $5$ and degree $2(g + d -
	1)$.  
	To see this we proceed as follows. For $\ce = \ce(\sce, \scf_\bullet)$, using the formula for
	$\aut^{\bp^1/k}_{\sce, \scf_\bullet}$ from \autoref{lemma:aut-as-sub},
	we can rewrite \begin{align} \label{equation:aut-class-over-p1}
	\frac{1}{\{\aut_{\ce}\}} = \{\res_{\bp^1_k/k}(\bg_m)\}
		\frac{1}{\{\aut^{\bp^1/k}_{\sce}\}}
		\frac{1}{\{\aut^{\bp^1/k}_{\scf_\bullet}\}}= \frac{1}{\bl -1}
	\frac{1}{\{\aut^{\sl}_{\sce}\}} \frac{1}{\{\aut^{\sl}_{\scf_\bullet}\}},
\end{align} where we interpret $\{\aut^{\sl}_{\scf_\bullet}\} = 1$ when $d = 3$.
Similarly, \begin{align} \label{equation:aut-class-over-dual-numbers}
	\{\aut_{\ce|_D}\} = \{\res_{D/k}(\bg_m)\} \{\aut^{\sl}_{\sce|_D}\}
\{\aut^{\sl}_{\scf_\bullet|_D}\}
=
(\bl -1)\bl\{\aut^{\sl}_{\sce|_D}\}
\{\aut^{\sl}_{\scf_\bullet|_D}\}.  \end{align} 
Hence, using
\eqref{equation:aut-class-over-p1} and
\eqref{equation:aut-class-over-dual-numbers}, we can rewrite
\eqref{equation:ce-fixed-sum} as \eqref{equation:e-and-f-single-sum}.

We next make a sequence of simplifications of 
\eqref{equation:e-and-f-single-sum}.
Then, summing over the same pairs $(\sce,\scf_\bullet)$ as in
	\eqref{equation:e-and-f-single-sum}, we can rewrite it as 
	\begin{equation}
		\begin{aligned}
		\label{equation:e-and-f-single-sum-two} \frac{1}{\bl - 1} &\left(
		\sum_{\sce} \frac{1}{\{\aut^{\sl}_{\sce}\}} \right) \left(
		\sum_{\scf_\bullet} \frac{1}{\{\aut^{\sl}_{\scf_\bullet}\}}
	\right)   \bl^{\systemdim d g} \\
	&\cdot \prod_{x \in \bp^1_k} \frac{\left(
	\sum_{R \in \mathcal R} \bl^{-r(R)} \right) (\bl
-1)\bl\{\aut^{\sl}_{\sce|_D}\} \{\aut^{\sl}_{\scf_\bullet|_D}\}}{\bl^{h^0(D,
\sch(\sce|_D, \scf_\bullet|_D))}}, \end{aligned}
\end{equation}
where the parenthesized sum of
$\scf_\bullet$ is interpreted as $1$ in the case $d = 3$, in this line and in
the remainder of the proof.  

Next, observe that $2 + \dim \aut^{\sl}_{\sce|_D} + \dim \aut^{\sl}_{\scf_\bullet|_D} = h^0(D, \sch(\sce|_D,
\scf_\bullet|_D))$. Indeed, this can be checked separately in the cases $d = 3,
4$, and $5$. When $d = 3$, both sides equal $8$, when $d = 4$, both sides equal
$24$, and when $d = 5$, both sides equal $80$.  Therefore, we can rewrite
\eqref{equation:e-and-f-single-sum} as 
\begin{equation}
	\begin{aligned}
	\label{equation:e-and-f-single-sum-three} \frac{1}{\bl - 1} &\left(
	\sum_{\sce} \frac{1}{\{\aut^{\sl}_{\sce}\}} \right) \left(
	\sum_{\scf_\bullet} \frac{1}{\{\aut^{\sl}_{\scf_\bullet}\}} \right)
	\bl^{\systemdim d g} \\
	& \cdot\prod_{x \in \bp^1_k} \left( \sum_{R \in \mathcal
	R} \bl^{-r(R)} \right) \frac{(\bl -1)\bl}{\bl^2}
	\frac{\{\aut^{\sl}_{\sce|_D}\}}{\bl^{\dim \aut^{\sl}_{\sce|_D}}}
	\frac{\{\aut^{\sl}_{\scf_\bullet|_D}\}}{\bl^{\dim
	\aut^{\sl}_{\scf_\bullet|_D}}} \end{aligned} 
\end{equation}
	
	Using multiplicativity of
	Euler products, as proven in \autoref{lemma:evaluated-product}, this becomes
	\begin{equation} 		\begin{aligned} 
			\label{equation:e-and-f-sums-and-products}
&\frac{1}{\bl - 1} \left( \sum_{\sce}
			\frac{1}{\{\aut^{\sl}_{\sce}\}} \right) \left(
			\sum_{\scf_\bullet}
		\frac{1}{\{\aut^{\sl}_{\scf_\bullet}\}} \right) \bl^{\systemdim d
		g} \left( \prod_{x \in \bp^1_k} \left( \sum_{R \in \mathcal R}
		\bl^{-r(R)} \right) \left( 1-\bl^{-1} \right) \right)\\ &\cdot
		\left( \prod_{x \in
		\bp^1_k}\frac{\{\aut^{\sl}_{\sce|_D}\}}{\bl^{\dim
	\aut^{\sl}_{\sce|_D}}} \right) \cdot \left( \prod_{x \in
	\bp^1_k}\frac{\{\aut^{\sl}_{\scf_\bullet|_D}\}}{\bl^{\dim
\aut^{\sl}_{\scf_\bullet|_D}}} \right). \end{aligned} \end{equation} 

Then, by the Tamagawa number formula for $\sl_n$, and its slight generalization
from \autoref{lemma:pgl-tamagawa}, \begin{align}
	\label{equation:tamagawa-evaluation-e}
	\sum_{\sce} \frac{1}{\{\aut^{\sl}_{\sce}\}}\prod_{x \in
	\bp^1_k}\frac{\{\aut^{\sl}_{\sce|_D}\}}{\bl^{\dim \aut^{\sl}_{\sce|_D}}}
	&= \bl^{-\dim \sl_{\rk \sce}}, \\ 
	\label{equation:tamagawa-evaluation-f}
	\sum_{\scf_\bullet}
\frac{1}{\{\aut^{\sl}_{\scf_\bullet}\}}\prod_{x \in
\bp^1_k}\frac{\{\aut^{\sl}_{\scf_\bullet|_D}\}}{\bl^{\dim
\aut^{\sl}_{\scf_\bullet|_D}}} &= \bl^{-\dim \sl_{\rk \scf_\bullet}}, \end{align}
where $-\dim \sl_{\rk \scf_\bullet}$ is interpreted as $0$ in the case $d = 3$.
Again, in \eqref{equation:tamagawa-evaluation-e} and \eqref{equation:tamagawa-evaluation-f}, the bundles have degrees as described after
\eqref{equation:e-and-f-single-sum}.
Therefore, \eqref{equation:e-and-f-sums-and-products} simplifies to
\begin{align}
		\label{equation:simplified-hurwitz-class}
		\frac{1}{\bl - 1} \bl^{\systemdim d g - \dim \sl_{\rk \sce} -
			\dim \sl_{\scf_\bullet}} \prod_{x \in \bp^1_k} \left( \sum_{R \in \mathcal
R} \bl^{-r(R)} \right) \left( 1-\bl^{-1} \right).  \end{align} 

Hence, using \autoref{lemma:dimension-aut-comparison},
\eqref{equation:unsimplified-hurwitz-class} simplifies to
	\begin{align} \label{equation:unsimplified-hurwitz-class} \frac{1}{\bl -
		1} \bl^{\dim \hur d g k + 1} \prod_{x \in \bp^1_k} \left(
		\sum_{R \in \mathcal R} \bl^{-r(R)} \right) \left( 1-\bl^{-1}
		\right).  \end{align} which equals the right hand side of
	\eqref{equation:rhur-class}.  \end{proof}

	Specializing \autoref{theorem:hurwitz-class} to the simply branched case
	gives the following corollary.

\begin{corollary} \label{corollary:simply-branched-hurwitz-class} For $2 \leq d
	\leq 5$, and $k$ a field of characteristic not dividing $d!$, in the
	case $\mathcal R = \{(1^d), (2, 1^{d-2})\}$ corresponding to simply
	branched curves, we have \begin{align*} \{\rhur d g k{\mathcal R}\} \equiv
	\bl^{\dim{\hur d g k}}(1-\bl^{-2}).  \end{align*} in $\grcStacks$
	modulo codimension $\codimbound d g$
	if $d \neq 2$, and in $\grStacks$ when $d =2$.
\end{corollary} Note that in the case $d = 2$, this corollary is equivalent to
the statement of \autoref{theorem:hurwitz-class} and is really proven
in \autoref{section:2}.  \begin{proof} Simply plug in $\mathcal R = \{(1^d), (2,
	1^{d-2})\}$ into \autoref{theorem:hurwitz-class}.  Then, $\sum_{R \in
	\mathcal R} \bl^{-r(R)} = 1 + \bl^{-1}$ and so 
	\begin{align*} \prod_{x
		\in \bp^1_k} (1-\bl^{-1})\left( \sum_{R \in \mathcal R}
		\bl^{-r(R)} \right) 
		&= \prod_{x \in \bp^1_k} (1-\bl^{-2})  & \\ 
		&= \prod_{x \in \bp^1_k} (1-\bl^{-2}t)|_{t = 1} \\
		&= \prod_{x \in \bp^1_k} (1-t)|_{t = \bl^{-2}} &
		\text{by \cite[\S3.8, Property 4]{bilu:thesis}}\\
		&= \frac{1}{Z_{\bp^1_k}(\bl^{-2})}  & \text{by \cite[Ex.\
		6.1.12]{biluH:motivic-euler-products-in-motivic-statistics}}
		\\ &= \left( 1- \bl^{-1} \right)\left(
1-\bl^{-2} \right).  \end{align*} 
Therefore, modulo codimension $\codimbound d
g$
in $\grcSpaces$ when $d \neq
2$, (and in $\grStacks$ when $d = 2$,) \begin{align*} \{\rhur d g
	k{\mathcal R}\} &\equiv \frac{\bl^{\dim{\hur d g k}}}{(1-\bl^{-1})}
	\prod_{x \in \bp^1_k} (1-\bl^{-1})\left( \sum_{R \in \mathcal R}
	\bl^{-r(R)} \right) \\ &=  \frac{\bl^{\dim{\hur d g k}}}{1-\bl^{-1}}
	\left( 1- \bl^{-1} \right)\left( 1-\bl^{-2} \right) \\ &= \bl^{\dim{\hur
	d g k}} (1-\bl^{-2}).  \qedhere \end{align*} 
\end{proof}

	When we allow the ramification profile to be arbitrary in
	\autoref{theorem:hurwitz-class} we obtain the following corollary
	counting all degree $d$ $S_d$ Galois covers of $\bp^1$.
	In the cases $d = 4$ and $d = 5$, there does not seem to be any obvious
	simplification of the motivic Euler product.

\begin{corollary} \label{corollary:full-hurwitz-class} For $k$ a field of
	characteristic not dividing $d!$, \begin{align*} \hur d g k\equiv
		\begin{cases} \bl^{\dim{\hur 2 g k}}(1-\bl^{-2})	 &
			\text{ if } d = 2 \\ \bl^{\dim{\hur 3 g k}} (1+\bl^{-1})
			\left( 1 - \bl^{-3} \right) & \text{ if } d = 3 \\
			\frac{\bl^{\dim{\hur 4 g k }}}{(1-\bl^{-1})} \prod_{x
			\in \bp^1_k} \left(1+\bl^{-2} - \bl^{-3} - \bl^{-4}
			\right).  & \text{ if } d= 4 \\ \frac{\bl^{\dim{\hur 5 g
			k}}}{(1-\bl^{-1})} \prod_{x \in \bp^1_k}
			\left(1+\bl^{-2}- \bl^{-4} - \bl^{-5} \right).  & \text{
	if } d = 5 \\ \end{cases} \end{align*} in $\grcStacks$	modulo
	codimension $\codimbound d g$ if $d \neq 2$, and in $\grStacks$ when $d =2$.
\end{corollary} \begin{proof} The case $d = 2$ is already covered in
	\autoref{corollary:simply-branched-hurwitz-class}, since $\rhur 2 g
	k{\{(1^2), (2)\}} = \hur 2 g k$.  Taking 
	\begin{align*}
		\mathcal R = \{(1^4), (2,
	1^2), (3, 1), (2^2), (4)\}
	\end{align*}
the $d = 4$ case follows from plugging
	$\mathcal R$ into \autoref{theorem:hurwitz-class} and using $\rhur 4 g
	k{\mathcal R} = \hur 4 g k$.  Taking $\mathcal R = \{(1^5), (2, 1^3),
	(2^2, 1), (3, 1^2), (3, 2), (4, 1), (5)\}$ the $d = 5$ case follows from
	plugging $\mathcal R$ into \autoref{theorem:hurwitz-class} and using
	$\rhur 5 g k{\mathcal R} = \hur 5 g k$.  Finally, let us check the $d =
	3$ case.  Here, for $\mathcal R = \{(1^3), (2,1), (3)\}$, we have $\rhur
	3 g k{\mathcal R} = \hur 3 g k$.  So, by
	\autoref{theorem:hurwitz-class}, 
	using \cite[\S3.8, Property 4]{bilu:thesis} and 
	by \cite[Ex.\
		6.1.12]{biluH:motivic-euler-products-in-motivic-statistics}
		as in the proof of
		\autoref{corollary:simply-branched-hurwitz-class},
	\begin{align*} \{\rhur d g
		k{\mathcal R}\} &\equiv \frac{\bl^{\dim{\hur d g k}}}{1-\bl^{-1}}
		\prod_{x \in \bp^1_k} (1-\bl^{-1})\left( \sum_{R \in \mathcal R}
		\bl^{-r(R)} \right) \\ &= \frac{\bl^{\dim{\hur d g
	k}}}{1-\bl^{-1}} \prod_{x \in \bp^1_k} (1-\bl^{-1})\left( 1 + \bl^{-1} + \bl^{-2} \right) \\ &= \frac{\bl^{\dim{\hur d g
k}}}{1-\bl^{-1}} \prod_{x \in \bp^1_k} (1-\bl^{-3})\\ &= \frac{\bl^{\dim{\hur d
g k}}}{1-\bl^{-1}} \frac{1}{Z_{\bp^1_k}(\bl^{-3})}\\ &= \frac{\bl^{\dim{\hur d g
k}}}{1-\bl^{-1}} \left( 1-\bl^{-2} \right)\left( 1-\bl^{-3} \right)\\ &=
\bl^{\dim{\hur d g k}} (1+\bl^{-1}) \left( 1-\bl^{-3} \right), \end{align*} where
we work in $\grcStacks$ modulo codimension $\codimbound d g$.  \end{proof}

\section{Degree 2} \label{section:2}

Following the notation introduced in \cite{arsieV:stacks-cyclic-covers}, let
$\rboxed{\ba_{\sm}(1,n)} \subset \spec \left( \sym^\bullet H^0(\bp^1,
	\sco(n)^\vee) \right)$ 
denote the open
subscheme parameterizing those degree $n$ forms on $\bp^1$ whose associated
closed subscheme is reduced.

\begin{lemma} \label{lemma:hyperelliptic-isomorphism} For $k$ a field with $\chr
k \neq 2$, there is an isomorphism of stacks $\hur 2 g k \simeq
[\ba_{\sm}(1,2g+2)/\bg_m]$, for an appropriate action of $\bg_m$ on
$\ba_{\sm}(1,2g+2)$.  \end{lemma} \begin{remark} \label{remark:} This can be
	deduced from the proofs of \cite[Thm.\  4.1, Cor.\
	4.7]{arsieV:stacks-cyclic-covers}, though there the authors work with a
	further quotient by the $\pgl_2$ action on the base $\bp^1$.  The
	$\bg_m$ action on $\ba_{\sm}(1,n)$ in \autoref{lemma:hyperelliptic-isomorphism}
	is explicitly given by $\alpha \cdot f(x) =
\alpha^{-2} f(x)$, though we will not need this in what follows.  \end{remark}
\begin{proof} First, we verify that $\hur 2 g k$ is equivalent to the fibered category whose $S$-points
	parameterize pairs $(\scl, i: \scl^{\otimes 2} \ra \sco_{\bp^1_S}),$ for
	$\scl$ a degree $-g - 1$ invertible sheaf on $\bp^1_S$, and $i$ an
	injective homomorphism of sheaves. 	Indeed, to connect this to our
	given definition of $\hur 2 g k$, we follow \cite[Rem.\ 3.3 and Prop.\
	3.1]{arsieV:stacks-cyclic-covers}: given a cover $\rho: H \ra \bp^1_S$,
	we have a natural action of $\mu_2$ on $H$ over $\bp^1$. 
	This comes from the isomorphism $\mu_2 \simeq \bz/2\bz$ as we are
	assuming $\chr(k) \neq 2$.	
	From this action,
	we obtain an isomorphism $\rho_* \sco_H \simeq \sco_{\bp^1_S} \oplus
	\scl$, for $\scl$ the subsheaf on which $\mu_2$ acts by $(t,s) \mapsto
	t \cdot s$, i.e., $\scl$ is the weight $1$ eigenspace of $\mu_2$,
	and $\sco_{\bp^1_S}$ is the weight $0$ eigenspace.  The description of
	$\scl$ as the $1$ eigenspace for the $\mu_2$ action yields a map $i: \scl \otimes
	\scl \ra \sco$.  In the other direction, given $(\scl, i: \scl^{\otimes
	2} \ra \sco_{\bp^1_S})$, we can recover $H =
	\spec_{\sco_{\bp^1_S}}(\sco_{\bp^1_S} \oplus \scl).$
	The given maps respect automorphisms over $\bp^1$, as the only
	nontrivial automorphism in both cases is given by the hyperelliptic
	involution. Hence, they define an equivalence of algebraic stacks.

	Next, consider the cover $\widetilde{\hur 2 g k}$ of $\hur 2 g k$ given
	as the stackification of the fibered category whose
	$S$ points parameterize triples $(\scl, \phi: \scl \simeq \sco(-g-1), i:
	\scl^{\otimes 2} \ra \sco)$, with $i$ injective.  
	Note that $\widetilde{\hur 2 g k} \to \hur 2 g k$
	is indeed surjective because $\scl \simeq \sce^\vee$ is a degree $-g-1$
	line bundle on $\bp^1$ by \autoref{lemma:degree-of-e}.
	Observe that
	$\widetilde{\hur 2 g k}$ has a natural action of $\bg_m$ acting by
	automorphisms of $\scl$, so that $\hur d g k = [\widetilde{\hur d g
k}/\bg_m]$.  Said another way, quotienting by $\bg_m$ forgets the data of
	the isomorphism $\phi$.

It remains to identify $\widetilde{\hur 2 g k}$ with $\ba_{\sm}(1,2g+2)$.
Indeed, this was done in the course of the proof of \cite[Thm.\
4.1]{arsieV:stacks-cyclic-covers}.  Briefly, given an $S$-point $(\scl, \phi,
i)$, associate the map $i \circ (\phi^{-1})^{\otimes 2} :
\sco_{\bp^1_S}(-2g-2) \ra \sco_{\bp^1_S}$ corresponding to a section of
$H^0(\bp^1_S, \sco(2g+2))$.  Conversely, given a section $f \in H^0(\bp^1_S,
\sco(2g+2))$, associate the triple $(\sco(-g-1), \id: \sco(-g-1) \ra
	\sco(-g-1), f: \sco(-g-1)^{\otimes 2} \ra \sco)$.  \end{proof}

We are now ready to prove \autoref{theorem:hurwitz-class} in the case $d = 2$.
\subsection{Proof of $d=2$ case of \autoref{theorem:hurwitz-class}}
\label{subsection:d=2-proof}

Note that the only
	allowable collection of ramification profiles
	is $\mathcal R = \{ (2), (1,1)\}$.  Since
$[\hur 2 g k \simeq \ba_{\sm}(1,2g+2)/\bg_m]$ by
	\autoref{lemma:hyperelliptic-isomorphism}, and $\bg_m$ is special, we
	have $\{\hur 2 g k\}\{\bg_m\} = \ba_{\sm}(1,2g+2)$.  Since \begin{align*}
		\{\bg_m\} \frac{\bl^{\dim{\hur 2 g k}}}{1-\bl^{-1}} \left(
		\prod_{x \in \bp^1_k} \left( \sum_{R \in \mathcal R} \bl^{-r(R)}
	\right) \left( 1-\bl^{-1} \right) \right) &= \frac{\bl-1}{1-\bl^{-1}}
	\cdot \bl^{2g+2} \prod_{x \in \bp^1_k} \left( 1-\bl^{-2} \right) \\ &=
\bl^{2g+3} \frac{1}{Z_{\bp^1_k}(\bl^{-2})} \\ &= \bl^{2g+3} \left( 1-\bl^{-1}
\right)\left( 1-\bl^{-2} \right), \end{align*} 
	(by \cite[Ex.\ 6.1.12]{biluH:motivic-euler-products-in-motivic-statistics}
	and \cite[\S3.8, Property 4]{bilu:thesis}, as in the proof of
\autoref{corollary:simply-branched-hurwitz-class})
it suffices to verify
\begin{align*} \ba_{\sm}(1,2g+2)&= \bl^{2g+3} \left( 1-\bl^{-1} \right)\left(
	1-\bl^{-2} \right).  \end{align*} 
	Indeed, this follows from \cite[Lem.\
	5.9(a)]{Vakil2015}.  In a bit more detail, taking $a =2$ in \cite[Lem.\
	5.9(a)]{Vakil2015}, the expression $K_{<2}(t)$ there is the generating
	function for which the coefficient of $t^n$ is the class of $w_{1^n}$ in
	the notation of \cite[(5.1)]{Vakil2015}. Here, $w_{1^a}$ is the class of
	the space of degree $n$ reduced divisors on $\bp^1$. Therefore, $w_{1^n}
= \{[\ba_{\sm}(1, n)/\bg_m]\}$, and so we only need check $\{w_{1^n}\} = \bl^n
	- \bl^{n-2}$. But indeed, this is the coefficient of $t^n$ in the
	expansion of \begin{equation*} \frac{Z_{\bp^1}(t)}{Z_{\bp^1}(t^2)} =
		\frac{(1-t^2 \bl)(1-t^2)}{(1-t\bl)(1-t)} = (1-t^2
		\bl)(1+t)\left(\sum_{i=0}^\infty (t\bl)^i\right).\qedhere
\end{equation*} 

\begin{remark} \label{remark:} The construction
	above used to compute the class of $\hur 2 g k$ is admittedly fairly ad
	hoc in the context of this paper.  A similar construction, more in line
	with the themes of this paper could be obtained by realizing a given
	hyperelliptic curve $\rho: H \ra \bp^1$ as a subscheme of $\bp (
	\left(\rho_* \sco_H \right)^\vee)$.  One can verify that $\bp (\rho_*
	\sco_H^\vee)$ is fppf locally isomorphic to $\bp\left( \sco_{\bp^1}
	\oplus \sco_{\bp^1}(g+1) \right)$, and use this to deduce that $\hur 2 g
	k$ is the quotient of the smooth members of a certain linear series on
	$\bp\left( \sco_{\bp^1} \oplus \sco_{\bp^1}(g+1) \right)$ by the
	automorphisms of $\bp\left( \sco_{\bp^1_k} \oplus \sco_{\bp^1_k}(g+1)
	\right)$ preserving the projection to $\bp^1_k$, and then use this
	description to compute $\{\hur 2 g k\}$, obtaining a formula similar to
	that of
	\autoref{theorem:hurwitz-class}.
	However, such a proof would only calculate the class in $\grcStacks$
	modulo a certain codimension, as opposed to the proof we give here,
	which actually calculates the class in $\grStacks$.
\end{remark}

\appendix
\section{A proof of a Theorem of Ekedahl \\ By Aaron Landesman and Federico Scavia}

	\maketitle
	The main result of this appendix is a proof of the following Theorem of T. Ekedahl. 
	We retain the notation for the Grothendieck ring of stacks described in
	\autoref{subsection:notation}.
	\begin{theorem}[Ekedahl]\label{ekedahl}~\protect{\cite[Thm.\
		4.3]{ekedahlGeometricInvariantFinite2009}}
		Let $k$ be a field. Then, for all integers $n \geq 1$,
		$\{BS_n\}=1$ in $\grStacks$.
	\end{theorem}
	
	Unfortunately, Ekedahl passed away prior to publishing \cite{ekedahlGeometricInvariantFinite2009}, and so the article was never refereed. 
	There are a number of typos and errors appearing in the proof of \cite[Thm.\ 4.3]{ekedahlGeometricInvariantFinite2009}. The objective of this appendix is to point out the fixes necessary.
	
	\smallskip
	 Let $k$ be a field, let $G$ be a finite group, and let $V$ be a
	 $G$-representation of dimension $d\geq 0$ over $k$. If $H$ is a
	 subgroup of $G$, we denote by $V^H$ the subscheme of $V$ fixed by $H$,
	 and by $V_H$ the locally closed subscheme parameterizing the locus whose stabilizer is exactly $H$. 
	 If there is a point of $V$ whose stabilizer is exactly $H$, we call $H$ a
	 \emph{stabilizer subgroup} of $G$. 
	 The normalizer $N_G(H)$ of $H$ acts on $V^H$ and $V_H$, and $V_H$ is an open subscheme of $V^H$. By definition, a \emph{stabilizer flag} of length $n$ is a sequence 
	\[f = (\{e\}=:H_0 \subset H_1  \subset \cdots \subset H_n)\] 
	of subgroups of $G$ such that, for all $0\leq i\leq n-1$, $H_{i+1}$ is a stabilizer subgroup of the $G$-action on $V$. We say that $f$ is \emph{strict} if $H_i\subsetneq H_{i+1}$ for all $i$. We set $n_f:=n$, $H_f:=H_n$, $d_f:=\dim V^{H_f}$ and $N_G(f):=\cap_{0\leq i\leq n} N_G(H_i)$.
	
	\begin{remark}
	Our definition of stabilizer flag differs from the one used by Ekedahl
	\cite[p. 10]{ekedahlGeometricInvariantFinite2009}, as he required that $H_{i+1}$ be a stabilizer subgroup of the action of $\cap_{j\leq i}N_G(H_i)$ on $V^{H_i}$. In particular, in our definition it is not necessarily true that $H_f\subset N_G(f)$.
	\end{remark}
	
	The conjugation action of $G$ on itself induces a $G$-action on the
	collection of all stabilizer flags. We say that two stabilizer flags are
	\emph{conjugate} to each other if they belong to the same orbit under this action.

\begin{proposition}\label{formula}
 Let $K\subset G$ be the kernel of the $G$-action on $V$. We have:
 \begin{align}
	 \label{equation:inclusion-exclusion-stacky-class}
	\{BG\}\L^d=\{[V_K/G]\}-\sum_f (-1)^{n_f}\{BN_G(f)\}\L^{d_f},
 \end{align}
where $f$ runs over a set of representatives of conjugacy classes of strict stabilizer flags of length $n_f\geq 1$.
\end{proposition}

\autoref{formula} corrects \cite[Thm.\
3.4]{ekedahlGeometricInvariantFinite2009}. The formula there looks the same as
ours (up to signs), but it is wrong as it is claimed with a different definition
of stabilizer flag. The error there stems from the falsity of \cite[Lem.\
3.3(iv)]{ekedahlGeometricInvariantFinite2009},
as illustrated by \autoref{example:ekedahl-counterexample} below.

We note that the proof of \autoref{formula} follows similar lines to that of \cite[Thm.
3.4]{ekedahlGeometricInvariantFinite2009}.
In particular, it uses results from
\cite[Lem.\ 3.3(i), (ii), and (iii)]{ekedahlGeometricInvariantFinite2009}
even though
\cite[Lem.\ 3.3(iv)]{ekedahlGeometricInvariantFinite2009}
is incorrect. It may be helpful for the reader to consult these statements.

\begin{example}
	\label{example:ekedahl-counterexample}
	The result \cite[Lem.\
	3.3(iv)]{ekedahlGeometricInvariantFinite2009} claims that
	$V_H = (V^H)_H$,
	where $V^H$ is considered as an $N_G(H)$ representation. 
	However, when $G = S_3$ and $H$ is the subgroup generated by
	$(12)$, and $G$ acts as the 3-dimensional permutation representation,
	then $V_H = \{(a,a,b) : a \neq b\}$, while $N_G(H) = H$ and $V^H =
	\{(a,a,b)\}$. So here, when $V^H$ is considered as an $N_G(H) = H$
	representation, we have that $H$ acts trivially and $(V^H)_H = V^H \neq
	V_H$.  

	Since \cite[Lem.\ 3.3(iv)]{ekedahlGeometricInvariantFinite2009}
	is implicitly used in the proof of \cite[Thm.\
	3.4]{ekedahlGeometricInvariantFinite2009}, \cite[Thm.\
	3.4]{ekedahlGeometricInvariantFinite2009} is also incorrect.  To
	produce a counterexample to the statement of
	\cite[Thm.\
	3.4]{ekedahlGeometricInvariantFinite2009} (even after correcting the $+$
	sign appearing in the statement there to the $-$ sign of
\eqref{equation:inclusion-exclusion-stacky-class}), we can again take $G = S_3$.
	Then, the only strict stabilizer flags in the sense of
	\cite[p. 10]{ekedahlGeometricInvariantFinite2009} (which are defined in a slightly different way than 
		in this appendix)
	up to conjugacy are $\{e\},\{e\}
	\subset S_2, \{e\} \subset S_3.$ In this case, with our knowledge that
	$\{B S_3\} = 1$, the formula of \cite[Thm.\
	3.4]{ekedahlGeometricInvariantFinite2009} claims $\bl^3 = (\bl^3 -
	\bl^2) + (\bl^2) + (\bl)$. Of course, what is missing from this formula
	is that we should subtract off a term $\bl$ coming from the sequence of
	subgroups $\{e\} \subset S_2 \subset S_3$, which is a stabilizer flag in
	the sense of this appendix, but not in the sense of \cite[p.
	10]{ekedahlGeometricInvariantFinite2009}.
\end{example}

\begin{proof}[Proof of \autoref{formula}]
	Let $f$ be a strict stabilizer flag. 
	We have
	\begin{align*}
		[V^{H_f}/N_G(f)] - [V_{H_f}/N_G(f)] =
[\coprod_{\substack{H \subset G \\ g\in N_G(f)/(N_G(f)\cap
	N_G(H))}} V_{gHg^{-1}}/N_G(f)]
	\end{align*}
	where, on the right hand side, 
	where $H$ runs among a set of representatives of $N_G(f)$-conjugacy classes of
	subgroups of $G$ acting on $V$ and properly containing $H_f$. 

	For any fixed $H \subset G$, we have
	\begin{align*}
		[\coprod_{g\in N_G(f)/(N_G(f)\cap
	N_G(H))}V_{gHg^{-1}}/N_G(f)] = 
	[V_{H}/N_G(f)\cap N_G(H)].
	\end{align*}
	For any
	such $H$, construct a strict stabilizer flag $f'$ by appending $H$ at
	the end of $f$. Then \[N_G(f)\cap N_G(H)=N_G(f').\]

	We conclude
	\begin{equation}\label{flag}\{[V_{H_f}/N_G(f)]\}=\{BN_G(f)\}\L^{d_f}-\sum_{f'}\{[V_{H_{f'}}/N_G(f')]\},\end{equation} 
	where $f'$ runs over a set of representatives of conjugacy classes of strict stabilizer flags of length $n_f+1$ and starting with $f$.
	
	We now wish to prove by induction on $m\geq 1$ that 
	\begin{equation}\label{induct}
	\{BG\}\L^d=\{[V_K/G]\}-\sum_{0< n_f< m}(-1)^{n_f}\{BN_G(f)\}\L^{d_f}-(-1)^m\sum_{n_f=m}\{[V_{H_f}/N_G(f)]\},
	\end{equation}
	where $f$ runs among a set of representatives of conjugacy classes of strict stabilizer flags.
	When $m=1$, \eqref{induct} coincides with \eqref{flag} for
	$f=(K)$. Assume now that \eqref{induct} holds for some $m>1$. One
	obtains the formula for $m+1$, by starting from the formula for $m$ and
	applying \eqref{flag} to every flag $f$ of length $m$. 
	
	Since $G$ is finite, there are only finitely many strict stabilizer flags. The conclusion follows by choosing $m$ to be larger than the length of every strict stabilizer flag.
\end{proof}
	
	Having replaced \cite[Thm.\ 3.4]{ekedahlGeometricInvariantFinite2009}
	by \autoref{formula}, the proof of \autoref{ekedahl}	can be completed
	as in \cite{ekedahlGeometricInvariantFinite2009}. From now on, let
	$G=S_n$ be the group of permutations of $\Sigma:=\{1,2,\dots,n\}$, and
	let $V$ be the $n$-dimensional permutation representation of $S_n$.
	
	 A \emph{flag} is a pair $(S,R)$, where $S$ is a finite set, and $R$ is a sequence $R_1\subset R_2\subset\dots\subset R_n$ of equivalence relations $R_i\subset S\times S$ on $S$. An isomorphism of flags $(S',R')\to (S,R)$  is a bijection $S'\xrightarrow{\sim} S$ sending $R'_i$ to $R_i$ for all $i$. We denote by $N_R(S)$ the automorphism group of $(S,R)$. 

\begin{lemma}\label{wreath}
	Assume that $G=S_n$ and that $V$ is the standard $n$-dimensional
	representation of $S_n$. Let $f$ be a strict stabilizer flag, and denote by $H_i$ the stabilizer subgroups appearing in $f$. For every $i$, let $R_i$ be the equivalence relation determined by the orbit partition of the $H_i$-action on $\Sigma$, and let $R$ be the flag on $\Sigma$ given by the $R_i$. 
	
	(a) We have $N_{S_n}(f)=N_R(\Sigma)$.
		
	(b) If $N_{S_n}(f)=S_n$, then either $f=(\{e\})$ or $f=(\{e\}\subset S_n)$.
		
	(c) Assume that $\{BS_m\}=1$ for all $m<n$ and that $N_{S_n}(f)\neq S_n$. Then $\{BN_{S_n}(f)\}=1$. 
\end{lemma}

\begin{proof}
	(a) This follows from the fact that, for every $i$, a bijection $\sigma$ of $\Sigma$ respects $R_i$ if and only if it normalizes $H_i$.
	
	(b) If $N_{S_n}(f)=S_n$, then for every $i$, $R_i$ is respected by every bijection of $\Sigma$. It follows that either $R_i$ is the diagonal in $\Sigma\times \Sigma$ or $R_i=\Sigma\times \Sigma$. Now (b) follows from (a).
	
	(c) We may assume that $\{BN_{S_n}(f')\}=1$ for all flags such that
	$n_{f'}<n_f$. By \cite[Prop.\ 4.2]{ekedahlGeometricInvariantFinite2009},
	$N_{S_n}(f)$ is a direct product of wreath products $N'\wr S_r:=(N')^r\rtimes S_r$, where $N'$ is the normalizer of a flag of smaller length, and $S_r$ acts by permutation of the $r$ factors $N'$. 

	In what follows, we use the symbol $\symm$ for the stacky symmetric power as introduced in \cite[p.
	5]{ekedahlGeometricInvariantFinite2009}.
	We also use the symbol $\wr$ for wreath product. This was introduced and
	notated $\int$ in \cite[p. 5]{ekedahlGeometricInvariantFinite2009},
	but we use $\wr$ instead of $\int$ in order to keep our notation
	consistent with the rest of the paper. 

	Because, for $G$ and $H$ finite groups, $B(G \times H) \simeq BG \times
	BH$, it suffices to show $\{ B(N' \wr S_r)\} = 1$.
We have $B(N' \wr S_r)\simeq BN'\wr BS_r \simeq \symm^r(BN')$,
as explained in \cite[p. 5]{ekedahlGeometricInvariantFinite2009},
By inductive assumption, $\{B(N'\wr S_r)\}=\sigma_s^t(\{BN'\})=\sigma_s^t(1)=1$. For the symbol $\sigma_s^t$, see \cite[Prop.\ 2.5]{ekedahlGeometricInvariantFinite2009}. 
\end{proof}

\begin{proof}[Proof of \autoref{ekedahl}]
Let $V$ be the $n$-dimensional permutation representation of $S_n$ over $k$, and let
$U:=V_{\{e\}}\subset V$ be the free locus of the $S_n$-action. By
\autoref{formula}, 
\[\{BS_n\}\L^n=\{U/S_n\}-\sum_f(-1)^{n_f}\{BN_{S_n}(f)\}\L^{d_f},\]
where $f$ runs among conjugacy classes of strict stabilizer flags. By \autoref{wreath}(b), we may rewrite this as
\[\{BS_n\}(\L^n-\L)=\{U/S_n\}-\sum_f(-1)^{n_f}\{BN_{S_n}(f)\}\L^{d_f},\]	
where now $f$ runs among conjugacy classes of strict stabilizer flags such that $N_{S_n}(f)\neq S_n$.	By \autoref{wreath}(c), we have $\{BN_{S_n}(f)\}=1$ for all such $f$.

We claim that $\{U/S_n\}$ is a polynomial in $\L$ with integer coefficients. The
stacks
$V_H/N_{S_n}(H)$ are isomorphic to parts of a locally closed stratification of $V/S_n$. This is well known from general principles when $\on{char}k=0$ or when $\on{char}k>0$ does not divide $n$, but Ekedahl gave a proof in arbitrary characteristic in \cite[Prop.\ 1.1(ii)]{ekedahlGeometricInvariantFinite2009}.

To show $\{U/S_n\}$ is a polynomial in $\bl$, let $f$ be a strict stabilizer flag. Then, as in the proof of \autoref{formula}, we have
\[\{V_{H_f}/N_{S_n}(f)\}=\{V^{H_f}/N_{S_n}(f)\}-\sum_{f'}\{V_{H_{f'}}/N_{S_n}(f')\},\]
where $f'$ runs among conjugacy classes of strict stabilizer flags starting with $f$ and of length $n_f+1$.

Applying the previous formula iteratively, we obtain
\[\{V/S_n\}=\{U/S_n\}-\sum_f (-1)^{n_f}\{V^{H_f}/N_{S_n}(f)\},\]
where $f$ runs among conjugacy classes of strict stabilizer flags of positive length.
For every flag $f$, we claim that the quotient $W_f:=N_{S_n}(f)/(H_f\cap N_{S_n}(f))$ is a product of symmetric groups, and $V^{H_f}$ is a permutation representation of $W_f$. 
To see this, note that $N_{S_n}(f)$ can be identified with $N_{R_f}(\Sigma)$ via
\autoref{wreath} for a sequence of equivalence relations $R_f$ given as $R_1 \subset R_2
\subset \cdots \subset R_{n_f}$.
Under this identification,
$H_f$ is identified with the subgroup of permutations acting trivially on
the equivalence classes defined by $R_{n_f}$.
Therefore, the action of $W_f$ on $V^{H_f}$ is generated by permutations
switching two equivalence classes of
$R_{n_f}$ 
for which there exists an isomorphism of those two classes respecting $R$.
Therefore, $W_f$ 
is a product of symmetric groups acting by a permutation representation on
$V^{H_f}$.
Hence, by the fundamental theorem for symmetric polynomials,
$V^{H_f}/N_{S_n}(f)=V^{H_f}/W_f$ is an affine space over $k$. Since $V/S_n$ is also isomorphic to affine space, we deduce that $\{U/S_n\}$ is a polynomial in $\L$, as claimed.
We conclude that $\{BS_n\}$ can be written as a rational function in $\L$ with integer coefficients, and with denominator $\L^n-\L$. By \cite[Lem.\ 3.5]{ekedahlGeometricInvariantFinite2009}, this implies that $\{BS_n\}=1$.
\end{proof}

%
%

\bibliographystyle{alpha} \bibliography{MyLibrary,Pardini}

\end{document}